\documentclass{article}

\usepackage{a4}
\usepackage{graphicx}
\usepackage{subfigure}
\RequirePackage{amsthm,amsmath,amssymb,amsfonts,enumerate}

\usepackage{tikz}
\usepackage{makeidx}         
\usepackage{graphicx}        
\usepackage{multicol}        

\usepackage{amsmath}
\usepackage{amsfonts}
\usepackage{amsbsy}
\usepackage{graphics}
\usepackage{epsfig}
\usepackage{amssymb}
\usepackage{mathrsfs}
\usepackage{dsfont}

\usepackage{url}

\usepackage{color}
\usepackage[dvipsnames, table]{xcolor}
\usepackage{comment}

\def\Px{\mathsf{P}}
\def\mgbb{{\boldsymbol{\gamma}}}

\def\eqd{\stackrel{{d}}{=}}

\def\RR{\mathds{R}}
\def\QQ{\mathds{Q}}
\def\NN{\mathds{N}}
\def\SX{{\mathscr X}}
\def\Ib{\mathbf{I}}
\def\xb{\mathbf{x}}
\def\Xb{\mathbf{X}}
\def\Ex{\mathsf{E}}
\def\var{\mathsf{var}}
\def\simd{\stackrel{d}{\sim}}
\def\me{\epsilon}
\def\ml{\lambda}
\def\ms{\sigma}
\newcommand{\vsp} {\vspace{0.3cm}}
\def\PP{\mathbb{P}}
\def\SB{{\mathscr B}}
\def\SS{{\mathcal S}}
\def\0b{\mathbf{0}}
\newcommand{\bea}{\begin{eqnarray*}}
\newcommand{\eea}{\end{eqnarray*}}
\newcommand{\be}{\begin{eqnarray}}
\newcommand{\ee}{\end{eqnarray}}
\def\Rb{\mathbf{R}}
\def\ub{\mathbf{u}}
\def\CR{\mathsf{CR}}
\def\mg{\gamma}
\def\dd{\mbox{\rm d}}
\def\rad{\stackrel{\rm d}{\ra}}
\def\ra{\rightarrow}
\def\SN{{\mathcal N}}
\def\ma{\alpha}
\def\e1{\mathrm{e}}
\newcommand{\fin} {\mbox{}~\hfill{\lower-0.3ex\hbox{$\triangleleft$}}}
\def\Ind{\mathds{1}}
\def\SO{{\mathcal O}}
\def\SH{{\mathcal H}}
\def\SV{\mathscr{V}}
\def\bb{\mathbf{b}}

\newtheorem{proposition}{Proposition}[section]
\newtheorem{lemma}{Lemma}[section]
\newtheorem{corollary}{Corollary}[section]
\newtheorem{definition}{Definition}[section]
\newtheorem{remark}{Remark}[section]

\begin{document}


\title{Non-asymptotic quantisation of spherically symmetric distributions}



\author{Luc Pronzato\footnote{(corresponding author) Laboratoire I3S, CNRS-Universit\'e C\^ote d'Azur, Sophia Antipolis, France, {\tt pronzato@i3s.unice.fr}} \ and Anatoly Zhigljavsky\footnote{School of Mathematics, Cardiff University, UK, {\tt ZhigljavskyAA@cardiff.ac.uk}}}

\date{\today}

\maketitle

\begin{abstract}
Zador’s celebrated theorem is a cornerstone of optimal quantisation, establishing both the weak limit of the empirical distribution of an $n$-point optimal quantiser in $\RR^d$ and the decay rate of the associated $L_s$-mean quantisation error. However, for large dimensions $d$, observing this asymptotic behaviour demands an astronomically large sample size $n$, which grows super-exponentially with $d$. Through a detailed analysis of the quantisation problem for spherically symmetric distributions, we demonstrate that for moderate $n$ random quantisers uniformly distributed on a sphere of suitable radius $R$ achieve exceptional performance. The expected distortion, expressed as a triple integral, can be computed with arbitrary precision, and the optimal radius $R$ can be efficiently determined numerically. Leveraging results from extreme-value theory, we derive approximations for $R$, particularly in scenarios where $n$ scales with $d$. Depending on the growth rate of $n$, $R$ may either converge to zero or approach a limiting value $R_\infty$ that is independent of $s$.
\end{abstract}

{\bf keywords:} quantisation, distortion, spherically symmetric distribution, Zador's theorem, extreme-value theory





\section{Introduction}
\label{sec:intr}

For $\mu$ a measure on $\SX\subseteq\RR^d$ (or a $d$-dimensional Riemannian manifold),
with density $\varphi$ and finite moment of order $s'>s\geq 1$, the $(\mu,s)$-distortion on an $n$-point set $\Xb_n=\{\xb_1,\ldots,\xb_n\}\in\RR^{n\times d}$ is $D_{\mu,s}(\Xb_n)=\Ex_\mu\{ \min_{i=1,\ldots,n} \|U-\xb_i\|^s \}$, where $U\simd\mu$. Zador's celebrated theorem \cite{Zador82} states that the empirical distribution of an $n$-point optimal quantiser $\Xb_n^*$ minimising $D_{\mu,s}(\Xb_n)$ converges weakly to the probability measure with density proportional to $\varphi^{d/(d+s)}$ as $n\to\infty$, and the normalised optimal $(\mu,s)$-distortion $n^{1/d}\,D_{\mu,s}(\Xb_n^*)$ reaches asymptotically its minimum as $n\to\infty$. The normalisation by $n^{1/d}$ is essential, in the sense that these asymptotic considerations are not applicable if $n^{1/d}$  is not large enough. For large $d$, this requires astronomically large values of $n$ to approach the asymptotic behaviour.

The purpose of this paper is to show that for $\mu$ spherically symmetric, random quantisers with a suitable spherically symmetric distribution can achieve exceptional performance. When $n$ is large but $n^{1/d}$ is small, the best random quantisers we obtain are for designs $\Xb_n$ uniform on a sphere $\SS_{d-1}(R)$ with suitable radius $R$, and the variability of $D_{\mu,s}(\Xb_n)$ vanishes as $n$ increases. Using extreme-value theory, we derive approximations for the radius $R$ of $\SS_{d-1}$ and for the best expected $(\mu,s)$-distortion. In the particular case $s=2$, $R$ is proportional to $\Ex_\mu\{\|U\|\}$ with a constant depending only on $n$ and $d$, and the best expected $(\mu,2)$-distortion is $\Ex_\mu\{\|U\|^2\}-R^2$; see Corollary~\ref{Coro:astar-s2-Pa}. We then identify different asymptotic regimes where $d$ increases and $n$ grows with $d$. When $\mu$ possesses a norm-concentration property and concentrates on a sphere of radius $r$ as $d \to \infty$, the optimal radius $R$ of the random quantiser exhibits distinct behaviours depending on the growth rate of $n$ relative to $d$:
for a sub-exponential growth, $R \to 0$; for a super-exponential growth, $R \to r$; for an exponential growth $n = \ml^d(1 + o(1))$, $R \to r\sqrt{1 - 1/\ml^2}$; see Proposition~\ref{Prop:extreme-value-general-mu}.
Overall, the paper provides a framework for generating nearly optimal quantisers in the large $n$ and $d$ situation for spherically symmetric distributions.
Sequences of nested designs with low quantisation error can thus be efficiently generated, offering a compelling alternative to the classical greedy-packing algorithm for space-filling design (see, e.g., \cite{PZ2022}). While this paper focuses exclusively on spherically symmetric distributions, the conclusion highlights that extending this approach to the more common case of quantising the uniform measure in the $d$-dimensional hypercube yields promising numerical results.

\vsp
The paper is organised as follows. Section~\ref{S:distancecdt-etc} introduces the key concepts used throughout the paper. It also derives an explicit integral expression of the expected $(\mu,s)$-distortion of a random quantiser whose $n$ points are independently identically distributed (i.i.d.) according to a spherically symmetric distribution $\PP$. Section~\ref{S:ball-sphere-annulus} examines three emblematic cases: (\textit{i}) $\mu$ is uniform on the unit sphere, (\textit{ii}) $\mu$ is uniform in the unit ball, and (\textit{iii}) $\mu$ follows a multivariate spherically symmetric normal distribution. In the latter two scenarios, in addition to the case where $\PP$ is uniform on a sphere, we also analyse the cases where $\PP$ is uniform in a ball and where $\PP$ is a multivariate spherically symmetric normal distribution. Section~\ref{S:asymptotic} applies extreme-value theory to derive asymptotic results ($n \to \infty$) for random quantisers uniform on a sphere. The analysis is conducted in two scenarios: when $\mu$ is uniform on a sphere; when $\mu$ is spherically symmetric and satisfies a norm-concentration property, with the ball-uniform and normal distributions serving as key examples.
Section~\ref{S:conclusion} provides a brief conclusion to the paper.

\vsp
We denote by $\SB_d(r)$ the $d$-dimensional (closed) Euclidean ball with center $\0b_d$ and radius $r$ 
and by $\SS_{d-1}(r)$ the $d$-dimensional Euclidean sphere with center $\0b_d$ and radius $r$; $\|\cdot\|$ is the $\ell_2$-norm.

\section{Distance c.d.f., quantisation error and distortion}\label{S:distancecdt-etc}



\subsection{Definitions and basic properties}

For any $n$-point set $\Xb_n$ we denote by $d(\cdot,\Xb_n)$ the distance function defined by $\xb\in\RR^d \mapsto d(\xb,\Xb_n)=\min_{\xb_i\in\Xb_n}\|\xb-\xb_i\|$; $\nu$ denotes a probability measure on $\RR^d$.

\begin{definition}\label{D:distance-cdf}
For $\Xb_n$ a fixed $n$-point set in $\RR^d$, the \emph{distance c.d.f.} $F(\cdot\,;\Xb_n,\nu)$ is the c.d.f.\ of the random variable $d(U,\Xb_n)$ when $U$ is distributed with $\nu$:
\bea 
F(t;\Xb_n,\nu) = \nu\left\{U\in \bigcup_{i=1}^n \SB_d(\xb_i,t)\right\}  = \nu\{d(U,\Xb_n) \leq t\}\,.
\eea
When $\Xb_n=\Rb_n$ is a random $n$-point set generated with some probability measure $\Px$ on $\RR^{n\times d}$, we define the \emph{mean distance c.d.f.} $F_n(\cdot\,; \Px,\nu)$ of the random variable $d(U,\Rb_n)$ by $F_n(t; \Px,\nu) = \Ex_\Px\{F(t;\Rb_n,\nu) \}$.
\end{definition}

If $\nu$ is the delta measure $\delta_\ub$ concentrated at $\ub\in\RR^d$, the mean distance c.d.f.\ is simply $F_n(t; \Px,\delta_\ub)=\Px \left\{ d(\ub,\Rb_n)\leq t  \right\}$.

When $\SX$ is a compact subset of $\RR^d$, we denote by $\CR(\Xb_n)= \max_{\xb\in\SX} d(\xb,\Xb_n)$ the covering radius of $\Xb_n$. If $\nu$ is equivalent to the Lebesgue measure on $\SX$, $F(\cdot\,;\Xb_n,\nu)$ has a density on $[0,\CR(\Xb_n)]$, which we denote by $f(\cdot\,;\Xb_n,\nu)$, and the essential supremum of the random variable $d(U,\Xb_n)$ equals $\CR(\Xb_n)$.
When $\nu$ is the uniform measure on $\SX$, we simply denote the distance c.d.f.\ by $F(\cdot\,;\Xb_n)$: the random variable $d(U,\xb_n)$ is supported on $[0,\CR(\Xb_n)]$ and $F(\cdot\,;\Xb_n)$ contains all information about the filling of $\SX$ by $\Xb_n$. In particular, the $\mg$-quantiles $q_\mg(\Xb_n)$ of $F(\cdot\,;\Xb_n)$ provide useful space-filling characteristics. Assume that $\SX$ is connected and coincides with the closure of its interior. Then, the density $f(\cdot\,;\Xb_n)$ is strictly positive on $(0,\CR(\Xb_n))$ and $q_\mg(\Xb_n)$ is defined by
\bea 
F[q_\mg(\Xb_n);\Xb_n]=\mg \ \mbox{ for any }\mg\in(0,1)\,,
\eea
with $q_{0}(\Xb_n)=0$ and $q_{1}(\Xb_n)=\CR(\Xb_n)$.

\begin{definition}\label{D:quantisation-error}
The \emph{$L_s$-mean quantisation error} of a general probability measure $\mu$ on $\RR^d$ with finite $s$-th moment $\Ex\{\|X\|^s\}$ for a given $n$-point set $\Xb_n$ is the $L_s(\mu)$-norm of the distance function $d(\cdot,\Xb_n)$:
\bea 
E_{\mu,s}(\Xb_n) = \|d(\cdot,\Xb_n)\|_{L_s} = \left[\int_\SX d^s(\xb,\Xb_n)\,\mu(\dd\xb)\right]^{1/s}\,, \quad s>0\,.\;\;\;
\eea
The quantity $E_{\mu,s}^s(\Xb_n)$ is called the $(\mu,s)$-distortion related to $\Xb_n$ {\rm \cite{Pages97}}. When $\Xb_n=\Rb_n$ is random, we define the expected $(\mu,s)$-distortion as $D_{\mu,s}(\Px)=\Ex_\Px\{E_{\mu,s}^s(\Rb_n)\}$.
\end{definition}

For any $s>0$, the $(\mu,s)$-distortion satisfies (see, e.g., \cite[p.~150]{Feller71-2}):
\be
E_{\mu,s}^s(\Xb_n) = \Ex_\mu\{d^s(U,\Xb_n)\} = s\, \int_0^\infty t^{s-1}\,[1-F(t;\Xb_n,\mu)] \,\dd t \,. \label{moment-ds-cdf}
\ee
The result is easily obtained when $\mu$ is such that $F(\cdot\,;\Xb_n,\mu)$ has a density $f(\cdot\,;\Xb_n,\mu)$. Indeed, by swapping the order of integration we get
\bea
\Ex_\mu\{d(U,\Xb_n)\} &=& \int_0^\infty \tau\,f(\tau;\Xb_n,\mu)\,\dd \tau = \int_0^\infty \left(\int_0^\tau \dd t\right) \,f(\tau;\Xb_n,\mu)\,\dd \tau \\
&& \hspace{-1cm} = \, \int_0^\infty \left(\int_t^\infty  f(\tau;\Xb_n,\mu)\,\dd \tau\right) \,\dd t = \int_0^\infty [1-F(t;\Xb_n,\mu)] \,\dd t \,.
\eea
Now, for any $s>0$, $G^{(s)}(\cdot\,;\Xb_n,\mu)$, defined by $G^{(s)}(t\,;\Xb_n,\mu)=F(t^{1/s};\Xb_n,\mu)$ for any $t\geq 0$, is the c.d.f.\ of $d^s(U,\Xb_n)$, and we obtain
\bea
\Ex_\mu\{d^s(U,\Xb_n)\} = \int_0^\infty [1-F(t^{1/s};\Xb_n,\mu)] \,\dd t = s\, \int_0^\infty t^{s-1}\,[1-F(t;\Xb_n,\mu)] \,\dd t \,. 
\eea


\begin{remark}\label{R:distortion-evaluation}
For any probability measure $\mu$ such that $\Ex\{\|X\|^{2s}\}<\infty$, the $(\mu,s)$-distortion of an arbirary $n$-point set $\Xb_n$ can be approximated by its Monte-Carlo estimator $E_{\mu_N,s}^s(\Xb_n)=(1/N)\sum_{i=1}^N d^s(U_i,\Xb_n)$ where the $U_i$ are i.i.d.\ with $\mu$, and $E_{\mu_N,s}^s(\Xb_n)$ satisfies a classical central limit theorem:
\bea
\sqrt{N} \left[E_{\mu_N,s}^s(\Xb_n)-E_{\mu,s}^s(\Xb_n)\right] \rad \SN(0,V(\Xb_n))\,, \quad N\to\infty\,,
\eea
where $V(\Xb_n)=E_{\mu,2s}^{2s}(\Xb_n)-E_{\mu,s}^{2s}(\Xb_n)$. When $\mu$ is supported on $\SX$ compact, $\CR(\Xb_n)<\infty$ and a direct application of Hoeffding's inequality gives:
\bea
\mbox{for all } \ma\in(0,1)\,, \ \mu\left\{|E_{\mu_N,s}^s(\Xb_n)-E_{\mu,s}^s(\Xb_n)|>\ma\, \CR(\Xb_n) \right\} < 2\,\e1^{-2\,N\ma^2} \,,
\eea
showing the exponentially fast concentration of $E_{\mu_N,s}^s(\Xb_n)$ to its mean $E_{\mu,s}^s(\Xb_n)$ as $N$ increases. When $\mu$ has unbounded support, depending on its tail properties, other concentration inequalities can also be derived.
\fin
\end{remark}

\subsection{Quantisation error of random quantisers}

Let $\Rb_n=\{\xb_1,\ldots,\xb_n\}$ denote a random $n$-point set generated with a probability measure $\Px$ on $\RR^{n\times d}$. Our objective is to choose $\Px$ so that the expected $(\mu,s)$-distortion, $D_{\mu,s}(\Px)$ of Definition~\ref{D:quantisation-error}, is minimum.

From \eqref{moment-ds-cdf}, we can write
\be\label{Ls-mean-meandcdf}
D_{\mu,s}(\Px) =  s \int_{t\geq 0} t^{s-1}[1-F_n(t;\Px,\mu)]\,\dd t
\ee
where $F_n(\cdot\,; \Px,\mu) = \Ex_\Px\{F(t;\Rb_n,\mu) \}$ is the mean distance c.d.f.\ of Definition~\ref{D:distance-cdf}; that is, the c.d.f.\ of the random variable $d(U,\Rb_n)$ where both $U$ and $\Rb_n$ are random. The expected $(\mu,s)$-distortion $D_{\mu,s}(\Px)$ is the $(\mu,s)$-distortion for the c.d.f.\ $F_n(\cdot\,; \Px,\mu)$. In this paper we focus on distortion and quantisation error, but we could have also considered the $\mg$-quantiles of $F_n(\cdot\,; \Px,\mu)$; see Remark~\ref{R:quantiles}.
Note that Jensen inequality yields the following upper bound on the expected quantisation error: $\Ex_\Px\{ E_{\mu,s}(\Rb_n) \}\leq [D_{\mu,s}(\Px)]^{1/s}$ for $s \geq 1$.

By swapping the order of expectations, we get
\be
F_n(t; \Px,\mu)  &=& \Ex_\Px \left\{\Ex_\mu \left\{\Ind_{\{\ub\in\RR^d:\,d(\ub,\Rb_n)\leq t\}}(U) \right\}\right\} \nonumber \\
&=&  \Ex_\mu \left\{\Ex_\Px\left\{\Ind_{\{\ub\in\RR^d:\,d(\ub,\Rb_n)\leq t\}}(U) \right\}\right\} \nonumber \\
&=& \label{eq:prod_partial_cdf} \Ex_\mu  \left\{ F_n(t; \Px,\delta_U) \right\}\,,
\ee
where $F_n(t; \Px,\delta_\ub) = \Px  \left\{ d(\ub,\Rb_n)\leq t  \right\}$.

In the following, we consider random $n$-point set $\Rb_n$ with i.i.d.\ points having the distribution $\mathbb{P}$, so that $\Px(\dd\xb_1,\ldots,\dd\xb_n)=\prod_{i=1}^n \PP(\dd\xb_i)$; we shall denote this distribution $\Px=\PP^{[n]}$.
In Sections~\ref{S:PP-distribution-d2-Xspherical} and \ref{S:F-distribution-d2-Xspherical} we show that $F_n(t; \PP^{[n]},\mu)$ can be expressed in the form of a double integral when $\PP$ is spherically symmetric. This will rely on the following property which is true for any $\PP$.

For any fixed $\ub  \in \RR^d$, any $t \geq 0$ and any $\Rb_n$ with i.i.d.\ points having the arbitrary distribution $\PP$, we have
\be
F_n(t; \PP^{[n]},\delta_\ub) &=& 1-\prod_{j=1}^n \PP\left\{ \ub  \notin \SB_d({\xb}_j,t)  \right\} \nonumber \\
&& \hspace{-3cm} = \, 1-\prod_{j=1}^n\left(1- \PP  \left\{ \ub  \in \SB_d({\xb}_j,t)  \right\} \right) = 1-\bigg(1-\PP \left\{ \|X - \ub\| \leq t \right\} \bigg)^n\,,
\label{eq:prod_partial}
\ee
where $X \simd \PP$.

The expression \eqref{eq:prod_partial} indicates that the calculations of the mean distance c.d.f.\ $F_n(t; \PP^{[n]},\mu)$
and then of the expected $(\mu,s)$-distortion require the calculation of the probability
\bea 
\PP \left\{ \|X - \ub\| \leq t \right\} = \PP \left\{\ub    \in \SB_d(X,t)  \right\}\,,
\eea
which depends on $\PP$, $\ub $ and $r$. It is a consequence of the exchange of order of integration in \eqref{eq:prod_partial_cdf} that we have to consider another distance c.d.f., where now $\ub$ is fixed and plays the role of a one-point quantiser, whereas $X$ is random. Of course, this symmetry vanishes when considering random $n$-point sets, and the presence of $n$ i.i.d.\ points is accounted for explicitly in \eqref{eq:prod_partial}.
It is noticeable that this is the only place where the size $n$ appears.

\subsection{$\PP \left\{ \|X- \ub\| \leq t \right\}$ when $\PP$ is spherically symmetric}\label{S:PP-distribution-d2-Xspherical}

By definition, the random vector $X=(X_1, \ldots,X_d)$ is spherically symmetric if it can be represented as $X \eqd  R \cdot Z^{(d)} $, where $R=\|X\|$, $Z^{(d)}=(Z_1,\ldots,Z_d)$ is uniformly distributed on the unit sphere $\SS_{d-1}(1)$, and $R$ and $Z^{(d)}$ are independent.

When $\PP$ is spherically symmetric the distribution of $\|X  - \ub \|$ for $X\simd \PP$ only depends on $\|\ub\|$ and $d$. This distribution is specified in Proposition~\ref{th:dist} below. In the following, $\beta_{a,b}$ denotes a random variable with the Beta-density
\bea
\varphi_{a,b}(t)=t^{a-1}(1-t)^{b-1}/B(a,b) \ \mbox{ for } 0\leq t \leq 1\,,
\eea
where $B(a,b)$ is the Beta-function and $a,b>0$. The c.d.f.\ of $\beta_{a,b}$ is $\Pr\{\beta_{a,b} \leq t\}=I_t(a,b)$, where $I_t(a,b)$ is the regularised incomplete Beta-function, for which we use the convention
\be\label{convention}
I_t(\cdot,\cdot)=   \left \{\begin{array}{ll}
    0 &\;\;  {\rm for}\;\; t\leq 0 \\
  1 & \;\; {\rm for}\;\; t \geq 1 \, .\\
  \end{array}
\right.
\ee
The following lemma (see, e.g., \cite[Sect.~2.2]{FangKN90}) will be used several times.

\begin{lemma}\label{L:projections-beta}
Let $X=(X_1,\ldots,X_d) \eqd  R \cdot Z^{(d)}$ be a spherically symmetric random vector in $\RR^d$. Then, for any $m\in\{1,\ldots,d-1\}$, we have
\bea
X^{(m)}=(X_1,\ldots,X_m) \eqd  R \cdot \sqrt{\beta_{m/2,(d-m)/2}}\cdot Z^{(m)} \,,
\eea
where $R=\|X\|$, $\beta_{m/2,(d-m)/2}$ and $Z^{(m)}$ are independent.  Moreover, the joint density of $V^{(m)}=X^{(m)}/R$ is
\bea
\varphi_m(v_1,\ldots,v_m)= \frac{\Gamma(d/2)}{\pi^{m/2}\,\Gamma((d-m)/2)} \, \left(1-\sum_{i=1}^m v_i^2 \right)^{(d-m)/2-1} \ \mbox{ for } \sum_{i=1}^m v_i^2\leq 1\,.
\eea
\end{lemma}

The next proposition is the key element for the derivation of exact and asymptotic results on the behaviours of the mean distance c.d.f.\ and the $(\mu,s$)-distortion.

\begin{proposition} \label{th:dist}
Assume that $d\geq 2$ and $X\simd\PP$ with $\PP$ spherically symmetric. Then, for any fixed $\ub \in \RR^d$ we have
\be
 \|X  - \ub \|^2 \eqd (\|\ub\|-R)^2 +4\,\|\ub\|\,R \, \beta_{\delta,\delta}\, , \label{eq:quantile2}
\ee
where $\delta=(d-1)/2$ and the random variables $R=\|X \|$ and $\beta_{\delta,\delta}$ are independent.

Moreover, when $\Rb_n\sim \Px=\PP_a^{[n]}$ with $\PP_a$ uniform on $\SS_{d-1}(a)$, $a\geq 0$, we have
\be\label{d2URn}
d^2(\ub,\Rb_n)\eqd (\|\ub\|-a)^2+4\,a\,\|\ub\|\,\zeta(n,d)\,,
\ee
where $\zeta(n,d)=\min_{i=1,\ldots,n} \zeta_i$ and the $\zeta_i\eqd \beta_{\delta,\delta}$ are i.i.d.
\end{proposition}

\begin{proof}
As $X$ is spherically symmetric, the distribution of $\|X  - \ub \| $ only depends on $\ub $ through $r=\|\ub\|$, and without any loss of generality we can assume that $\ub=(r,0, \ldots, 0)$. We thus have
\bea 
 \|X \!-\! \ub \|^2
\eqd  (r\!-\!X_1)^2\!+\! \sum_{j=2}^d X_j^2 =  r^2\! -\! 2\,r X_1\!  +\!R^2= (R\!-\!r)^2 \!+\!2\,rR(1\!-\!Z_1)\,, \nonumber
\eea
where we have denoted $Z_1=X_1/R$. From Lemma~\ref{L:projections-beta}, $R$ and $Z_1$ are independent and the expression of $\varphi_1(\cdot)$ gives  $(1-Z_1)/2 \eqd  \beta_{\delta,\delta}$, which yields \eqref{eq:quantile2}.
Since the $n$ points of $\Rb_n$ are i.i.d.\ with $\PP_a$, the representation \eqref{eq:quantile2} yields \eqref{d2URn}.
\end{proof}

From \eqref{eq:quantile2},  for all $ t\geq 0$ we have
\bea
\PP \left\{ \| X \!-\!\ub \|\! \leq \! { t } \right\}\!&=&  \! \Pr  \left\{(R-r)^2 +4\,Rr \, \beta_{\delta,\delta}\,  \leq  { t ^2} \right\}  \,,
\eea
with $r=\|\ub\| > 0$, and thus
\bea
 \PP  \left\{ \| X \!-\!\ub \|\! \leq \! { t } \right\}&=&\Pr  \left\{\, \beta_{\delta,\delta}\,  \leq  \frac{t^2 -(R-r)^2 }{4\,Rr}  \right\}\,. \;\;
\eea
By conditioning on the random variable $R= \|X \|$ and denoting $\Phi(\cdot )$ the c.d.f.\ of $R=\|X \|$, we obtain
\be\label{beta_approx_for_cube}
\PP  \left\{ \| X \!-\!\ub \|\! \leq \! { t } \right\}\!&=& \int_{\rho\geq 0}
\Pr  \left\{\, \beta_{\delta,\delta}\,  \leq  \frac{t^2 -(\rho-r)^2 }{4\,\rho r}  \right\} \dd\Phi(\rho )  \nonumber \\
&=& \int_{\rho\geq 0} I_\upsilon(\delta,\delta) \,\dd\Phi(\rho ) \,,
\ee
where
\be\label{v}
\upsilon=\upsilon(t,\rho,r)=[t^2 -(\rho-r)^2]/(4 \rho r)
\ee
and $I_\upsilon(\cdot,\cdot )$ is the regularised incomplete beta-function with the added convention~\eqref{convention}.
The integration in \eqref{beta_approx_for_cube} is over the support of the distribution of $R=\|X \|$.
If $r=\|\ub\| = 0$, then \eqref{beta_approx_for_cube} reduces to
\bea
    \PP  \left\{ \| X\!-\!\ub \|\! \leq \! { t } \right\} = \PP  \left\{\|X\| \leq t \right\} = \int_{0}^{{t}} \dd\Phi(\rho) \,.
\eea

When $\PP=\PP_a$, the uniform distribution on the sphere $\SS_{d-1}(a)$, the expression \eqref{beta_approx_for_cube} with the convention \eqref{convention} can be simplified as follows:
\be\label{PPaa}
\PP_a\left\{ \| X-\ub \| \leq t \right\} =
\left\{\begin{array}{ll}
0                      & \mbox{if } t\leq |r-a| \,,\\
I_\upsilon(\delta,\delta) & \mbox{if } |r-a| \leq t \leq a+r \,, \\
1                      & \mbox{if } a+r \leq t \,,
\end{array} \right.
\ee
where $r=\|\ub\|$ and $\upsilon=\upsilon(t,a,r)$ is given by \eqref{v}.

\subsection{$F_n(t; \PP^{[n]},\mu)$ and $D_{\mu,s}(\PP^{[n]})$ when $\PP$ is spherically symmetric}\label{S:F-distribution-d2-Xspherical}

From \eqref{eq:prod_partial_cdf} and \eqref{eq:prod_partial}, the mean distance c.d.f.\ $F_n(\cdot\,; \PP^{[n]},\mu)$ can be calculated explicitly as
\be \label{eq:prod_partial_cdf4}
F_n(t; \PP^{[n]},\mu) = 1- \Ex_\mu \left\{ \bigg(1-\PP \left\{ \| X-U \| \leq t \right\} \bigg)^n \right\}\,,
\ee
where $U\simd \mu$, $X \simd \PP$, and $X$ and $U$ are independent.
When $\PP$ is spherically symmetric, $\PP \left\{ \|X- \ub \| \leq t \right\} $ is given by \eqref{beta_approx_for_cube} which only depends on $\ub$ through $r=\|\ub\|$, and the expectation with respect to $U$ in \eqref{eq:prod_partial_cdf4} is reduced to an expectation with respect to the distribution of $\|U\|$. The calculation of $F_n(t; \PP^{[n]},\mu)$ thus amounts to a double integration, with respect to the distributions of $\|X\|$ and $\|U\|$.
From \eqref{Ls-mean-meandcdf}, the calculation of $D_{\mu,s}(\PP^{[n]})$ requires an additional integration with respect to $t$,
\be
D_{\mu,s}(\PP^{[n]}) &=&  s \int_{t\geq 0} t^{s-1} \Ex_\mu \left\{ \bigg(1-\PP \left\{ \| X-U \| \leq t \right\} \bigg)^n \right\} \,\dd t \,, \nonumber \\
&=& s \int_{t\geq 0} t^{s-1} \int_{r\geq 0} H(r,t;\Phi,n)\, \dd\Psi(r)\,\dd t \,, \label{Ls-mean-final}
\ee
where
\bea 
H(r,t;\Phi,n) =\bigg(1-\int_{\rho\geq 0} I_\upsilon(\delta,\delta) \,\dd\Phi(\rho ) \bigg)^n \,,
\eea
with $\upsilon=\upsilon(t,\rho,r)$ given by \eqref{v}, $\Psi(\cdot)$ the c.d.f.\ of $\|U\|$ for $U\simd\mu$ and $\Phi(\cdot)$ the c.d.f.\ of $\|X\|$ for $X\simd\PP$.

The integrals required to compute $D_{\mu,s}(\PP^{[n]})$ can be evaluated with arbitrary precision. When $\PP$ is appropriately parameterised, a high-performance random quantiser can be obtained by numerically minimising the distortion with respect to its defining parameters. For the remainder of this paper, we focus on quantisers defined by a distribution $\PP$ that depends on a single scalar parameter: the radius of a sphere (Section~\ref{S:sphere}), the radius of a ball (Section~\ref{S:ball}), or the variance of a multivariate normal distribution (Section~\ref{S:normal}). The optimisation is performed using a derivative-free line-search method, specifically the golden-section algorithm.

%

\begin{remark}
The expression \eqref{Ls-mean-final} gives the expected $(\mu,s)$-distortion for random quantisers. Therefore, according to the paradigm of the so-called ``probabilistic method" (see, e.g., \cite{AlonS2000}), for any $\PP$ there exists at least one non-random $n$-point set $\Xb_n$ with $(\mu,s)$-distortion $E_{\mu,s}^s(\Xb_n)\leq D_{\mu,s}(\PP^{[n]})$. See Section~\ref{S:sphere-2^d} for an illustration involving full factorial designs.
\fin
\end{remark}



\section{Exact, non-asymptotic, results}\label{S:ball-sphere-annulus}

Throughout this section, three emblematic cases of  spherically symmetric measures are investigated. Nevertheless, the results remain valid for any measure $\nu$, regardless of its symmetr, so long as the radial component $r=\|U\|$ follows the same distribution under $\nu$ as under $\mu$. For instance, in Section~\ref{S:sphere} ($\mu$ uniform on $\SS_{d-1}(1)$), $\nu$ can be a point mass $\delta_\ub$ for any unit vector $\ub$; in Sections~\ref{S:ball} ($\mu$ uniform in $\SB_d(1)$) and~\ref{S:normal} ($\mu$ normal), $\nu$ may be supported on a line segment or a line, respectively, provided the distribution of $r$ is preserved.

\subsection{$\mu$ is uniform on the unit sphere $\SS_{d-1}(1)$} \label{S:sphere}

Here $\mu$ is the spherical measure (uniform on $\SS_{d-1}(1)$), and the distribution of $\|U\|$ for $U\simd\mu$ is the delta measure at $r=1$. When $n=1$ and $\Xb_n=\{\0b_d\}$, the distance c.d.f.\ $F(t;\Xb_n,\mu)$ equals 0 for $t<1$ and 1 for $t\geq 1$. For any other $\Xb_n$ it has a density $f(\cdot\,;\Xb_n,\mu)$ on $[0,\CR(\Xb_n)]$.

We first consider random quantisers with distribution $\PP=\PP_a$ uniform on $\SS_{d-1}(a)$  and minimise $D_{\mu,s}(\PP^{[n]}_a)$ with respect to $a\in[0,1]$. The expression \eqref{PPaa} can be simplified as follows:
\bea 
\PP_a\left\{ \| X-\ub \| \leq t \right\} =
\left\{\begin{array}{ll}
0                      & \mbox{if } t\leq 1-a \,,\\
I_\upsilon(\delta,\delta) & \mbox{if } 1-a \leq t \leq 1+a \,, \\
1                      & \mbox{if } 1+a \leq t \,,
\end{array} \right.
\eea
where $\upsilon=[t^2 -(1-a)^2]/(4 a)$. Formula~\eqref{Ls-mean-final} then gives
\be\label{Ds-sphere}
D_{\mu,s}(\PP^{[n]}_a) =  (1-a)^s + s \int_{1-a}^{1+a} t^{s-1} \left[1-I_\upsilon(\delta,\delta)\right]^n\,\dd t \,.
\ee
This exact expression for the expected $(\mu,s)$-distortion depends on $d,n,s$ and $a$; the value $a^*=a^*(d,n,s)$ that minimises $D_{\mu,s}(\PP^{[n]}_a)$ with respect to $a$ can be obtained numerically with arbitrary precision for any given $d,n$ and $s$. In the special case $d=3$, $\delta=1$ and $I_v(1,1)=v$ for $v\in[0,1]$, so that $D_{\mu,s}(\PP^{[n]}_a)$ can be calculated explicitly for $s$ even. It is given by a polynomial in $a$ of degree $s$, with for example $D_{\mu,2}(\PP^{[n]}_a)=(1-a)^2+4\,a/(n+1)$ and $D_{\mu,4}(\PP^{[n]}_a)=(1-a)^4+8\,a[n(1-a)^2+2+2\,a^2]/[(n+1)(n+2)]$, which gives $a^*(3,n,2)=(n-1)/(n+1)$ and $a^*(3,n,4)$ as the root of a cubic equation, which satisfies $a^*(3,n,4)=1-4/n+16/n^2+\SO(1/n^3)$ as $n\to\infty$.

The left column of Figure~\ref{F:sphere} presents $a^*(d,n,s)$ as a function of $d$ for different $n$ and $s=1$ (top) and $s=10$ (bottom). For fixed $s$ and $d$, $a^*$ increases with $n$, and asymptotically, when $n$ tends to infinity, $a^*$ tends to 1 (see, e.g., \cite[Chap.~9]{GrafL2000}). Intuitively the choice $a=1$ seems natural for all $n$ and $s$. However, Figure~\ref{F:sphere} indicates that $a=1$ can be a poor choice, especially in high dimension. For any fixed $d$ and $n$, $a^*(d,n,s) \to 1$ when $s\to 0$ and is a decreasing function of $s$, with $a^*(d,n,\infty)=0$ if $n$ is small enough. Indeed, for $n \leq d$, $\CR(\Xb_n)\geq 1$ for any $\Xb_n$ and the inequality is strict if $\0b_n\not\in\Xb_n$. This follows from the fact that any $d$ points in $\RR^d$ belong to a $d$-dimensional hyperplane $\SH_d$; one of the two parts of $\SS_{d-1}(1)$ cut by $\SH_d$ necessarily contains a hemisphere whose pole is at distance at least one from each point. This implies that $a^*(d,n,s)=0$ for large enough $s$ when $n\leq d$. Also, for fixed $n$ and $s$, $a^*(d,n,s)$ decreases with $d$.

The central column of Figure~\ref{F:sphere} shows that, as expected, $D_{\mu,s}^{1/s}(\PP^{[n]}_{a^*})$ decreases with $n$ and increases with $d$ and $s$. The right column presents the efficiency of the naive quantiser with $n$ points i.i.d.\ on $\SS_{d-1}(1)$ relative to optimised random quantisers with $n$ points on $\SS_{d-1}(a^*)$. It shows that the benefit of using an appropriate $a$ is significant when $d$ is large or $n$ is small.

\begin{figure}[ht!]
\begin{center}
\includegraphics[width=.32\linewidth]{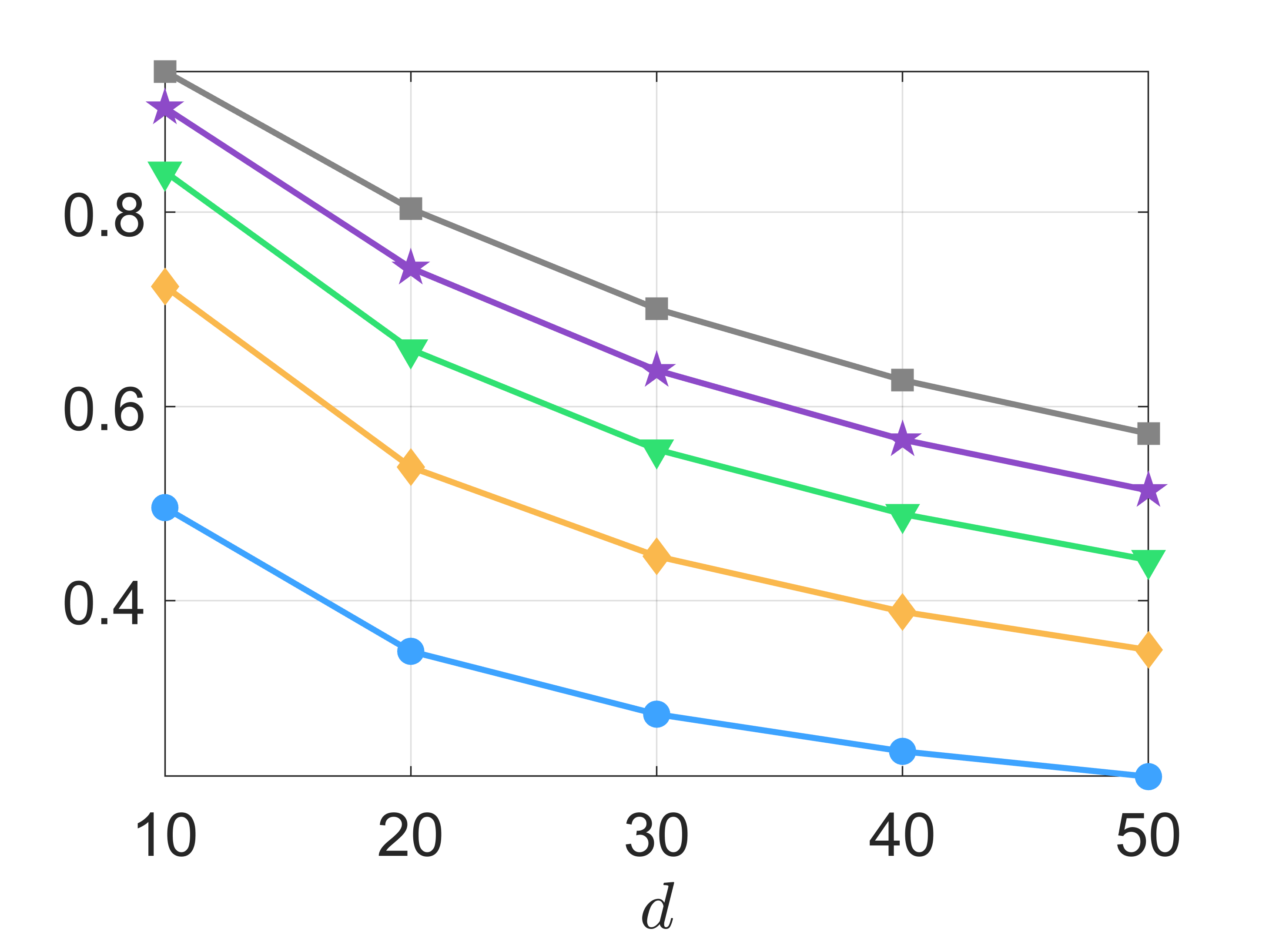}
\includegraphics[width=.32\linewidth]{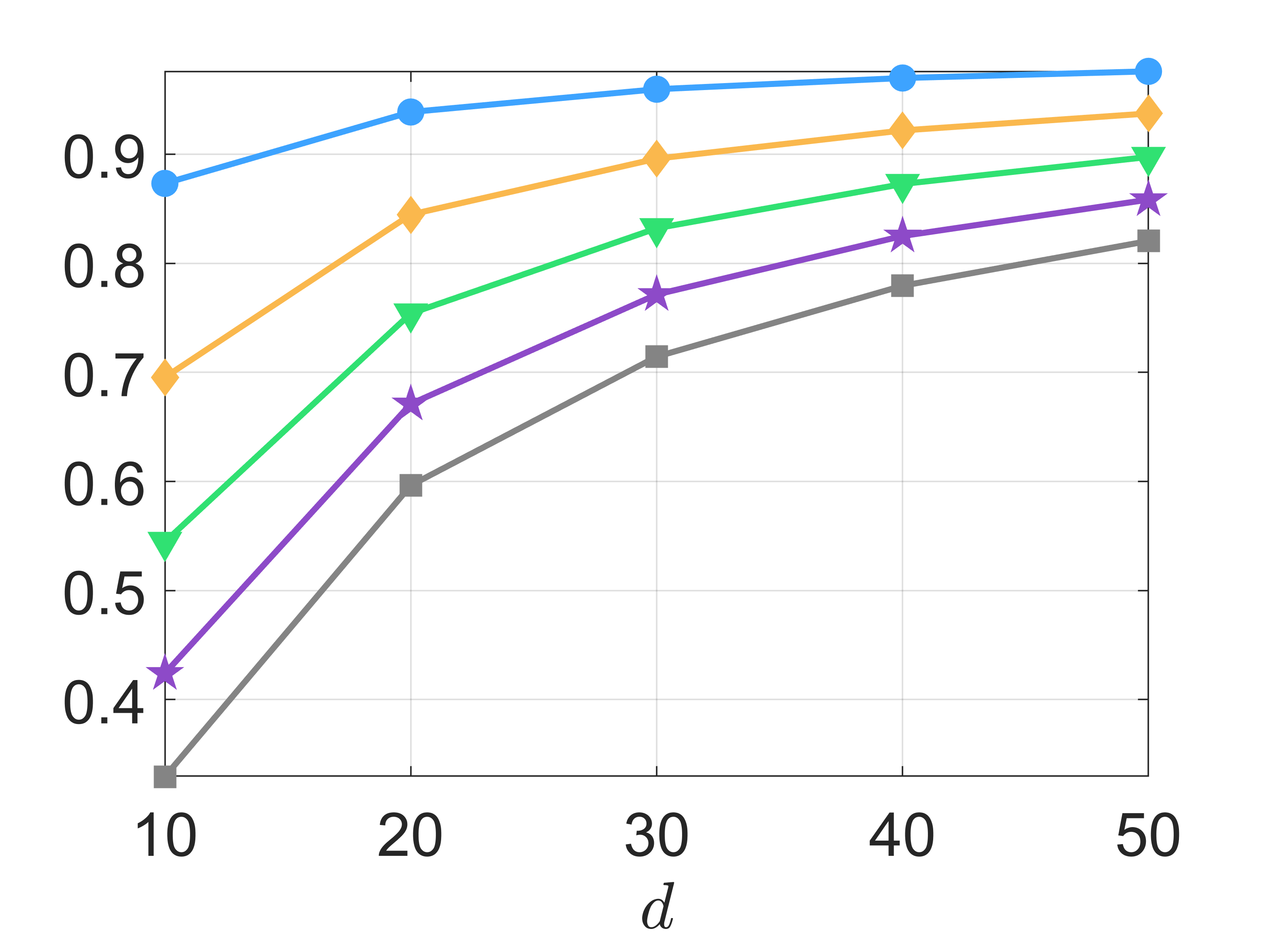}
\includegraphics[width=.32\linewidth]{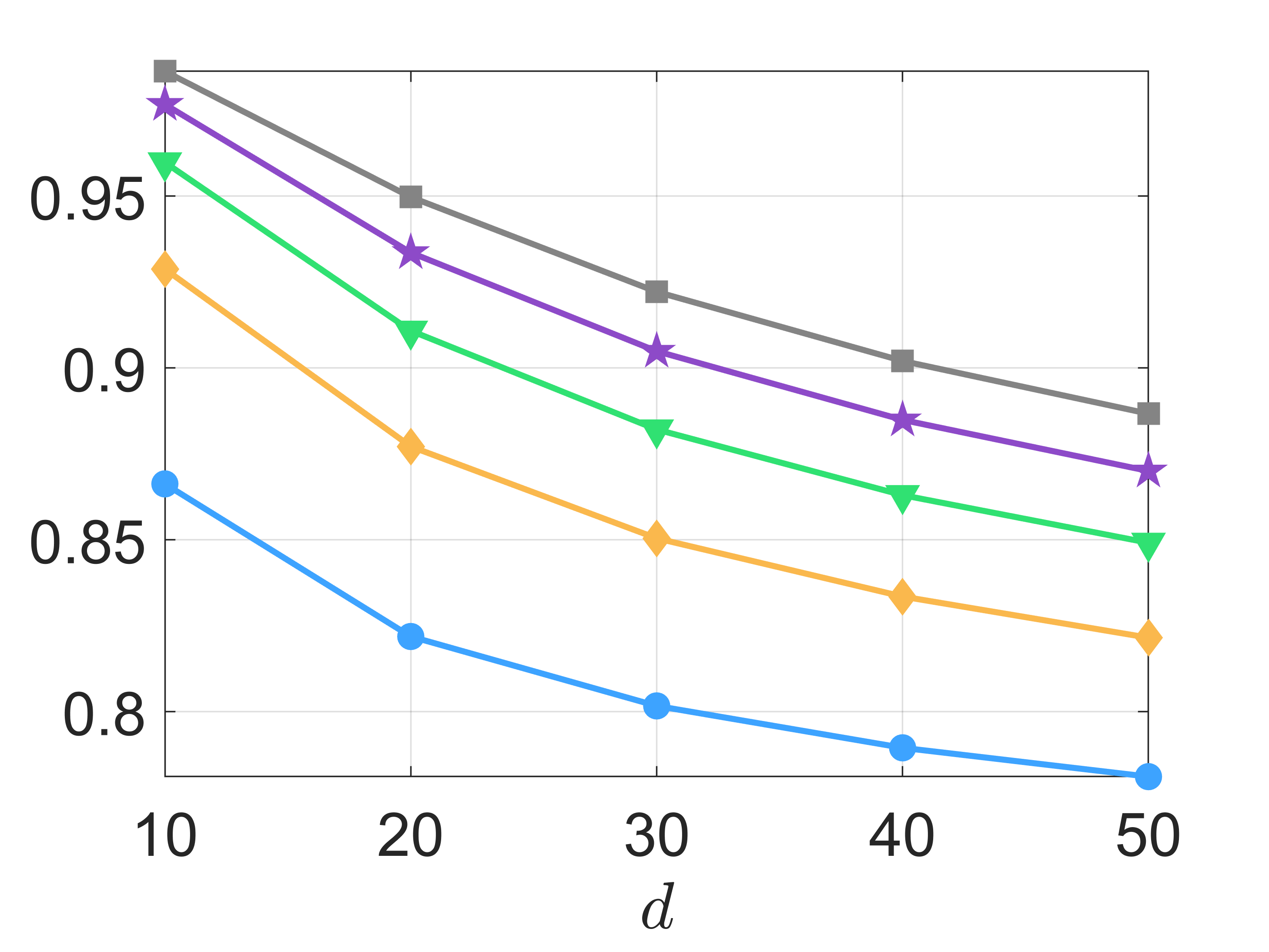} \\
\includegraphics[width=.32\linewidth]{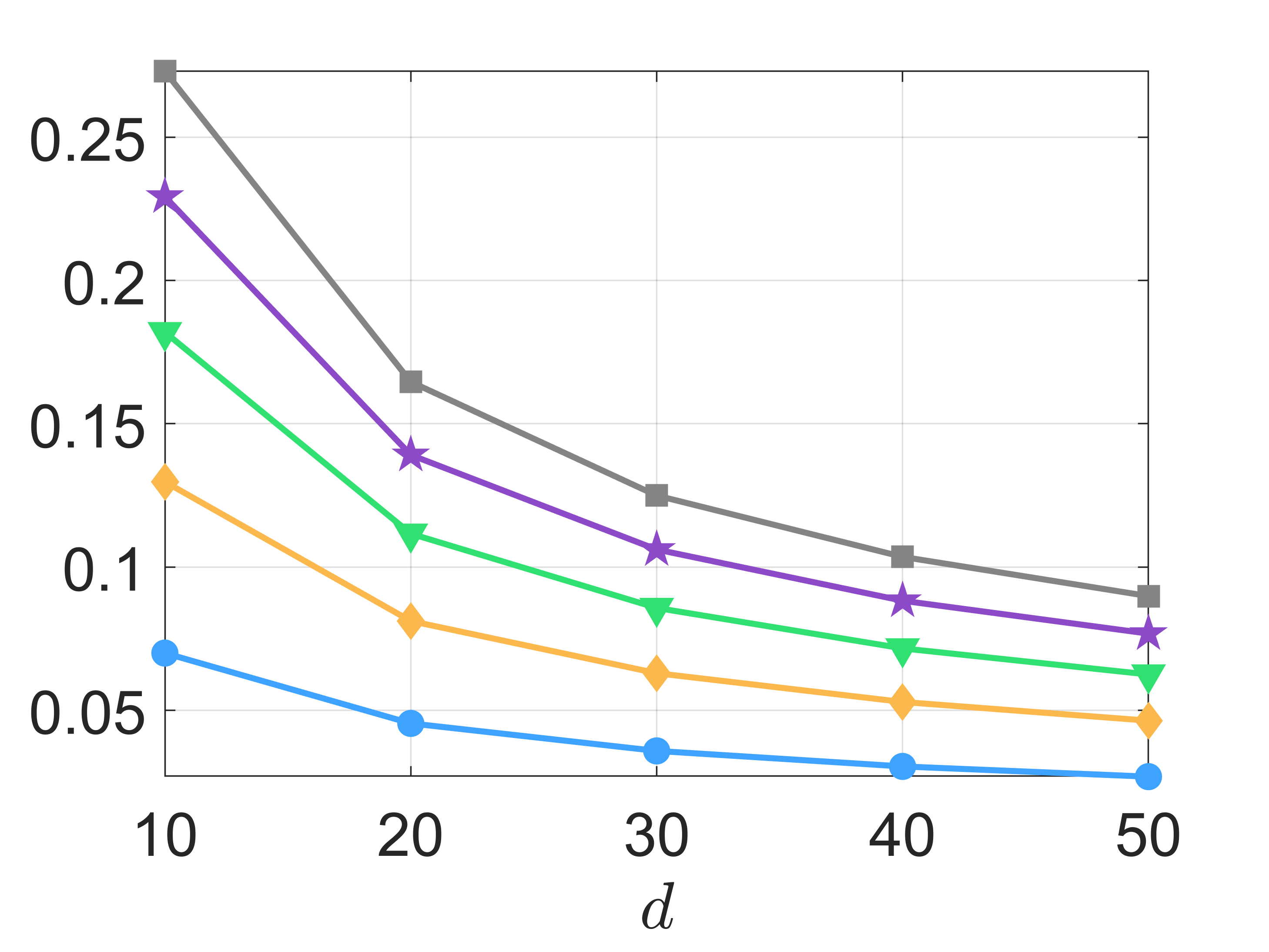}
\includegraphics[width=.32\linewidth]{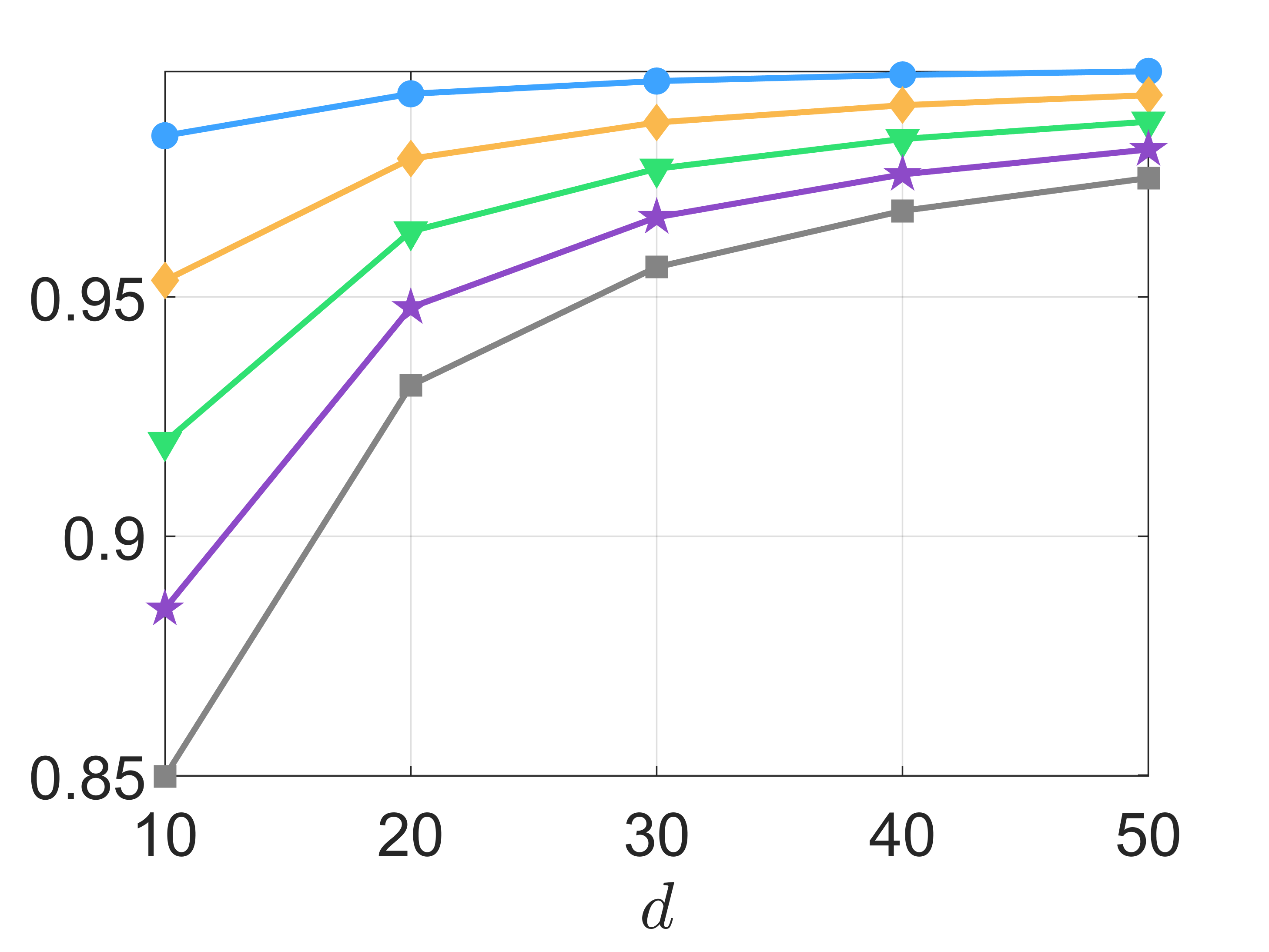}
\includegraphics[width=.32\linewidth]{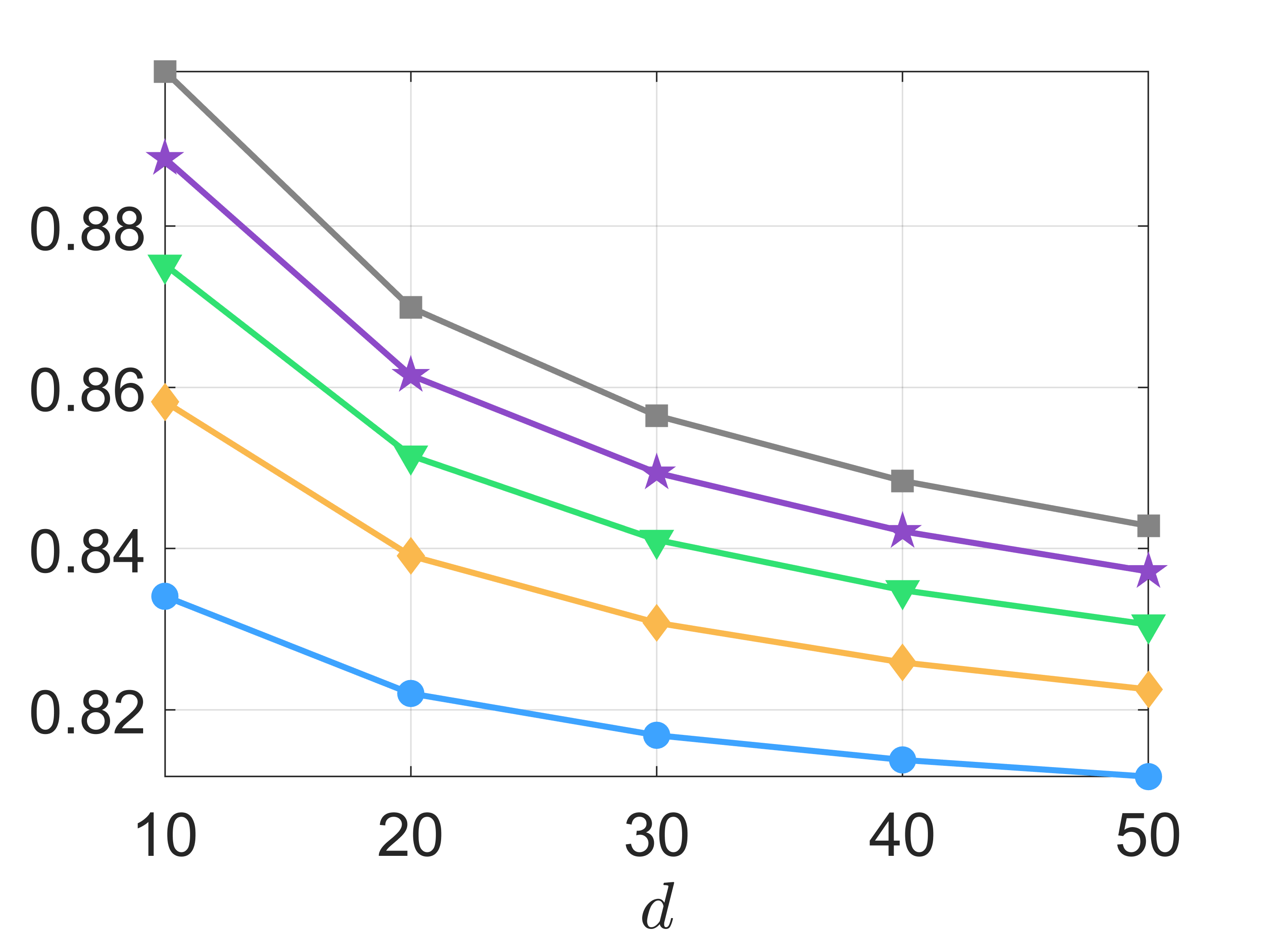}
\end{center}
\caption{\small Sphere. Value $a^*$ minimising $D_{\mu,s}(\PP^{[n]}_a)$ (left column), $D_{\mu,s}^{1/s}(\PP^{[n]}_{a^*})$ (central column) and ratio $D_{\mu,s}^{1/s}(\PP^{[n]}_{a^*})/D_{\mu,s}^{1/s}(\PP^{[n]}_1)$ (right column) as functions of $d$ for $s=1$ (top row) and $s=10$ (bottom row) and different values of $n$: $n=10$ ({\color{Cerulean} $\bullet$}), $n=10^2$ ({\color{Dandelion} $\blacklozenge$}), $n=10^3$ ({\color{green} $\blacktriangledown$}), $n=10^4$ ({\color{Orchid} $\bigstar$}), and $n=10^5$ ({\color{gray} {\tiny $\blacksquare$}}).}
\label{F:sphere}
\end{figure}

\vsp
The next figure (Figure~\ref{F:sphere2}, where $s=10$, $d=10$) shows that adding a delta measure at $\0b_d$ to $\PP_1$ may also yield a decrease of the expected distortion for small~$n$. Denote $\QQ_{\ma,a}=\ma\,\PP_a+(1-\ma)\,\delta_{\0b_d}$. Similarly to \eqref{Ds-sphere}, we obtain
\bea
D_{\mu,s}(\QQ^{[n]}_{\ma,a})  &=&  (1-a)^s + s \int_{1-a}^1 t^{s-1} \left[1-\ma\, I_\upsilon(\delta,\delta)\right]^n\,\dd t \\
&& + s\,\ma^n \int_1^{1+a} t^{s-1} \left[1- I_\upsilon(\delta,\delta)\right]^n\,\dd t \,,
\eea
with $\upsilon=[t^2 -(1-a)^2]/(4 a)$.
The left panel presents the optimal $\ma^*$ minimising $D_{\mu,s}(\QQ^{[n]}_{\ma,1})$ as a function of $n$: unsurprisingly, $\ma^*$ quickly approaches 1 as $n$ increases. One may observe that, for all $n$, $(1-\ma^*)>1/n$ (it corresponds to the probability that a realisation of a random set $\Rb_n$ contains a point at $\0b_d$).
Next, for a given $n$ ($n=20$), we use the corresponding optimal $\ma^*$ minimising $D_{\mu,s}(\QQ^{[n]}_{\ma,1})$ ($\ma^*\simeq 0.846$) and plot $D_{\mu,s}^{1/s}(\QQ^{[n]}_{\ma^*,a})$ and $D_{\mu,s}^{1/s}(\QQ^{[n]}_{1,a})=D_{\mu,s}^{1/s}(\PP^{[n]}_a)$ as functions of $a$ (central panel): the effect of introducing a delta measure at $\0b_d$ is significant for $a$ close to one, but choosing a suitable radius $a$ has bigger impact. Finally, we consider the efficiencies of random quantisers with $\QQ_{\ma^*,1}$ with respect to quantisers distributed with $\PP_1$ and
$\PP_{a^*}$ (with $\ma^*$ and $a^*$ depending on $n$) as functions of $n$ (right panel): when $a$ is set to 1, the choice of a suitable $\ma$ has some positive effect for small $n$, but the impact of choosing a suitable $a$ is stronger for all $n$.

\begin{figure}[ht!]
\begin{center}
\includegraphics[width=.32\linewidth]{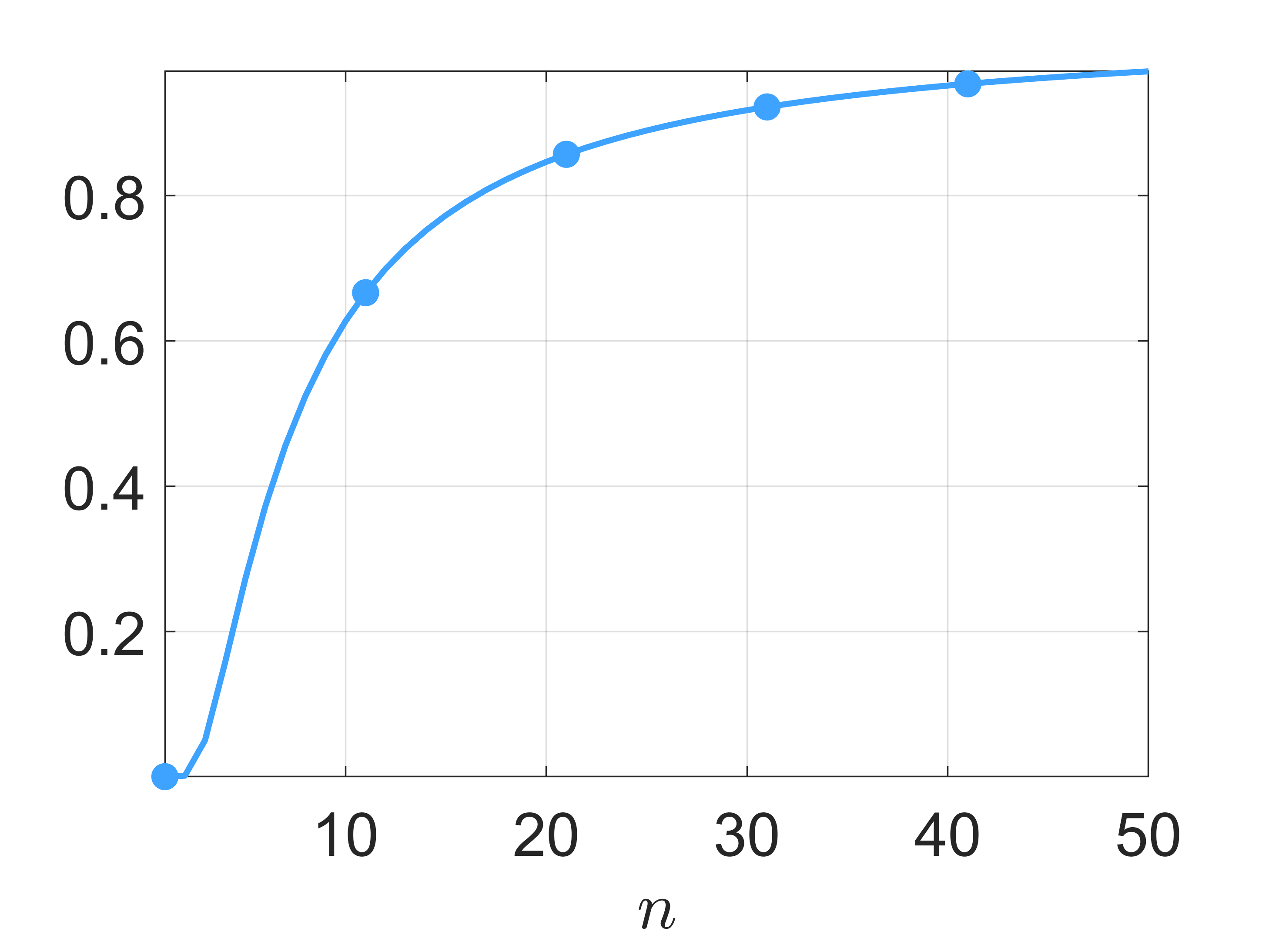}
\includegraphics[width=.32\linewidth]{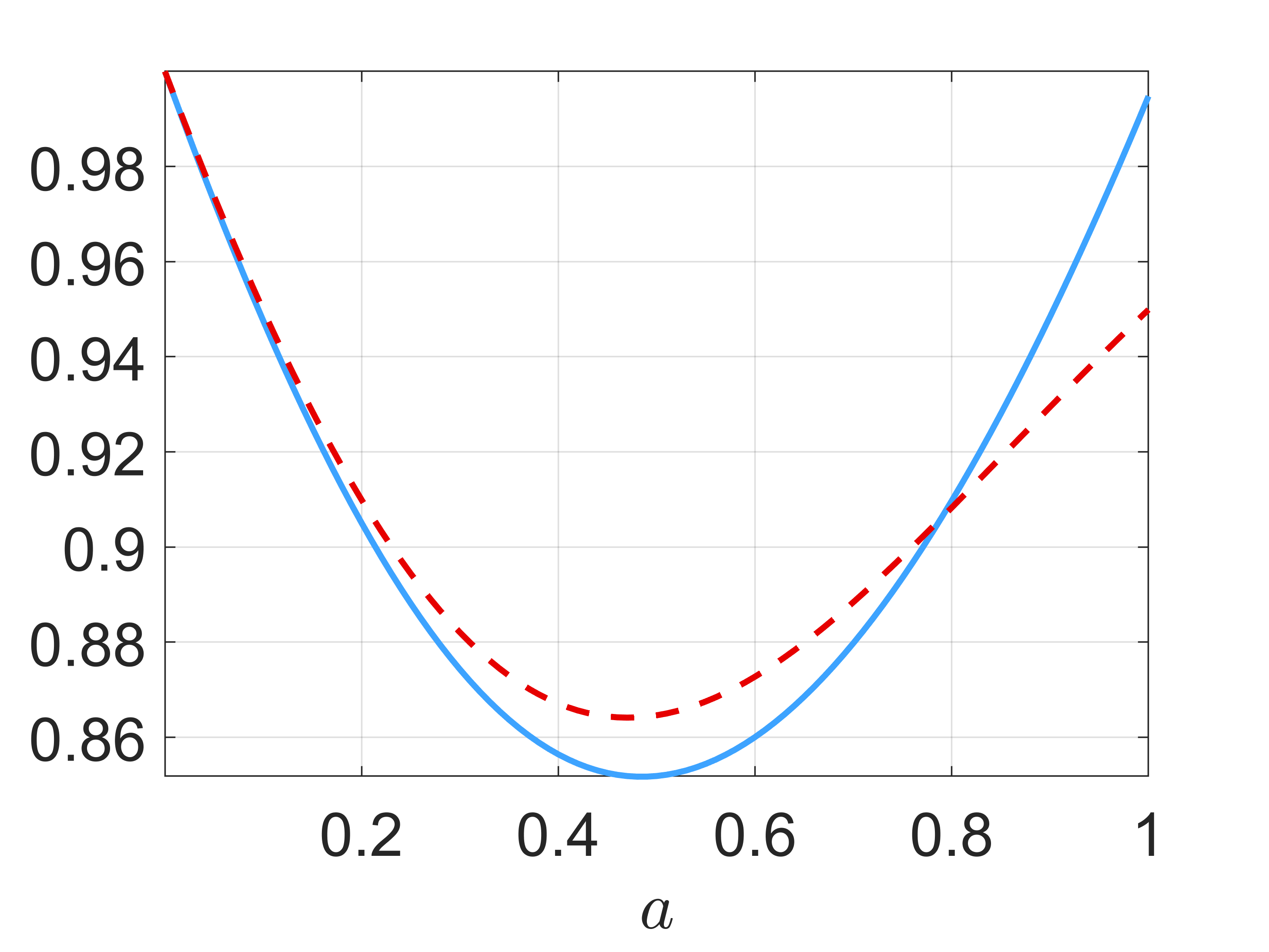}
\includegraphics[width=.32\linewidth]{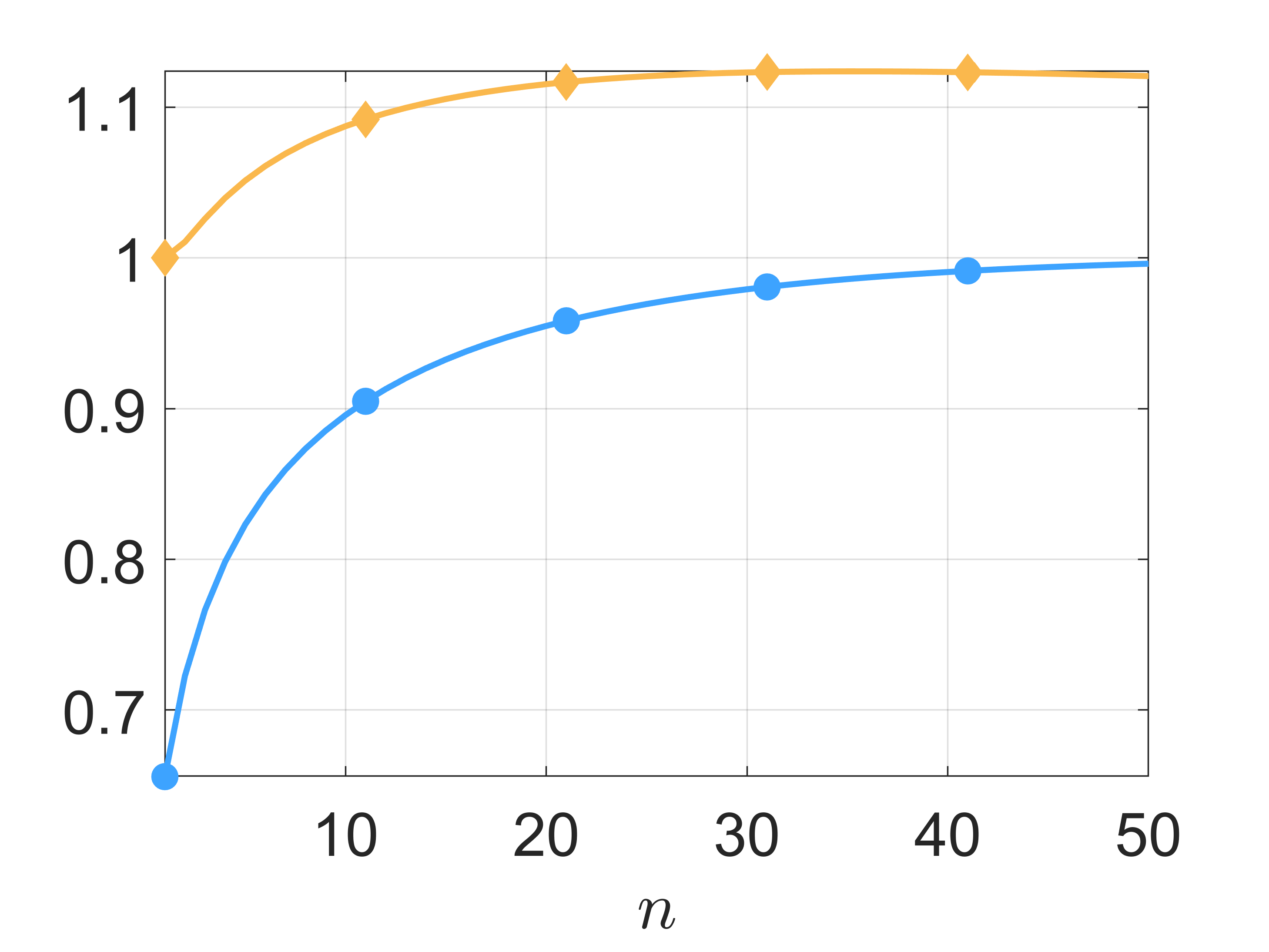} \\
\end{center}
\caption{\small Left: value $\ma^*$ minimising $D_{\mu,s}(\QQ^{[n]}_{\ma,1})$ as function of $n$. Center: $D_{\mu,s}^{1/s}(\QQ^{[n]}_{\ma^*,a})$ ({\color{red} {\bf - - -}}) and $D_{\mu,s}^{1/s}(\QQ^{[n]}_{1,a})$ ({\color{Cerulean} {\bf ---}}) as functions of $a$ for $n=20$ and $\ma^* = 0.846$. Right: $D_{\mu,s}^{1/s}(\QQ^{[n]}_{\ma^*,1})/D_{\mu,s}^{1/s}(\PP^{[n]}_1)$ ({\color{Cerulean} $\bullet$}) and $D_{\mu,s}^{1/s}(\QQ^{[n]}_{\ma^*,1})/D_{\mu,s}^{1/s}(\PP^{[n]}_{a^*})$ ({\color{Dandelion} $\blacklozenge$}) as functions of $n$; $s=10$, $d=10$.}
\label{F:sphere2}
\end{figure}

\subsection{Comparison between random quantisers uniform on a sphere and optimised full factorial $2^d$ designs} \label{S:sphere-2^d}
Let $n=2^d$ and $\Xb_n(b)$ be the full factorial $2^d$ design with points $\xb_i$ having coordinates $\pm b$, $b\in(0,1]$: the $\xb_i$ are vertices of the cube $[-b,b]^d$; they also belong to the sphere $\SS_{d-1}(a)$ with $a=\sqrt{d}\,b$. For $\SX=\SS_{d-1}(1)$, the Voronoi regions $\SV(\xb_i)=\{\xb\in\SX: \|\xb-\xb_i\| = \min_{\xb_j\in\Xb_n(b)} \|\xb-\xb_j\|\}$ are all identical up to a permutation of coordinates and do not depend on the value of $b$.

Consider first the case $s=2$. For $\mu$ uniform on $\SX=\SS_{d-1}(1)$, $E_{2,\mu}^2[\Xb_n(b)]=\Ex\{\|V-\bb\|^2\}$ where $\bb=(b,\ldots,b)\in\RR^d$ and $V=(V_1,\ldots,V_d)$ is uniform in $\SV(\bb)=\SS_{d-1}(1)\cap \RR_+^d$. Therefore
\bea
E_{2,\mu}^2[\Xb_n(b)]= \Ex\{ \|V\|^2 - 2\,b\sum_{i=1}^d  V_i + d b^2\} = 1 - 2\,bd\, \Ex\{V_1\} + db^2 \,.
\eea
Minimisation of $E_{2,\mu}^2[\Xb_n(b)]$ with respect to $b$ yields $b_2^*=\Ex\{V_1\}$ and the optimal $(\mu,2)$-distortion of a $2^d$ factorial design is $E_{2,\mu}^2[\Xb_n(b_2^*)]=1-d (b_2^*)^2$.
For $U^{(d)}=(U_1,\ldots,U_d)\simd \mu$, the first component $U_1$ has the density $\varphi_1(\cdot)$ of Lemma~\ref{L:projections-beta} on $[-1,1]$, therefore, $V_1$ has the density $2\,\varphi_1(\cdot)$ on $[0,1]$, so that
\bea
b_2^* =\Ex\{V_1\} &=& \frac{\Gamma(d/2)}{\sqrt{\pi}\,\Gamma((d+1)/2)}\,, \\
E_{2,\mu}^2[\Xb_n(b_2^*)] &=& 1 - \frac{d}{\pi} \left(\frac{\Gamma(d/2)}{\Gamma((d+1)/2)}\right)^2 \,.
\eea

More generally, the analytic expression for the $(\mu,s)$-distortion $E_{\mu,s}^s[\Xb_n(b)]$ can be derived for any even $s$, but the calculations become more complicated as $s$ increases. For $s=4$ we obtain
\bea
E_{4,\mu}^4[\Xb_n(b)] &=& \Ex\{\|V-\bb\|^4\} \\
&=& 1 +4\,b^2\, \Ex\left\{\left(\sum_{i=1}^d V_i \right)^2\right\} + d^2 b^4 + 2\,d b^2 - 4\, bd\, \Ex\{V_1\} (1+d b^2) \,.
\eea
Using $\Ex\{(\sum_{i=1}^d V_i)^2\}=1+d(d-1) \Ex\{V_1V_2\}$, where $\Ex\{V_1V_2\}=2/(\pi d)$ from $\varphi_2(\cdot,\cdot)$ of Lemma~\ref{L:projections-beta}, together with the expression above for $\Ex\{V_1\}$, we can express $E_{4,\mu}^4[\Xb_n(b)]$ as a fourth-degree polynomial in $b$ with coefficients depending on $d$. The optimal value $b_4^*$ can be obtained explicitly.
The values of $b_2^*$ and $b_4^*$ are very close for large $d$ as $a_j^*=\sqrt{d}\,b_j^*=\sqrt{2/\pi}+\SO(1/d)$ for $j=2,4$.

For $s=\infty$, $E_{\infty,\mu}[\Xb_n(b_\infty^*)]=\CR[\Xb_n(b)]$, the covering radius of $\Xb_n(b)$, with $(\CR[\Xb_n(b)])^2=1+d\,b^2-2\,b$ being minimum for $b_\infty^*=1/d$, which gives
$\CR[\Xb_n(b_\infty^*)]=\sqrt{1-1/d}$.

Figure~\ref{F:sphere3} compares the optimal full factorial designs $\Xb_n(b_s^*)$ for $s=2$ and 4 with optimised random quantisers $\PP_{a^*}^{[n]}$ on $\SS_{d-1}(1)$ with the same sample size $n=2^d$. The left panel shows the values of the radii of the various spheres on which the points lie, as functions of $d$. The right panel presents $D_{\mu,s}^{1/s}(\PP^{[n]}_{a^*})$ for the random quantisers and $E_{\mu,s}[\Xb_n(b_s^*)]$ for the full factorial designs, for $s=2$ and 4. For both types of point sets, the behaviours for $s=2$ and $s=4$ are similar. As $d$ tends to infinity, for the full factorial designs $a_s^*=\sqrt{d}\,b_s^*$ tends to the limit $\sqrt{2/\pi}\simeq 0.797885$ from above. For random quantisers, we observe that $a^*(s)=a^*(d,n,s)$ defined in Section~\ref{S:sphere} tends from below to a larger limiting value---we shall see in Section~\ref{S:extreme-value-sphere} that this limiting value equals $\sqrt{3}/2 \simeq 0.8660$.
The expected $(\mu,s)$-distortion of random quantisers shows little sensitivity to the choice of $a$ in the neighborhood of $a^*(s)$.
In particular, the plots of $D_{\mu,s}^{1/s}(\PP^{[n]}_{a_s^*})$ (not shown) and $D_{\mu,s}^{1/s}(\PP^{[n]}_{a^*(s)})$ almost coincide for $d>3$.

For both values of $s$, $a_s^* \approx a^*(s)$ for $d=7$. This dimension is the critical one at which random quantisers start exhibiting a smaller quantisation error than full factorial designs. As a consequence of the probabilistic method, it implies that for all $d\geq 7$ there exist non-random quantisers with $n=2^d$ that have smaller quantisation errors for $s=2$ and 4 than optimised full factorial designs.

\begin{figure}[ht!]
\begin{center}
\includegraphics[width=.49\linewidth]{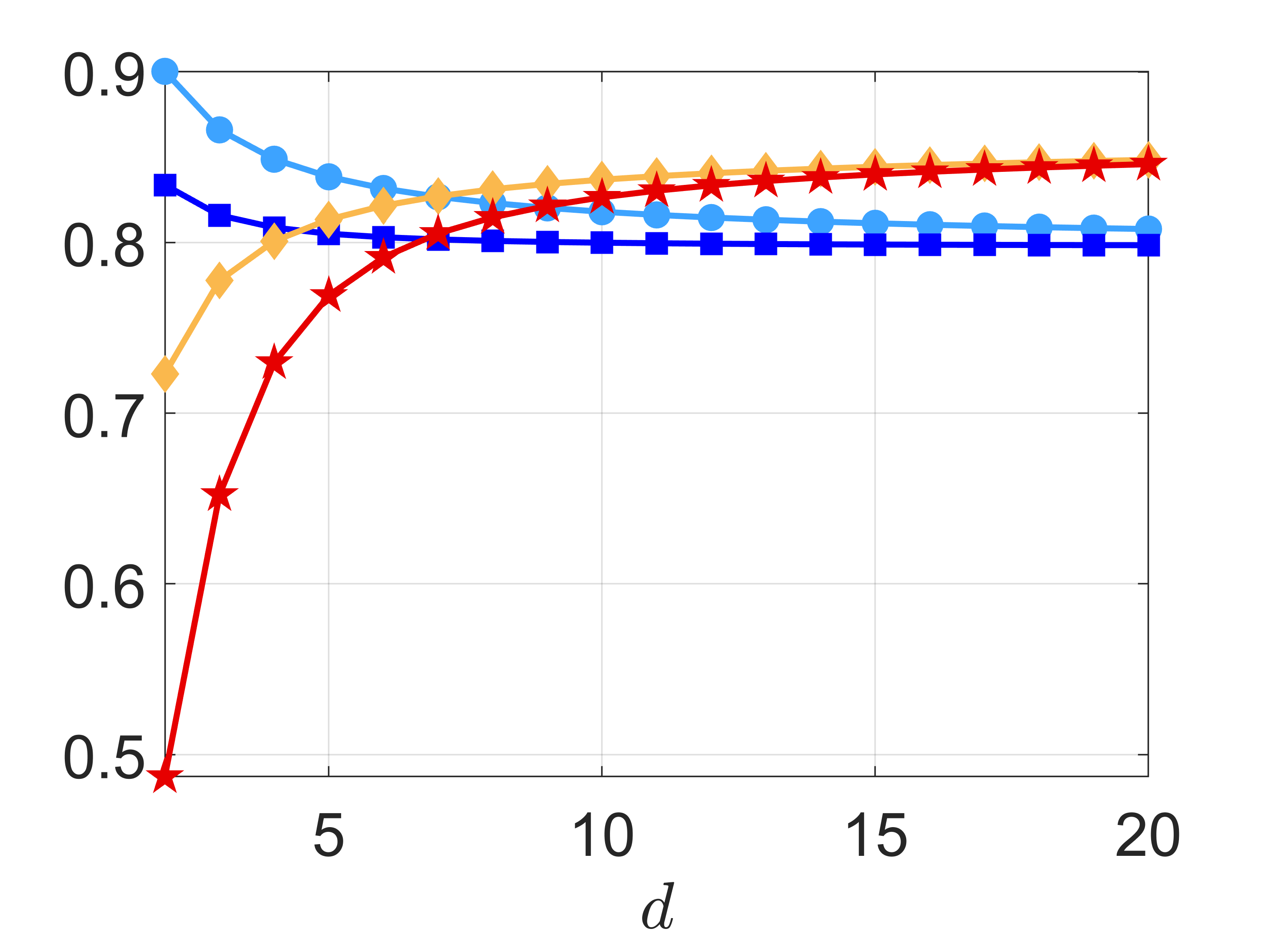}
\includegraphics[width=.49\linewidth]{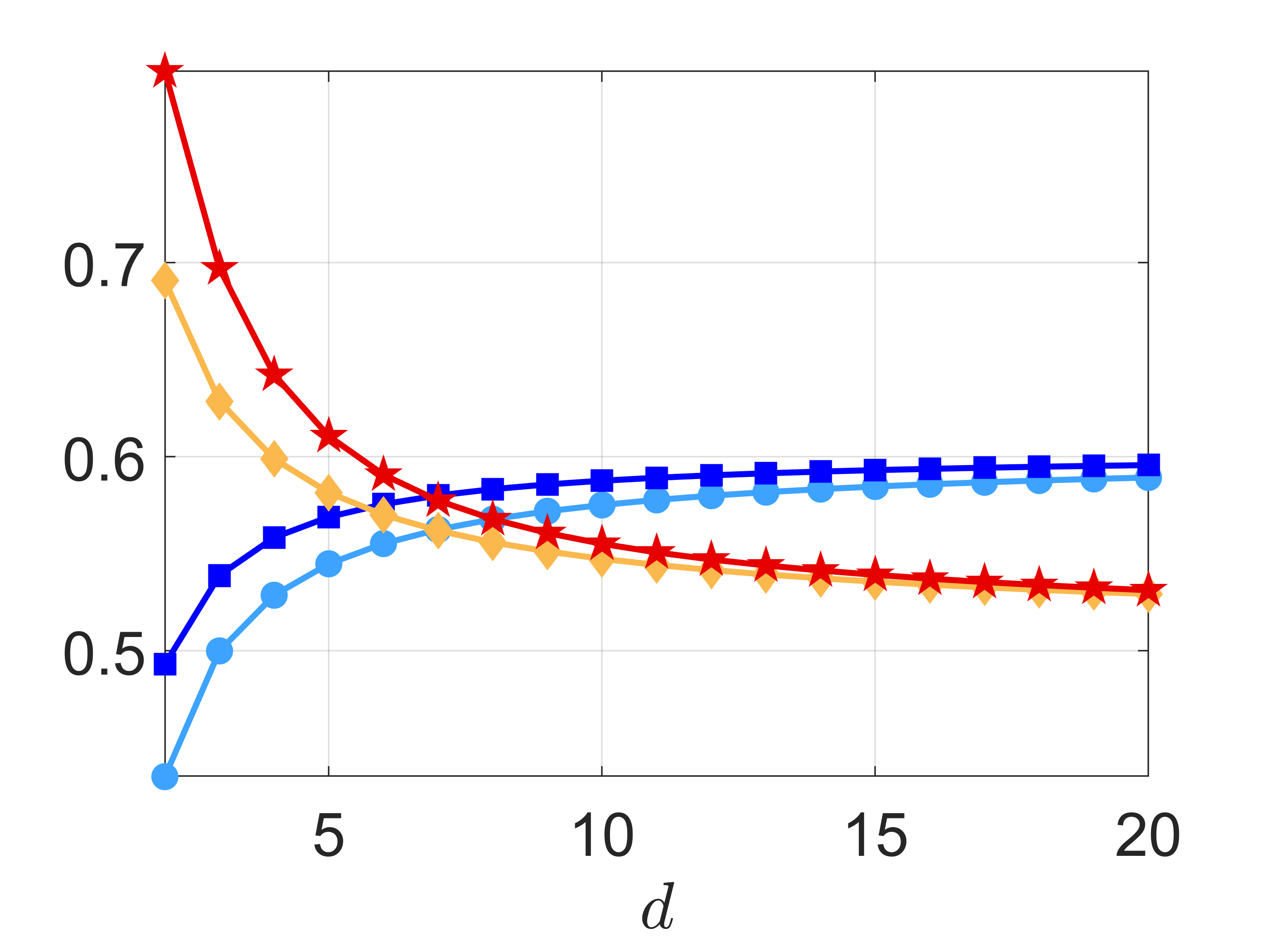} \\
\end{center}
\caption{\small Left: Values $a^*(s)$ minimising $D_{\mu,s}(\PP^{[n]}_a)$ for $n=2^d$ with $s=2$ ({\color{Dandelion} $\blacklozenge$}) and $s=4$ ({\color{red} $\bigstar$}), and values $a_s^*=\sqrt{d}\,b_s^*$ minimising $E_{\mu,s}[\Xb_n(b)]$ for $s=2$ ({\color{Cerulean} $\bullet$}) and $s=4$ ({\color{blue} {\tiny $\blacksquare$}}), as functions of $d$.
Right: $D_{\mu,s}^{1/s}(\PP^{[n]}_{a^*})$ for $n=2^d$ with $a^*=a^*(d,n,s)$, $s=2$ ({\color{Dandelion} $\blacklozenge$}) and $s=4$ ({\color{red} $\bigstar$}), and values $E_{\mu,s}[\Xb_n(b_s^*)]$ for $s=2$ ({\color{Cerulean} $\bullet$}) and $s=4$ ({\color{blue} {\tiny $\blacksquare$}}), as functions of $d$.}
\label{F:sphere3}
\end{figure}

\subsection{$\mu$ is uniform in the unit ball $\SB_d(1)$} 
\label{S:ball}




For $\mu$ uniform in $\SB_d(1)$, the density of $r=\|U\|$ for $U\simd\mu$ is $\psi(\tau)=d\,\tau^{d-1}$.





Consider first the case where $\PP=\PP_a$ uniform on $\SS_{d-1}(a)$: \eqref{Ls-mean-final} gives
\be\label{DS-Pa}
D_{\mu,s}(\PP_a^{[n]}) =  s \int_{t\geq 0} t^{s-1} \int_{r\geq 0}  H(r,t;\Phi_a,n) \,\psi(r)\,\dd r \,\dd t \,,
\ee
where $H(r,t;\Phi_a,n) =\bigg(1-\PP_a \left\{ \| X-\ub \| \leq t \right\} \bigg)^n$ with $\Phi_a(\cdot)$ the c.d.f.\ of the delta measure $\delta_a$ and $r=\|\ub\|$, so that $\PP_a\{ \| X-\ub \| \leq t \}$ is given by \eqref{PPaa}.

We also consider random quantisers whose $n$ points are i.i.d.\ with $\PP=\PP_{0,b}$ uniform in $\SB_d(b)$.
The c.d.f.\ $\Phi(\cdot)$ of $R=\|X\|$ has then the density $\phi_{0,b}(\rho)=d\,\rho^{d-1}/b^d$ for $\rho\in[0,b]$.
More general spherically symmetric distributions $\PP$ could also be used, but we found that, unless $n$ is extremely small, the class of distributions $\PP_{0,b}$ is rich enough.
The expected $(\mu,s)$-distortion $D_{\mu,s}(\PP_{0,b}^{[n]})$ is given by \eqref{Ls-mean-final}
where
\bea
H(r,t;\Phi,n) =\bigg(1- \int_a^b I_\upsilon(\delta,\delta) \,\phi_{0,b}(\rho)\,\dd\rho \bigg)^n \,,
\eea
with $\upsilon=\upsilon(t,\rho,r)$ given by \eqref{v}.

Figure~\ref{F:ball} presents the same information as Figure~\ref{F:sphere} for the ball configuration: we consider optimal design defined by $\PP_{0,b^*}$, where $b^*=b^*(d,n,s)$ minimises $D_{\mu,s}(\PP_{0,b}^{[n]})$ and the right column compares performances of $\PP_{0,b^*}$ and $\PP_{0,1}$.

\begin{figure}[ht!]
\begin{center}
\includegraphics[width=.32\linewidth]{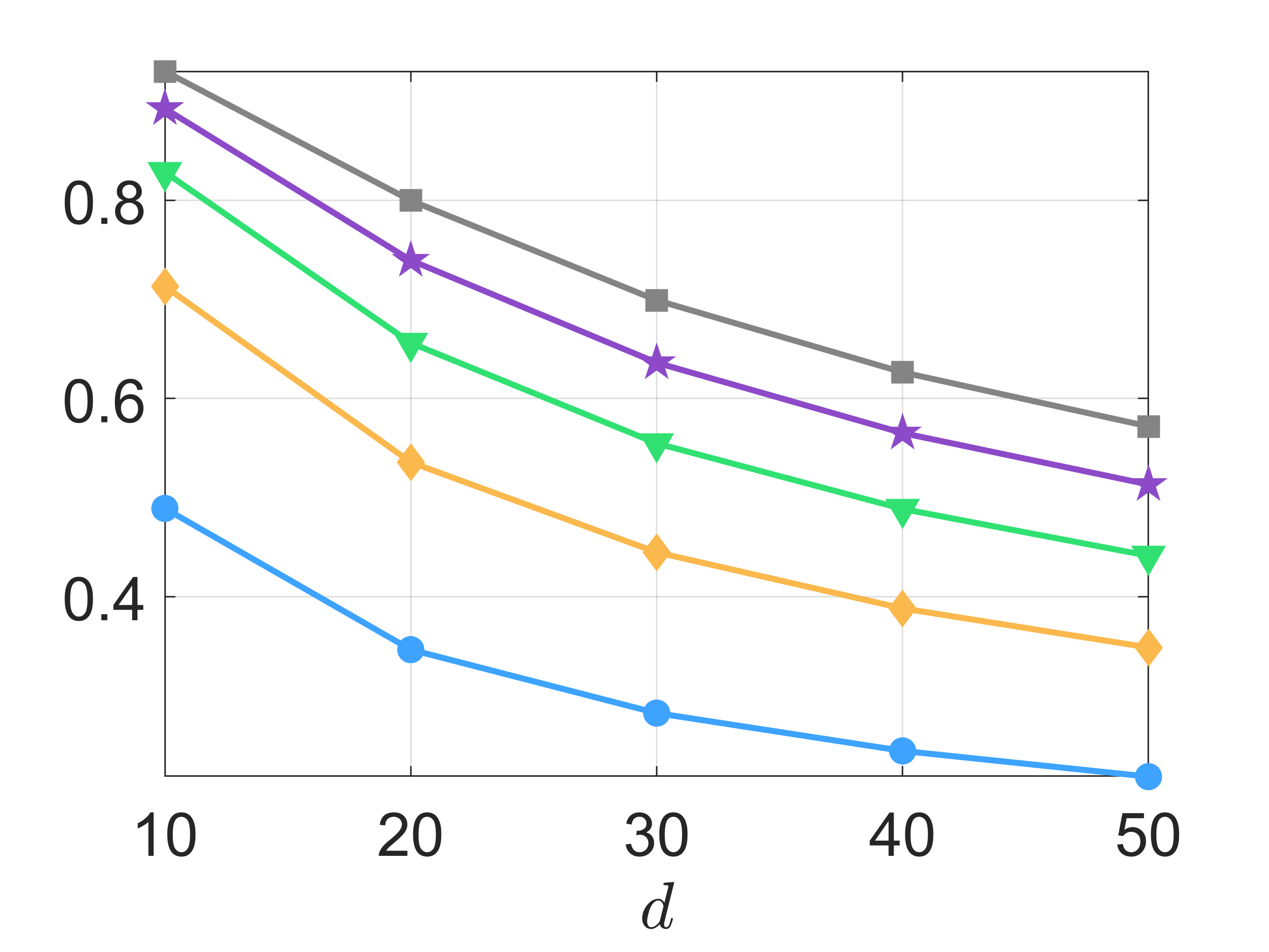}
\includegraphics[width=.32\linewidth]{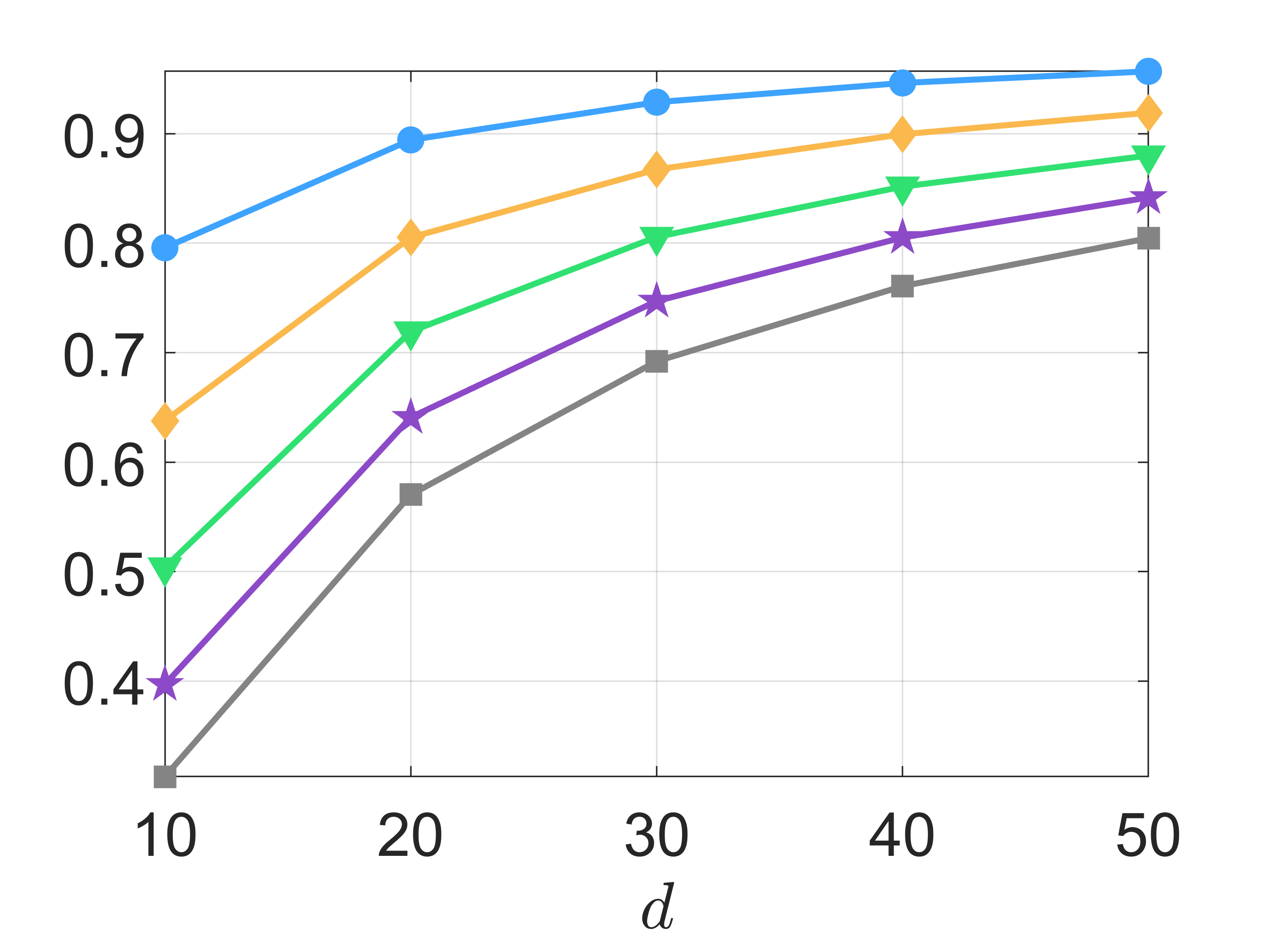}
\includegraphics[width=.32\linewidth]{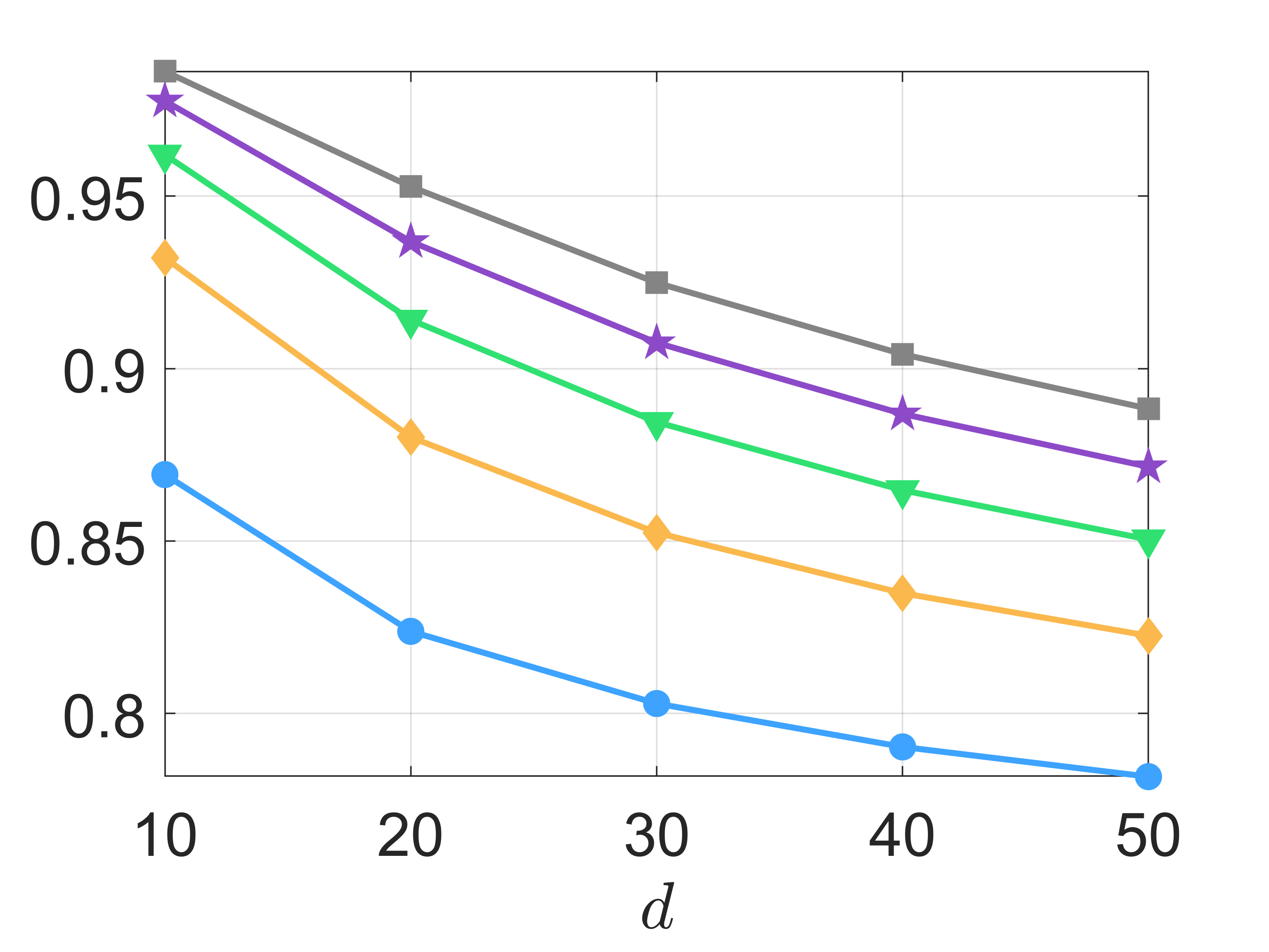} \\
\includegraphics[width=.32\linewidth]{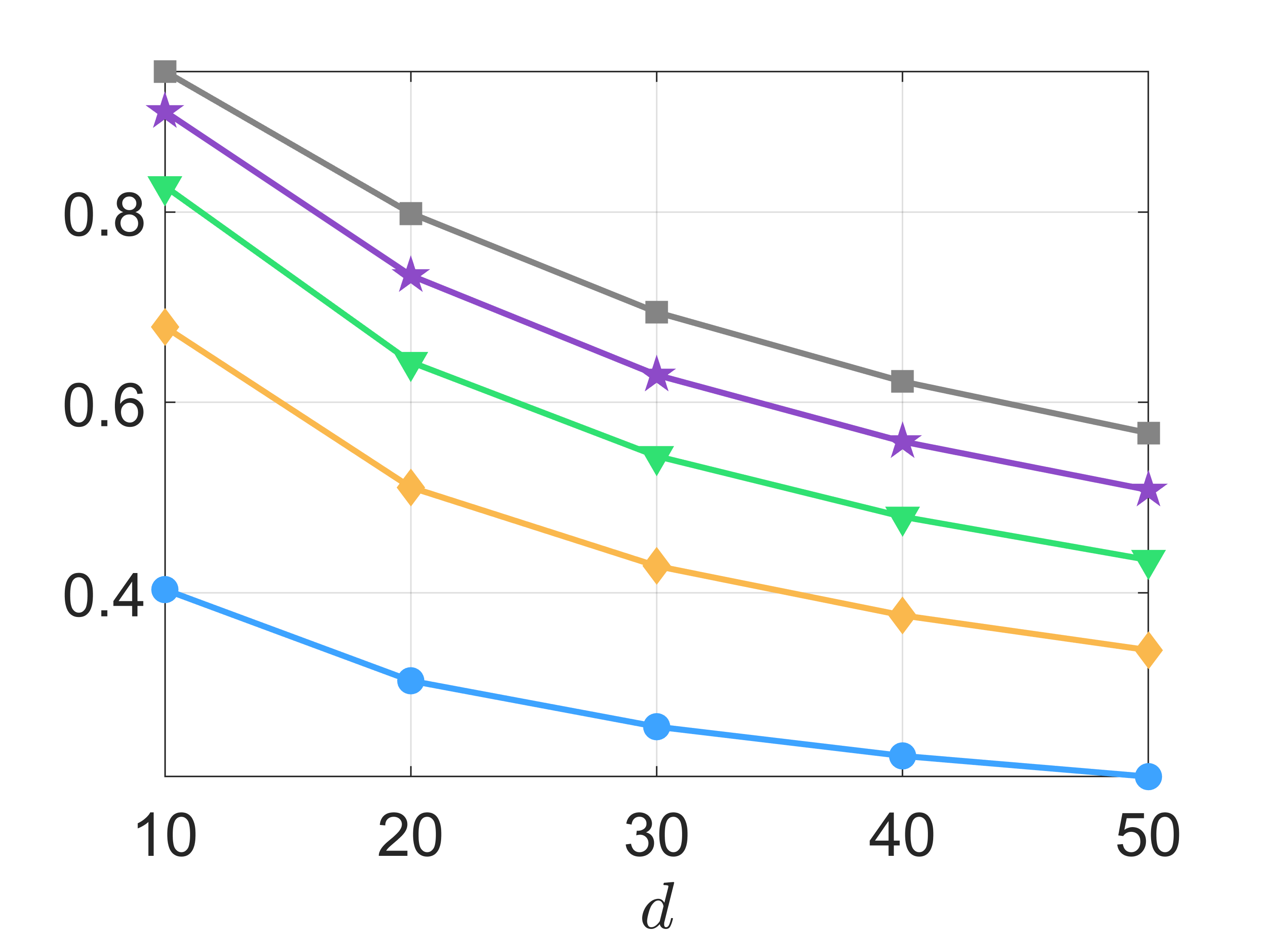}
\includegraphics[width=.32\linewidth]{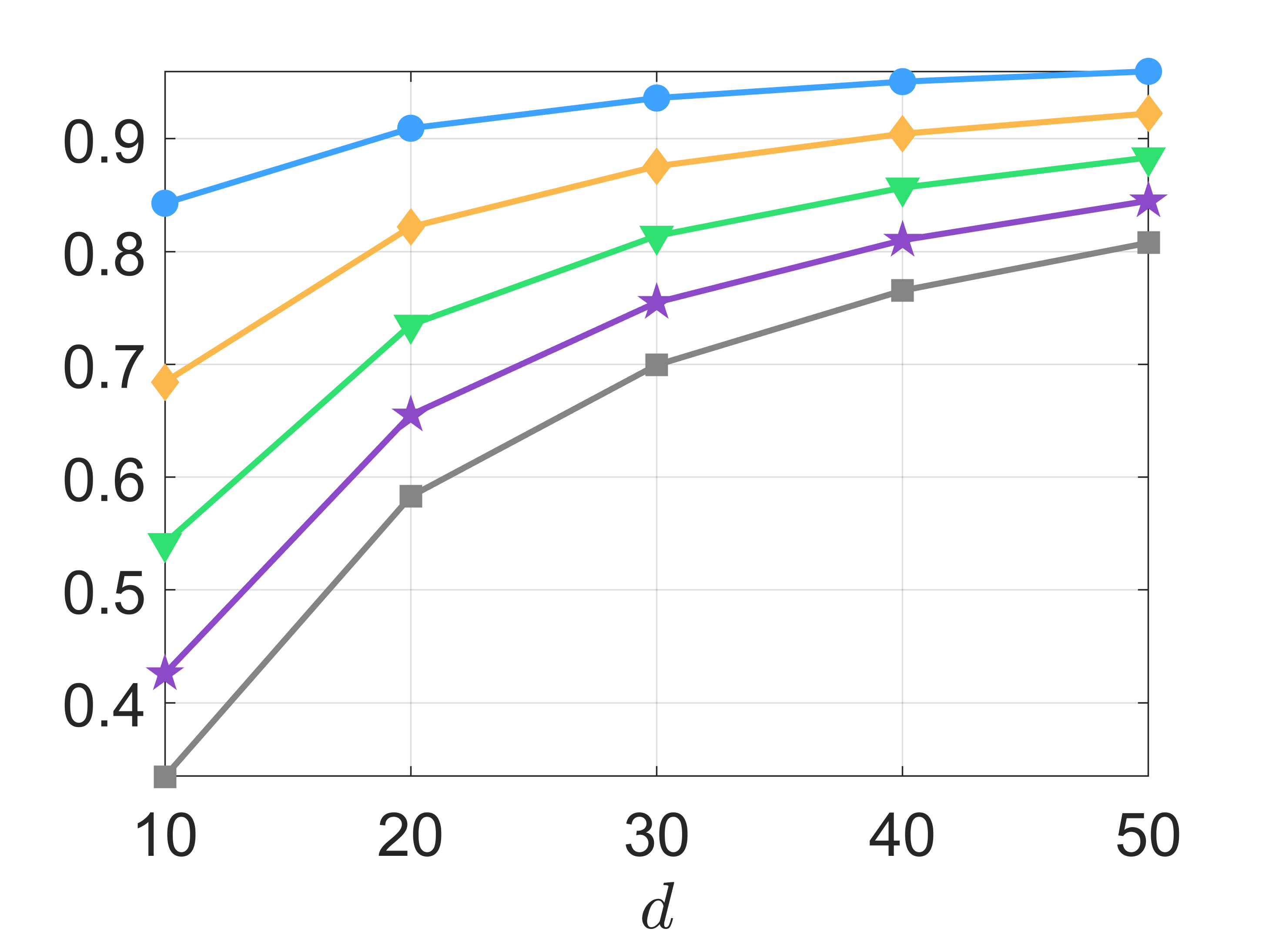}
\includegraphics[width=.32\linewidth]{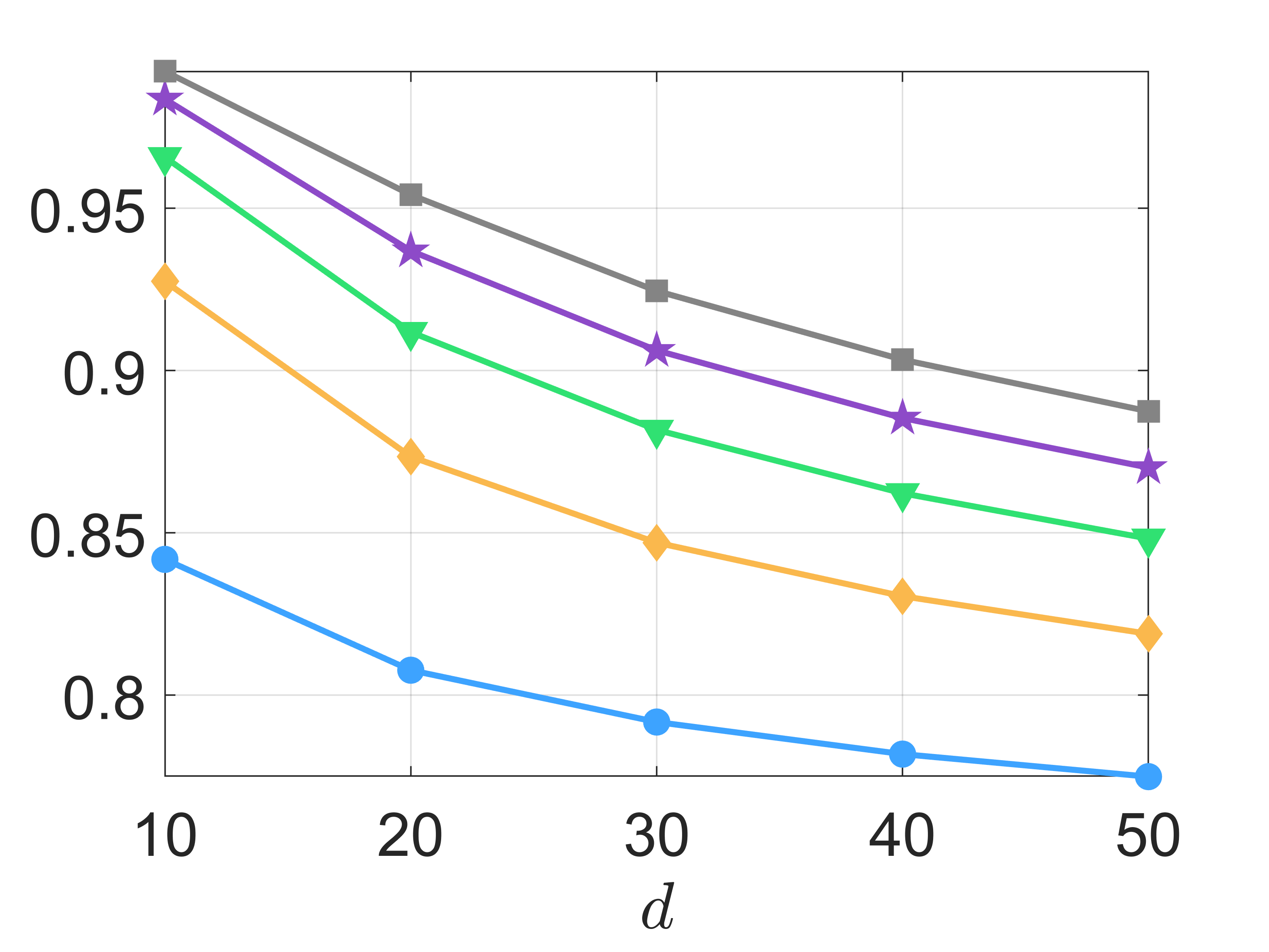}
\end{center}
\caption{\small Ball. Value $b^*$ minimising $D_{\mu,s}(\PP^{[n]}_{0,b})$ (left column), $D_{\mu,s}^{1/s}(\PP^{[n]}_{0,b^*})$ (central column) and ratio $D_{\mu,s}^{1/s}(\PP^{[n]}_{0,b^*})/D_{\mu,s}^{1/s}(\PP^{[n]}_{0,1})$ (right column) as functions of $d$ for $s=1$ (top row) and $s=10$ (bottom row) and different values of $n$: $n=10$ ({\color{Cerulean} $\bullet$}), $n=10^2$ ({\color{Dandelion} $\blacklozenge$}), $n=10^3$ ({\color{green} $\blacktriangledown$}), $n=10^4$ ({\color{Orchid} $\bigstar$}), and $n=10^5$ ({\color{gray} {\tiny $\blacksquare$}}).}
\label{F:ball}
\end{figure}

The quantities considered behave qualitatively as in Figure~\ref{F:sphere}.  Note, however, that $b^*$ is larger than $a^*$ plotted in Figure~\ref{F:sphere} (left columns---$b^*$ is much larger than $a^*$ on the bottom row with $s=10$), that $D_{\mu,s}(\PP_{0,b^*}^{[n]})$ is smaller than $D_{\mu,s}(\PP_{a^*}^{[n]})$ (central columns), and that the efficiency of a ``classical" random design compared to an optimised design is slightly larger for the ball than for the sphere (right columns). 
For a fixed dimension $d$, the empirical distribution of an optimal quantiser which minimises $E_{\mu,s}(\Xb_n)$ converges weakly to $\PP_{0,1}^{[n]}$ as $n \to \infty$ (see \cite[Table~7.1]{GrafL2000}). Consistent with this theoretical result, we empirically observe that the optimal parameters $b^*$ and $a^*$ both converge to 1 as $n \to \infty$; see also Section~\ref{S:extreme-value-ball} for further details on their asymptotic behaviour.

Figure~\ref{F:bstar_ALLastar_ball_s2_D} of Section~\ref{S:extreme-value-ball} displays the optimal values $b^*(d,n,s)$ and $a^*(d,n,s)$, respectively minimising $D_{\mu,s}(\PP_{0,b}^{[n]})$ and $D_{\mu,s}(\PP_a^{[n]})$, plotted as functions of $d$ for $s=2$ and two distinct values of $n$. The figure shows that $b^*>a^*$, and the gap $b^*-a^*$ narrows as $d$ increases, reflecting the fact that both $a^*$ and $b^*$ converge to zero as $d\to\infty$. The behaviour remains qualitatively unchanged for other values of $s$, especially when the dimension $d$ is large.


Given the observations in Figures~\ref{F:ball} and 
\ref{F:bstar_ALLastar_ball_s2_D}, two key patterns emerge:
(\textit{i}) for large dimensions $d$, the values of $a^*(d,n,s)$ and $b^*(d,n,s)$ become very close; (\textit{ii}) the quantiser $\PP_{0,b^*}^{[n]}$ significantly outperforms $\PP_{0,1}^{[n]}$ when $d$ is sufficiently large. Moreover, we also observe that $\PP_{a^*}^{[n]}$ outperforms $\PP_{0,b^*}^{[n]}$ for reasonable design sizes $n$; see Figure~\ref{F:ALLeff_ball_s2_D} in Section~\ref{S:extreme-value-ball}.
Since $b^*$ converges to 1 as $n \to \infty$, with $\PP_{0,1}^{[n]}$ being asymptotically optimal, a natural question arises: for a fixed dimension $d$, at what sample size $n=n^*(d,s)$ does $\PP_{a^*}^{[n]}$ stop being preferable to $\PP_{0,b^*}^{[n]}$? Empirical analysis reveals that this transition occurs only for extremely large values of $n$.
Define
\be\label{nstar-ball}
n^*(d,s) = \max\{n\in\NN: D_{\mu,s}(\PP_{a^*}^{[n]}) < D_{\mu,s}(\PP_{0,b^*}^{[n]}) \} \,.
\ee
The left panel of Figure~\ref{F:nstar-ball_s2-4-10} presents $n^*(d,s)$ (in log scale, obtained by dichotomous search) as a function of $d=3,\ldots,20$ for three values of $s$, indicating a superexponential increase of $n^*(d,s)$ with $d$.

\begin{figure}[ht!]
\begin{center}
\includegraphics[width=.49\linewidth]{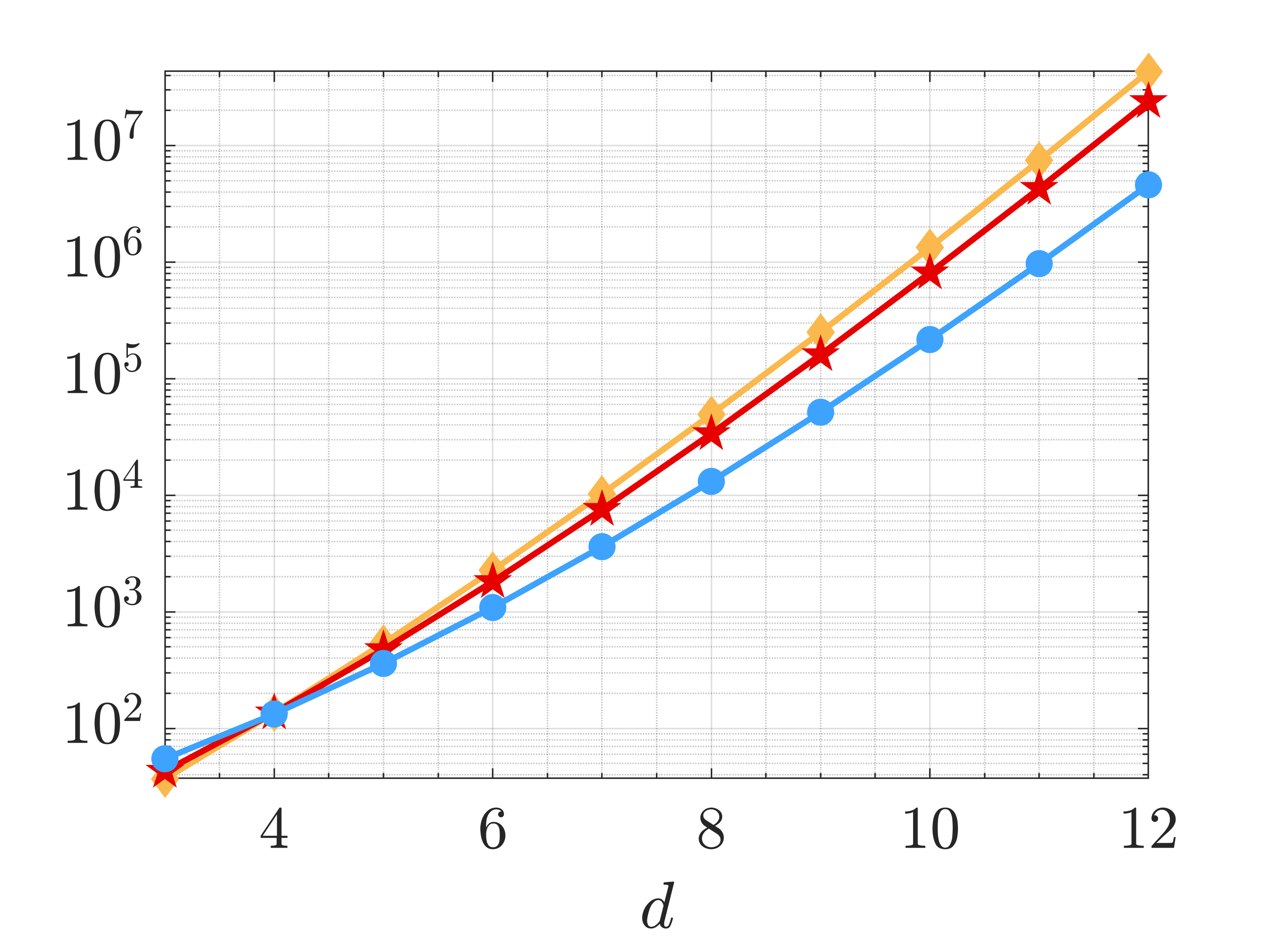}
\includegraphics[width=.49\linewidth]{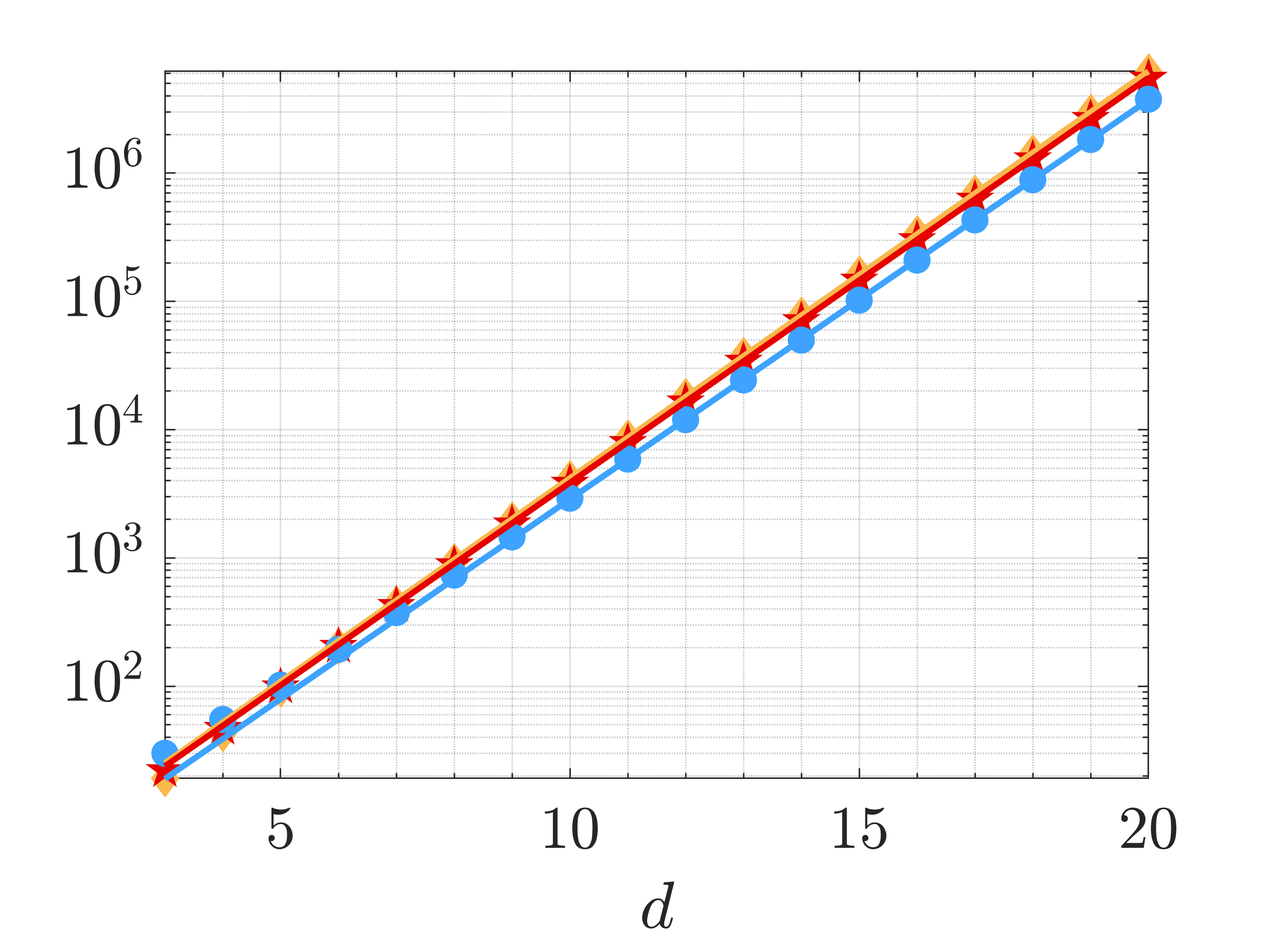}
\end{center}
\caption{\small $n^*(d,s)$ as a function of $d$ for $s=2$ ({\color{Dandelion} $\blacklozenge$}), $s=4$ ({\color{red} $\bigstar$}) and $s=10$ ({\color{Cerulean} $\bullet$}). Left: $\mu$ is uniform on $\SS_{d-1}(1)$ and $n^*(d,s)$ is given by \eqref{nstar-ball}. Right: $\mu$ is spherically normal and $n^*(d,s)$ is given by \eqref{nstar}.
}
\label{F:nstar-ball_s2-4-10}
\end{figure}


\subsection{$\mu$ is a spherically symmetric normal distribution}\label{S:normal}

Any spherically symmetric normal distribution in $\RR^d$ can be renormalised as $\SN(\0b_d,\Ib_d/d)$, with $\Ib_d$ the $d$-dimensional identity matrix. In this section, we assume that $\mu$ is the corresponding probability measure. When $U\simd\mu$, $d\,\|U\|^2$ has the chi-square distribution with $d$ degrees of freedom, so that the density of $\|U\|$ is
\be\label{psi-normal}
\psi(r)=\frac{d^{d/2}}{2^{d/2-1}\,\Gamma(d/2)}\,r^{d-1}\,\e1^{-d r^2/2} \,, \ r \geq 0\,.
\ee
The distribution of $r$ is more and more concentrated around 1 as $d$ increases. Its $k$-th moment is
\bea
M_{\psi,k} = \Ex_\mu\{\|U\|^k\}=(2^{k/2}/d^{k/2}) \Gamma((k+d)/2)/\Gamma(d/2)\,, \ k=1,2,\ldots
\eea
so that its mean is smaller than 1 for any $d$ and its variance is $1/(2\,d)-1/(8\, d^2)+ \SO(1/d^3)$, $d\to\infty$.

When $\PP=\PP_a$ uniform on $\SS_{d-1}(a)$, the expected $(\mu,s)$-distortion $D_{\mu,s}(\PP_a^{[n]})$ is given by \eqref{DS-Pa},
with $\psi(\cdot)$ given by \eqref{psi-normal}.

\vsp
When $\PP=\mu_\ms$ is the probability measure of $\SN(\0b_d,\ms^2\Ib_d/d)$, $R=\|X\|$ has the p.d.f.\ $\phi_\ms(\rho)=(1/\ms)\psi(\rho/\ms)$ where $\psi(\cdot)$ given by \eqref{psi-normal}, and the expected $(\mu,s)$-distortion $D_{\mu,s}(\mu^{[n]}_\ms)$ is given by \eqref{Ls-mean-final}
where
\bea
H(r,t;\Phi,n) =\bigg(1-\frac{1}{\ms}\,\int_{\rho\geq 0} I_\upsilon(\delta,\delta) \,\psi(\rho/\ms)\,\dd\rho \bigg)^n \,,
\eea
with $\upsilon=\upsilon(t,\rho,r)$ given by \eqref{v}.

\vsp
Let $\ms^* = \ms^*(d,n,s)$ denote the value of $\ms$ that minimises the distortion $D_{\mu,s}(\mu^{[n]}_\ms)$ and $a^* = a^*(d,n,s)$ denote the value of $a$ that minimises $D_{\mu,s}(\PP_a^{[n]})$, as in previous sections.

The empirical distribution of an optimal quantiser $\Xb_n^*$ minimising $E_{\mu,s}(\Xb_n)$ weakly converges to $\SN(\0b_d,\ms_\infty^*\,\Ib_d/d)$ as $n\to\infty$, with $\ms_\infty^*=1+s/d$; see \cite[Table~7.1]{GrafL2000}. Figure~\ref{F:normal-a-sigma-s} illustrates that $\ms^*=\ms^*(d,n,s)$ increases with $s$, similarly to $\ms_\infty^*$, and that $a^*=a^*(d,n,s)$ also increases with $s$, contrary to Sections~\ref{S:sphere} and \ref{S:ball}. The reason is that large values of $r=\|U\|$ get increasing importance as $s$ increases, while $r=1$ in Section~\ref{S:sphere} and $r\leq 1$ in Section~\ref{S:ball}. Note that $\ms^*(d,n,s)$ is significantly smaller than $\ms_\infty^*=1+s/d$ (Figure~\ref{F:normal-a-sigma-s} is for $d=10$).

\begin{figure}[ht!]
\begin{center}
\includegraphics[width=.49\linewidth]{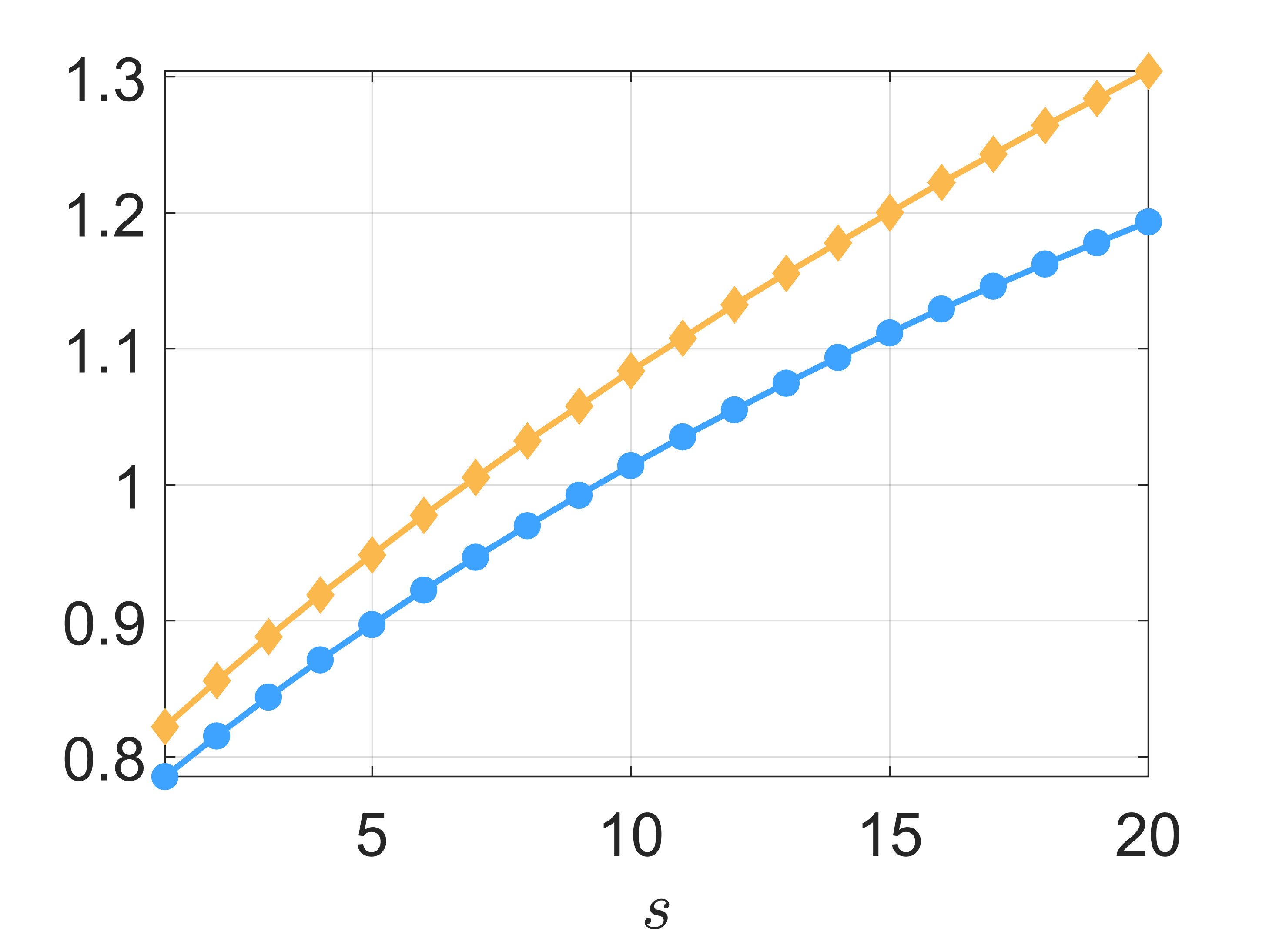}
\includegraphics[width=.49\linewidth]{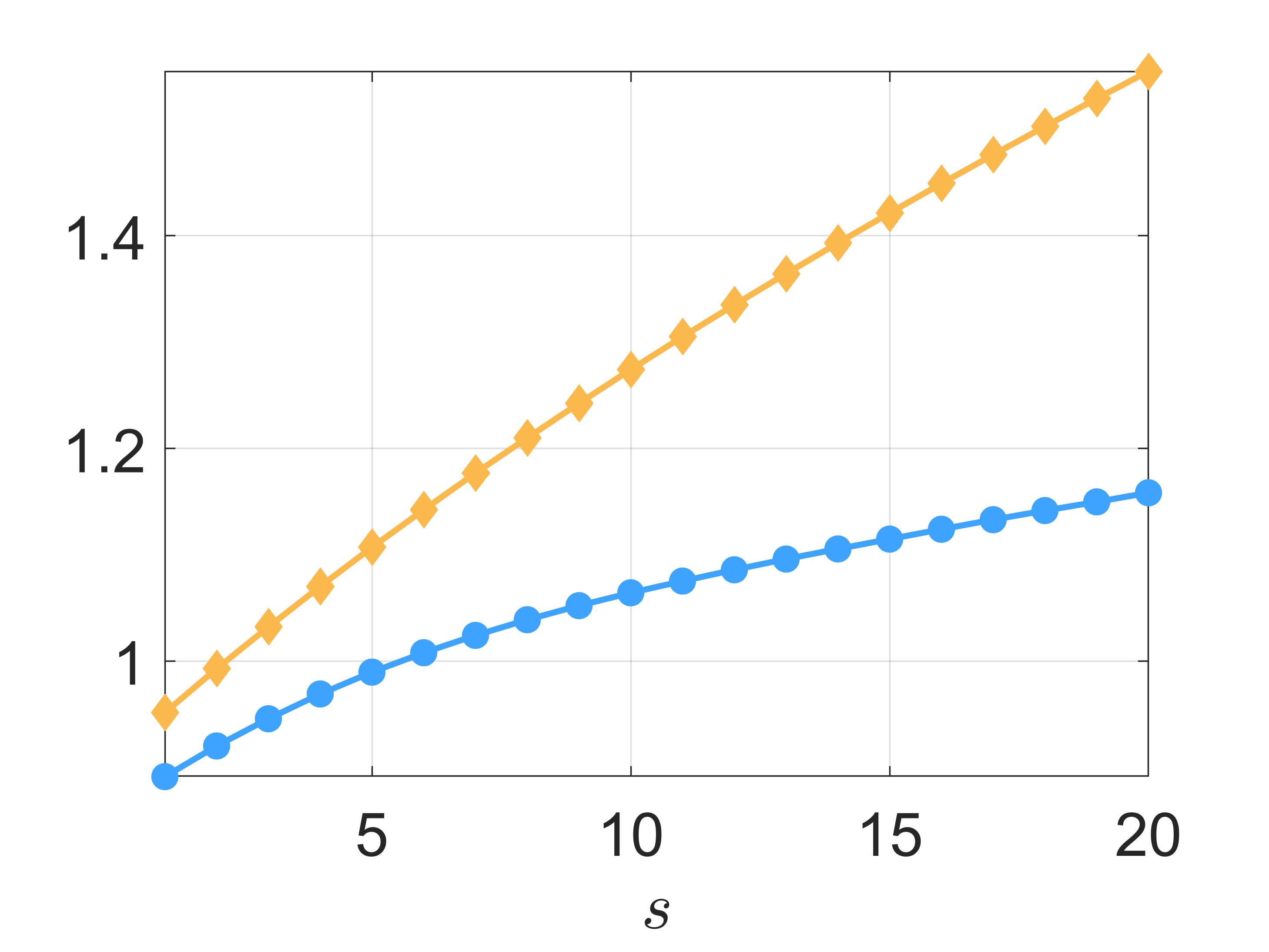}
\end{center}
\caption{\small Optimal values $a^*(d,n,s)$ ({\color{Cerulean} $\bullet$}) and $\ms^*(d,n,s)$ ({\color{Dandelion} $\blacklozenge$}) as functions of $s$ for $d=10$, $n=1\,000$ (left) and $n=100\,000$ (right).}
\label{F:normal-a-sigma-s}
\end{figure}


Figure~\ref{F:normal} presents $D_{\mu,s}^{1/s}(\mu^{[n]}_\ms)$ and $D_{\mu,s}^{1/s}(\PP_{a=\ms}^{[n]})$ as functions of $\ms$ for $s=2$, $n=1\,000$ and $d=5,10$ and $20$.
Note that, already for $d=5$ (left panel), $1\,000$ points are not enough for $\ms^*=\ms^*(d,n,s)\simeq 1.1832$ to approach the asymptotically optimal value $\ms^*_\infty=1.4$. For $d\geq 10$ (center and right panels) the plots of $D_{\mu,s}^{1/s}(\PP_\ms^{[n]})$ and $D_{\mu,s}^{1/s}(\mu^{[n]}_\ms)$ are close, especially near their minima.
The behaviour is similar when using other values of $s$.

The left panel of Figure~\ref{F:normal2} shows that $\ms^*-a^*$ tends to zero as $d$ increases, a consequence of $D_{\mu,s}(\mu_\ms^{[n]})$ tending to $D_{\mu,s}(\PP_\ms^{[n]})$. Moreover, as $d$ grows, $\psi(\cdot)$ is more and more concentrated around 1. This explains why for $d\geq 10$ the values of $a^*$ are close to those obtained for $\mu$ uniform on $\SS_{d-1}(1)$; see the curves on Figure~\ref{F:sphere}-top-left (they are plotted for $s=1$ but those obtained for $s=2$ are very similar).
The right panel presents the efficiencies of the asymptotically optimal distribution $\mu_{\ms^*_\infty}$ relative to $\PP_{a^*}$ and $\mu_{\ms^*}$, for $n=10\,000$, as functions of $d$. These efficiencies are always less than 1, meaning that the asymptotic regime is not yet reached for the values of $d$ considered. At $d=10$, however, we can observe that, contrary to the central panel of Figure~\ref{F:normal} where $n=1\,000$, now $D_{\mu,2}(\PP_{a^*}^{[n]}) > D_{\mu,2}(\mu^{[n]}_{\ms^*})$: with $n=10\,000$ points we approach the asymptotic regime and $\PP_{a^*}$ is dominated by a suitably chosen normal distribution.

\begin{figure}[ht!]
\begin{center}
\includegraphics[width=.32\linewidth]{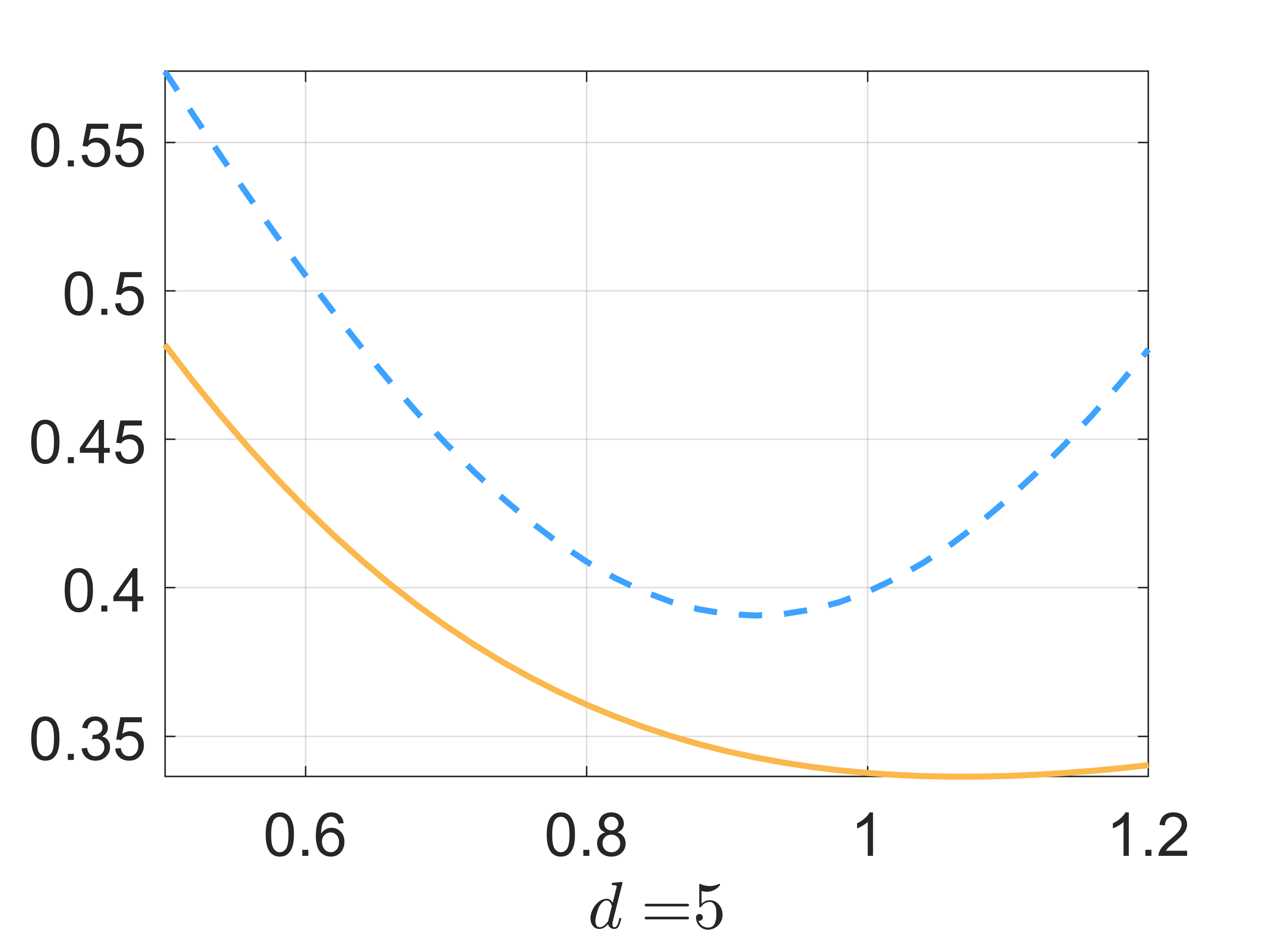}
\includegraphics[width=.32\linewidth]{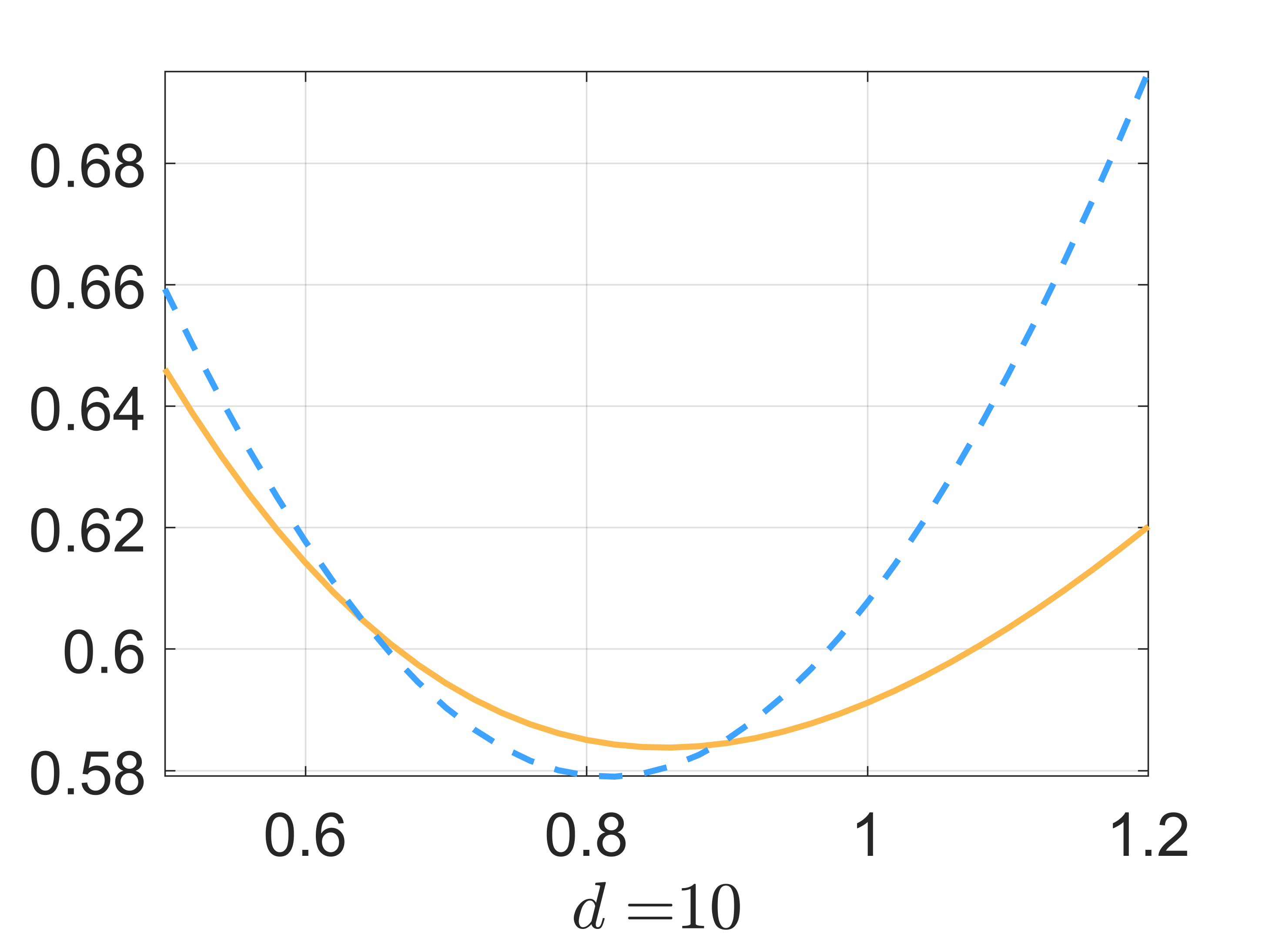}
\includegraphics[width=.32\linewidth]{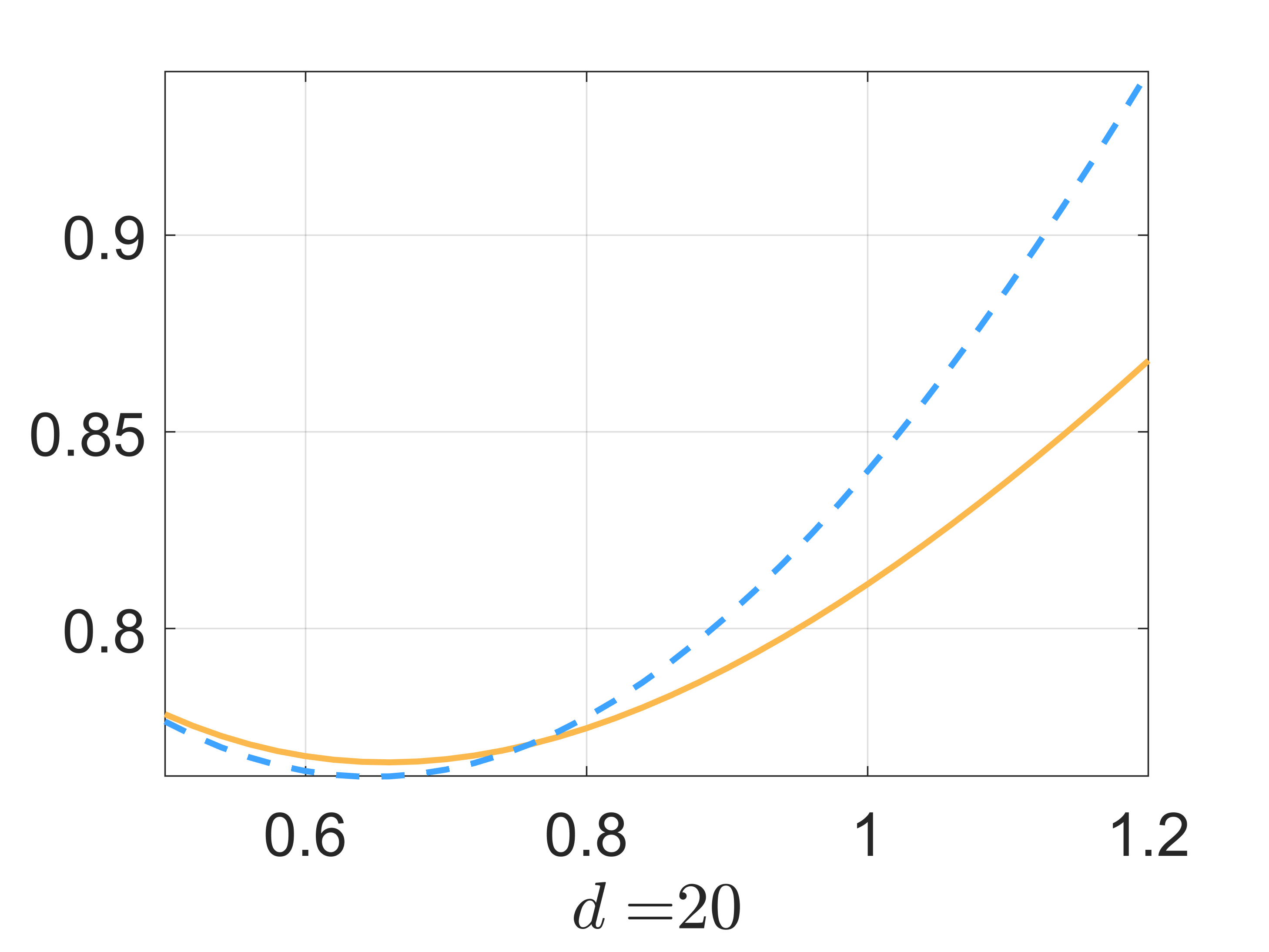}
\end{center}
\caption{\small $D_{\mu,s}^{1/s}(\PP^{[n]}_{a=\ms})$ ({\color{Cerulean} {\bf - - -}}) and $D_{\mu,s}^{1/s}(\mu^{[n]}_\ms)$ ({\color{Dandelion} {\bf ---}}) as functions of $\sigma$ for different $d$ with $n=1\,000$ and $s=2$.}
\label{F:normal}
\end{figure}


\begin{figure}[ht!]
\begin{center}
\includegraphics[width=.49\linewidth]{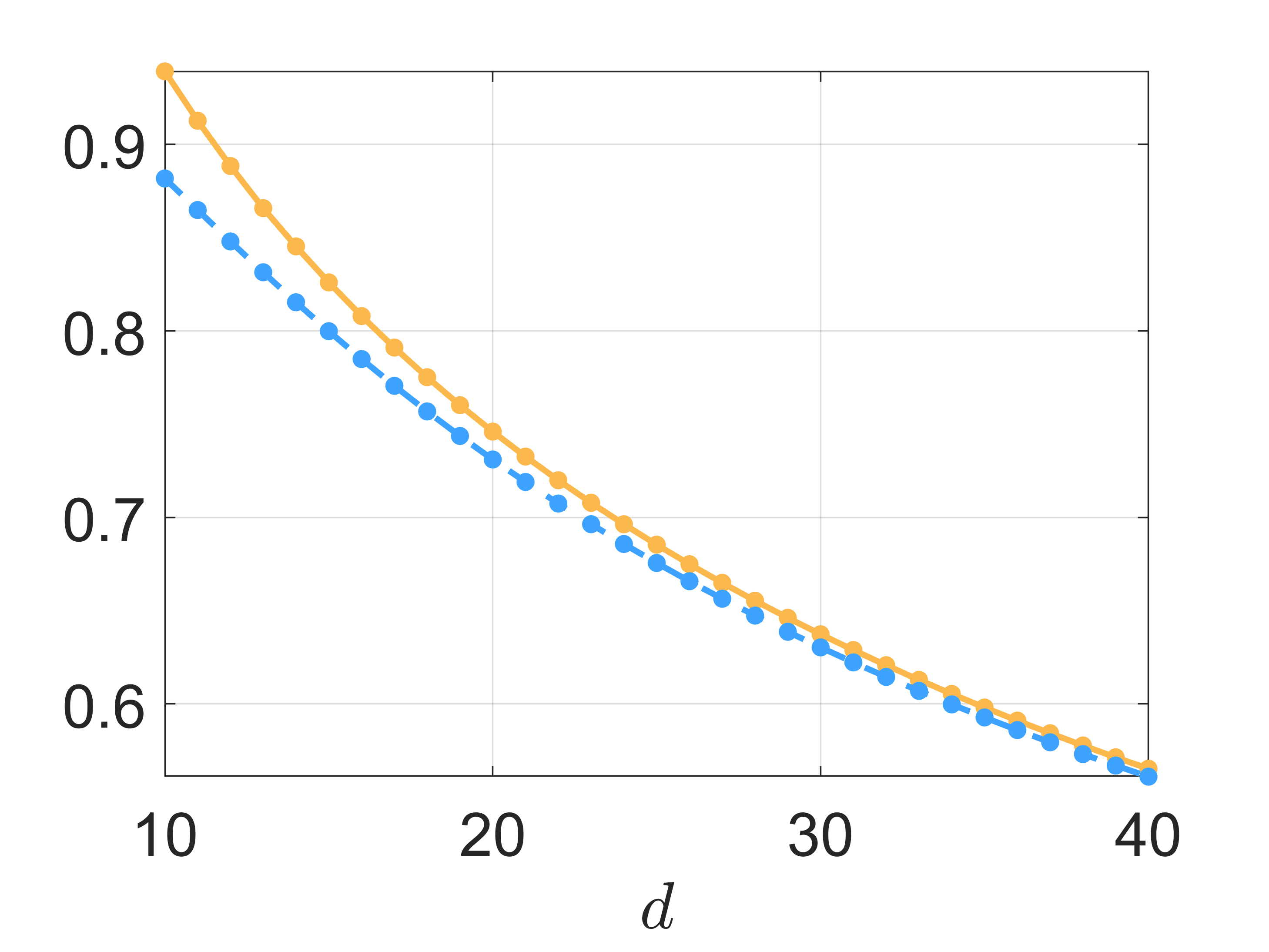}
\includegraphics[width=.49\linewidth]{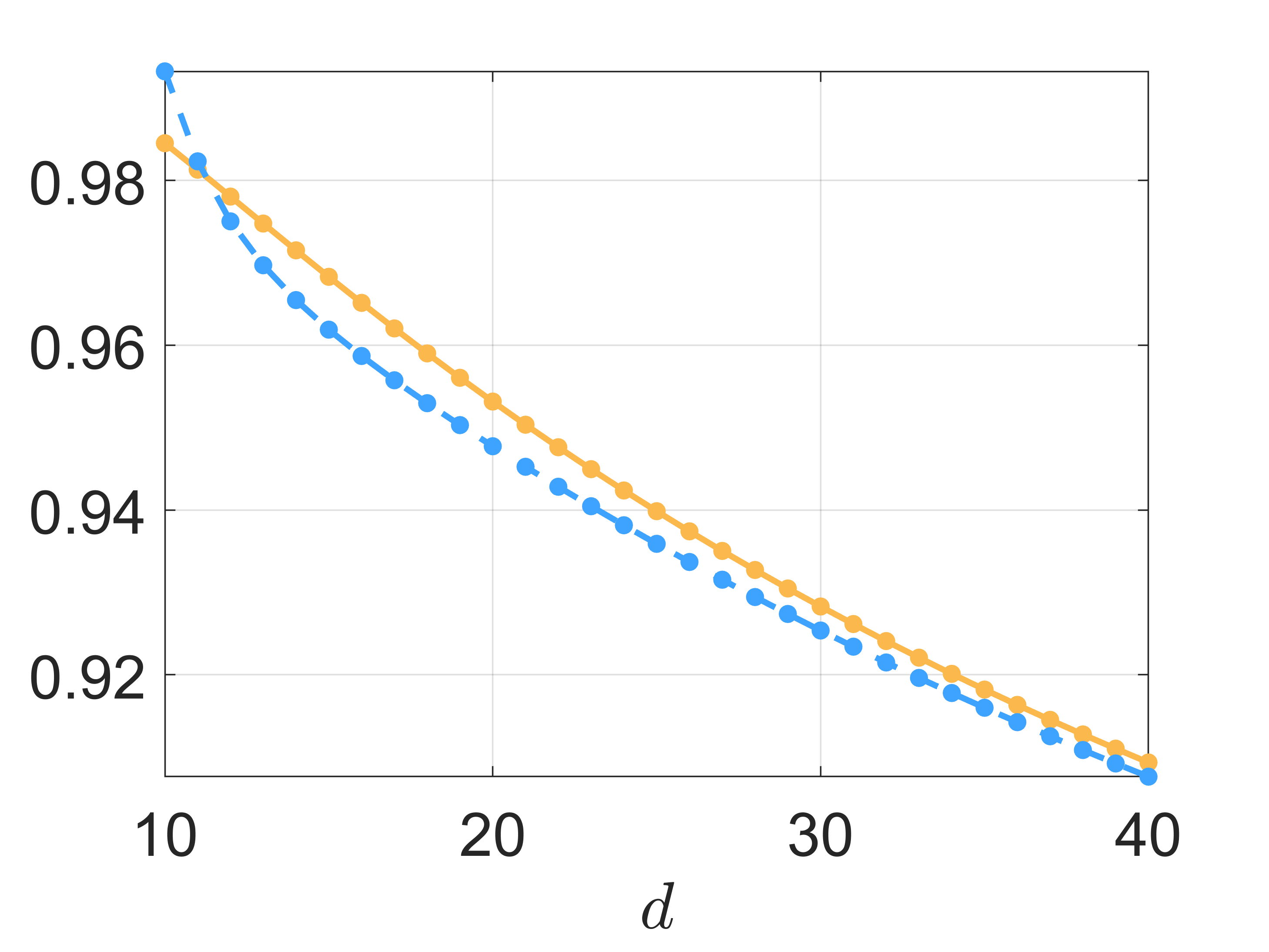}
\end{center}
\caption{\small Left: Optimal values $a^*(d,n,s)$ ({\color{Cerulean} {\bf - - -}}) and $\ms^*(d,n,s)$ ({\color{Dandelion} {\bf ---}}) as functions of $d$.
Right: ratios $D_{\mu,s}^{1/s}(\PP^{[n]}_{a^*})/D_{\mu,s}^{1/s}(\mu^{[n]}_{\ms^*_\infty})$ ({\color{Cerulean} {\bf - - -}}) and $D_{\mu,s}^{1/s}(\mu^{[n]}_{\ms^*})/D_{\mu,s}^{1/s}(\mu^{[n]}_{\ms^*_\infty})$ ({\color{Dandelion} {\bf ---}}) as functions of $d$ for $s=2$ and $n=10\,000$.}
\label{F:normal2}
\end{figure}



In view of the right panel of Figure~\ref{F:normal2}, and following the analysis in Section~\ref{S:ball}, we investigate the following question: for a fixed dimension $d$, up to which sample size $n$ does the random quantiser $\PP_{a^*}^{[n]}$ outperform $\mu_{\ms^*}^{[n]}$? Defining
\be\label{nstar}
n^*(d,s) = \max\{n\in\NN: D_{\mu,s}(\PP_{a^*}^{[n]}) < D_{\mu,s}(\mu^{[n]}_{\ms^*}) \} \,,
\ee
in the right panel of Figure~\ref{F:nstar-ball_s2-4-10}
we present $n^*(d,s)$ (in log scale) as a function of $d=3,\ldots,20$ for three values of $s$. The straight lines fitted to these logarithmic plots show almost perfect coincidence between $n^*(d,s)$ and the exponential increase $n_0\, \ml^d$, with numerical values of $n_0$ and $\ml$ depending weakly on $s$: $n_0 \simeq 2.81$, 2.68, 2.21 and $\ml \simeq 2.08$, 2.07, 2.05 for $s=2$, 4 and 10, respectively.

\section{Extreme-value approximations}\label{S:asymptotic}

In this section, extreme-value theory is used to derive approximations of the $(\mu,s)$-distortion $D_{\mu,s}(\PP_a^{[n]})$, where $\PP_a$ is uniform on $\SS_{d-1}(a)$ (Section~\ref{S:Rn-uniform-on-SS}).
In Section~\ref{S:n-growing-with-d}, three asymptotic regimes are identified that govern the limiting behaviour of $D_{\mu,s}(\PP_a^{[n]})$ when $n$ increases with $d$.
The case where $\mu$ is uniform on $\SS_{d-1}(1)$ is considered first (Section~\ref{S:extreme-value-sphere}): we show that to each asymptotic regime is associated a limiting value of the optimal radius $a^*$ that does not depend on $s$. In Section~\ref{S:extreme-value-general-mu}, these results are extended to general spherically symmetric measures $\mu$ satisfying a norm-concentration property, with the uniform distribution in $\SB_d(1)$ and the multivariate normal distribution as illustrative examples.


\subsection{Uniform quantisers on a sphere}\label{S:Rn-uniform-on-SS}

Consider random quantisers $\Rb_n$ with distribution $\Px=\PP_a^{[n]}$ where $\PP_a$ is uniform on $\SS_{d-1}(a)$. For large $n$, the distribution of $\zeta(n,d)$ in \eqref{d2URn} can be approximated using standard results from extreme-value theory. Here we assume $d\geq 3$ to avoid singularity in beta distributions.

\begin{lemma}\label{L:main-extreme-value2}
Assume that $d\geq 3$ and $\Px=\PP_a^{[n]}$. For any fixed $\ub\in\RR^d$, we have
\be\label{d2URnb}
\frac{1}{4\,a\,\|\ub\|\,\kappa_{n,d}} \left[d^2(\ub,\Rb_n)-(\|\ub\|-a)^2\right] \rad \xi_d \,, \ n\to\infty\,,
\ee
where $\xi_d$ is a random variable with Weibull c.d.f.\ $F_d(t)=1-\exp(-t^\delta)$, $t\geq 0$, and moments $\Ex\{\zeta_d^k\}=\Gamma(1+k/\delta)$ with $\delta=(d-1)/2$, and where $\kappa_{n,d}$ is the $1/n$ quantile of $\beta_{\delta,\delta}$, defined by $I_{\kappa_{n,d}}(\delta,\delta)=1/n$.
\end{lemma}

\begin{proof}
From standard extreme-value theory (see, e.g., \cite[p.~59]{KotzN2000}), the random variable
$\zeta(n,d)/\kappa_{n,d}$ in \eqref{d2URn} converges in distribution to the random variable $\xi_d$.
\end{proof}

\begin{remark}\label{R:moments-duRn}
Lemma~\ref{L:main-extreme-value2} yields the following approximations for the $s$-moment of $d(\ub,\Rb_n)$ for large $n$ and any given $\ub$:
\bea
M_s(d,n,a,r) = \int_{\zeta\geq 0} \left[(r-a)^2 + 4\,a\,r\,\kappa_{n,d}\,\zeta\right]^{s/2}\, \dd F_d(\zeta) \,,
\eea
with $r=\|\ub\|$.
Explicit expressions are easily obtained for even $s$. Assuming that $\|\ub\|=1$, this gives in particular the approximations
\bea
\var_2 &=& 16\,a^2\,\kappa_{n,d}^2\,\var\{\xi_d\} = 16\,a^2\,\kappa_{n,d}^2 \left[\Gamma(1+2/\delta)-\Gamma^2(1+1/\delta)\right] \,, \\
\var_4 &=& 64\,a^2\,\kappa_{n,d}^2\, \left\{(1-a)^4\,\var\{\xi_d\} + a^2\,\kappa_{n,d}^2\,\var\{\xi_d^2\} \right. \\
&& \left. + (1-a)^2\,a\,\kappa_{n,d}\,\left[\Gamma(1+3/\delta)-\Gamma(1+1/\delta)\Gamma(1+2/\delta)\right] \right\} \,.
\eea
for $\var_s=\var\{d^s(\ub,\Rb_n)\}$ when $s=2$ and $s=4$, respectively.
\fin
\end{remark}

The random variables $\zeta(n,d)$ in \eqref{d2URn} depend on $\ub$ and $\Rb_n$, but \eqref{d2URnb} shows that the dependence in $\Rb_n$ vanishes asymptotically, so that $\xi_d$ only depends on $\ub/\|\ub\|$. We hence obtain the following extreme-value approximation of $D_{\mu,s}(\PP_a^{[n]})$ for $\mu$ spherically symmetric.

\begin{definition}\label{D:main-extreme-value}
Assume that $d\geq 3$ and that $\mu$ is spherically symmetric. For any $\Rb_n\simd \PP_a^{[n]}$, the extreme-value approximation of the $(\mu,s)$-distortion $E_{\mu,s}^s(\Rb_n)$ is defined by
\be\label{Ds-approx}
\widehat E_{\mu,s}^s(n;a) = \int_{r\geq 0} \int_{\zeta\geq 0} \left[(r-a)^2 + 4\,a\,r\,\kappa_{n,d}\,\zeta\right]^{s/2}\, \dd F_d(\zeta) \, \dd \Psi(r) \,,
\ee
where $\Psi(\cdot)$ is the c.d.f.\ of $\|U\|$ with $U\simd \mu$, $F_d(t)=1-\exp(-t^\delta)$ with $\delta=(d-1)/2$, and $\kappa_{n,d}$ is the $1/n$ quantile of $\beta_{\delta,\delta}$ defined by $I_{\kappa_{n,d}}(\delta,\delta)=1/n$. The extreme-value approximation $\widehat D_{\mu,s}(\PP_a^{[n]})$ of $D_{\mu,s}(\PP_a^{[n]})$ is defined by $\widehat D_{\mu,s}(\PP_a^{[n]})=\widehat E_{\mu,s}^s(n;a)$.
\end{definition}

The approximation \eqref{Ds-approx} is derived as follows. Since $\mu$ is spherically symmetric, we have $U\simd r\cdot U^{(d)}$ where $U^{(d)}$ is uniformly distributed on $\SS_{d-1}(1)$ and $r$ has the c.d.f.\ $\Psi(\cdot)$. Since $\xi_d$ in \eqref{d2URnb} only depends on $U^{(d)}$ and has the c.d.f.\ $F_d(\cdot)$, the substitution of $(r-a)^2 + 4\,a\,r\,\kappa_{n,d}\,\zeta(U^{(d)})$ for $d^2(U,\Rb_n)$ in $E_{\mu,s}^s(\Rb_n)= \Ex_\mu \{d^s(U,\Rb_n)\}$ gives \eqref{Ds-approx}. In contrast with Remark~\ref{R:moments-duRn} where $\ub$ is fixed, averaging over $\ub$ removes the variability (but the means are identical; that is, $\widehat E_{\mu,s}^s(n;a)=\Ex_{\PP_a^{[n]}}\{d^s(\ub,\Rb_n)\}$).

For $s$ even, $\widehat E_{\mu,s}^s(n;a)$ can be expressed explicitly in terms of moments of $\|U\|$, $M_{\Psi,k}=\Ex_\mu\{\|U\|^k\}=\int_{r\geq 0}  r^k\,\dd \Psi(r)$, and moments $\Ex\{\zeta^k\}=\Gamma(1+k/\delta)$ of the Weibull distribution. In particular, for $s=2$ and 4 we get the following extreme-value approximations for the $(\mu,s)$-distortions:
\be
\widehat E_{\mu,2}^2(n;a) &=& M_{\Psi,2} - 2\,a\,M_{\Psi,1}+a^2 + 4\,a\,\kappa_{n,d}\,M_{\Psi,1}\,\Gamma(1+1/\delta)  \,, \label{E2^2-general}\\
\widehat E_{\mu,4}^4(n;a) &=& M_{\Psi,4} - 4\,a\,M_{\Psi,3}+6\,a^2\,M_{\Psi,2}-4\,a^3\,M_{\Psi,1}+a^4 \nonumber \\
&& + \, 8\,a\,\kappa_{n,d}\,\Gamma(1+1/\delta)
\left( M_{\Psi,3} - 2\,a\,M_{\Psi,2}+a^2\,M_{\Psi,1} \right) \nonumber \\
&& + \, 16\,a^2\,\kappa_{n,d}^2\,M_{\Psi,2}\,\Gamma(1+2/\delta) \,. \label{E4^2-general}
\ee

Denote by $\widehat {a^*}=\widehat {a^*}(d,n,s)$ the value of $a$ that minimises $\widehat E_{\mu,s}(n;a)$, which forms an extreme-value approximation of the optimal $a^*$ minimising $D_{\mu,s}(\PP_a^{[n]})$. In the next corollary, we give the expressions of $\widehat {a^*}(d,n,2)$ and $\widehat E_{\mu,2}^2(n;\widehat {a^*})$; the later providing an extreme-value approximation for $\min_a D_{\mu,2}(\PP_a^{[n]})$.

\begin{corollary}\label{Coro:astar-s2-Pa}
For $d\geq 3$, $\mu$ spherically symmetric and $\Rb_n\simd \PP_a^{[n]}$, we have
\be
\widehat {a^*} = \widehat {a^*}(d,n,2) &=& M_{\Psi,1}\,\left[ 1 - 2\,\kappa_{n,d}\,\Gamma(1+1/\delta) \right]\,, \label{hat_a_2} \\
\widehat E_{\mu,2}^2(n;\widehat {a^*}) &=& M_{\Psi,2} - (\widehat {a^*})^2 \,. \nonumber
\ee
\end{corollary}

For $\mu$ uniform on $\SB_d(1)$, Figure~\ref{F:bstar_ALLastar_ball_s2_D} indicates that $\widehat {a^*}(d,n,2)$ and $a^*(d,n,2)$ are practically indistinguishable for $n=1\,000$ and $n=100\,000$ with $d=10,\ldots,50$.
For the situation considered in Section~\ref{S:normal} where $U\simd\SN(\0b_d,\Ib_d/d)$, $\widehat {a^*}(d,n,2)$ is also very close to $a^*(d,n,2)$ when $n$ is large enough; its evolution as a function of $d$ for $n=10\,000$ is visually indistinguishable from that of $a^*(d,n,s)$ on the left panel of Figure~\ref{F:normal2}.

The expressions for $\widehat {a^*} =\widehat {a^*}(d,n,4)$ and $\widehat E_{\mu,4}^4(n;\widehat {a^*})$ can also be written in analytic form (by finding the roots of a third-degree polynomial), but the expressions are cumbersome.

\subsection{Three asymptotic regimes for $n$ growing with $d$}\label{S:n-growing-with-d}

In view of Lemma~\ref{L:main-extreme-value2}, the asymptotic behaviour of $\kappa_{n,d}$ is the key element in understanding that of $d^2(\ub,\Rb_n)$. For $d=3$, $\beta_{1,1}$ is uniform on $[0,1]$ and we simply have $\kappa_{n,3}=1/n$. For $d>3$ there is no explicit formula for $\kappa_{n,d}$, but it can easily be computed numerically by solving $I_{\kappa_{n,d}}(\delta,\delta)=1/n$. For fixed $t\leq 1/2$, $I_t(\delta,\delta)$ decreases as $\delta$ increases. Therefore, $\kappa_{n,d}$ is an increasing function of $d$ for any fixed $n\geq 2$ (it is also a decreasing function of $n$ for any fixed $d$). When $n$ tends to infinity with $d$ we have the following properties.

\begin{proposition} \label{th:kappa_n}
For fixed $d\geq 3$, the $1/n$ quantile $\kappa_{n,d}$ of  $\beta_{\delta,\delta}$ with $\delta=(d-1)/2$ tends to zero as $n$ tends to infinity. When $d \to \infty$ and $n=n(d)$ grows with $d$, we have the following three cases:\\
\noindent{\textit(i)} if $n^{1/d}\to\infty$, then $\lim_{d\to\infty} \kappa_{n,d} = 0$;\\
\noindent{\textit(ii)} if $n=C\ml^d\,(1+o(1))$ with $\ml>1$ and $C>0$, then
\be \label{asymptotic-kappa} \lim_{d\to\infty} \kappa_{n,d} = \frac12 \left[1-\sqrt{1-1/\ml^2}\right]\, ;
\ee
\noindent{\textit(iii)} 
if $(\log n)/d \to 0$, then $\lim_{d\to\infty} \kappa_{n,d} = \frac12$.
\end{proposition}

\begin{proof}
Since $x^{\delta-1}(1-t)^\delta\leq x^{\delta-1}(1-t)^{\delta-1}\leq x^{\delta-1}(1-x)^{\delta-1}$ for $\delta\geq 1$ and all $t\in[0,1)$ and $x\in[0,t]$, we have
\be\label{It-ineq1}
\frac{t^\delta (1-t)^\delta}{\delta B(\delta,\delta)} \leq I_t(\delta,\delta) \,. 
\ee
The derivation of an exploitable upper bound is more delicate.
Assume that $d \geq 3$ and let
$\delta'=\lfloor \delta \rfloor$, the largest integer smaller than or equal to $\delta$.
As $n\geq 2$, we are only interested in values of $t$ less than $1/2$. Since $\delta'\leq \delta$, we have $I_t(\delta,\delta) \leq I_t(\delta',\delta')$, which
can be calculated explicitly by successive integration by parts. Direct calculations yield the following special case of the well-known relation between the c.d.f.\ of the beta and binomial distributions:
\be\label{It1}
I_t(\delta',\delta') &=& \sum_{k=0}^{\delta'-1} \binom{2\delta'-1}{k} t^{2\delta'-1-k}(1-t)^k \,, \\
&=& t^{\delta'}(1-t)^{\delta'-1} \sum_{k=0}^{\delta'-1} \binom{2\delta'-1}{k} \left(\frac{t}{1-t}\right)^{\delta'-1-k} \,.
\ee
Using the property $t/(1-t) \leq 1$ for $t\leq 1/2$, with $\sum_{k=0}^{\delta'-1} \binom{2\delta'-1}{k}=4^{\delta'-1}$, we thus obtain
\be \label{upper_bound}
I_t(\delta,\delta) \leq I_t(\delta',\delta') \leq t^{\delta'}[4\,(1-t)]^{\delta'-1} \leq \frac12\,[4\,t(1-t)]^{\delta'-1}\,, \ t\leq 1/2\,.
\ee

Together with \eqref{It-ineq1}, this implies, for $n\geq 2$ and $d\geq 3$,
\be\label{bounds-kappa}
\frac12 \left[1-\sqrt{1- (2/n)^{1/(\delta'-1)}}\right]  \leq \kappa_{n,d} \leq \frac12 \left[1-\sqrt{(1-4\,c_d\,n^{-1/\delta})_+}\right] \,,
\ee
where $(x)_+=\max\{x,0\}$ and $c_d=[\delta\,B(\delta,\delta)]^{1/\delta} = 1/4+ (\log d)/(4\,d)+\SO(1/d)$, $d\to\infty$ (with, moreover, $(4\,c_d)^\delta<d+2$ for all $d \geq 2$, implying that the right-hand side of \eqref{bounds-kappa} is strictly smaller than  1/2 for $n\geq d+2$).

When $d$ is fixed and $n\to\infty$, or $n^{1/d}\to\infty$ as $d\to\infty$, $n^{-1/d}\to 0$ implying that the right-hand side of \eqref{bounds-kappa} tends to zero. Therefore, $\kappa_{n,d}\to 0$.

When $n=C\ml^d\,(1+o(1))$ with $\ml>1$ and $C>0$ (case (\textit{ii})), by taking the limits on left and right hand sides of \eqref{bounds-kappa}, as $\lim_{d\to\infty} c_d=1/4$, we obtain $\lim_{d\to\infty} \kappa_{n,d} = (1/2)\,\left[1-\sqrt{1-1/\ml^2}\right]$.

When $\log n = o(d)$, (case (\textit{iii})), $n^{-1/d}\to 1$ and the left-hand side of \eqref{bounds-kappa} tends to 1/2.
\end{proof}

The left panel of Figure~\ref{F:sphere-approx0} plots $\kappa_{n,d}$ as a function of $d$ for different $n$: the continuation of each curve (with $n$ fixed) goes to the limiting value 1/2 in view of Proposition~\ref{th:kappa_n}-(\textit{iii}). The bounds in \eqref{bounds-kappa} are asymptotically accurate for large $n$, particularly the upper bound, but lack precision for smaller $n$. For illustration, only the envelope for $n = 10^3$ is shown.

\subsection{$\mu$ uniform on the unit sphere $\SS_{d-1}(1)$} \label{S:extreme-value-sphere}

When $U\simd \mu$ uniform on $\SS_{d-1}(1)$, for any $\Rb_n\simd \PP_a^{[n]}$ the approximations \eqref{E2^2-general} and \eqref{E4^2-general}
become
\bea
\widehat E_{\mu,2}(n;a) \!\!\! &=& \!\!\! \left[ (1-a)^2 + 4\,a\,\kappa_{n,d}\,\Gamma(1+1/\delta) \right]^{1/2} \,, \\
\widehat E_{\mu,4}(n;a) \!\!\! &=& \!\!\! \left[ (1-a)^4 + 16\,a^2\,\kappa_{n,d}^2\,\Gamma(1+2/\delta) + 8\,a(1-a)^2\,\kappa_{n,d}\,\Gamma(1+1/\delta) \right]^{1/4} \,.
\eea
For $s=2$, Corollary~\ref{Coro:astar-s2-Pa} gives $\widehat {a^*}=\widehat {a^*}(d,n,2) = 1 - 2\,\kappa_{n,d}\,\Gamma(1+1/\delta)$ and $\widehat E_{\mu,2}^2(n;\widehat {a^*}) =1 - (\widehat {a^*})^2$.
The right panel of Figure~\ref{F:sphere-approx0} plots $\widehat {a^*}(d,n,2)-a^*(d,n,2)$ as a function of $d$ for different values of $n$: $\widehat {a^*}-a^* >0$ and the accuracy of the approximation of $a^*$ by $\widehat {a^*}$ tends to decrease with $d$ and increase with $n$.

\begin{figure}[ht!]
\begin{center}
\includegraphics[width=.49\linewidth]{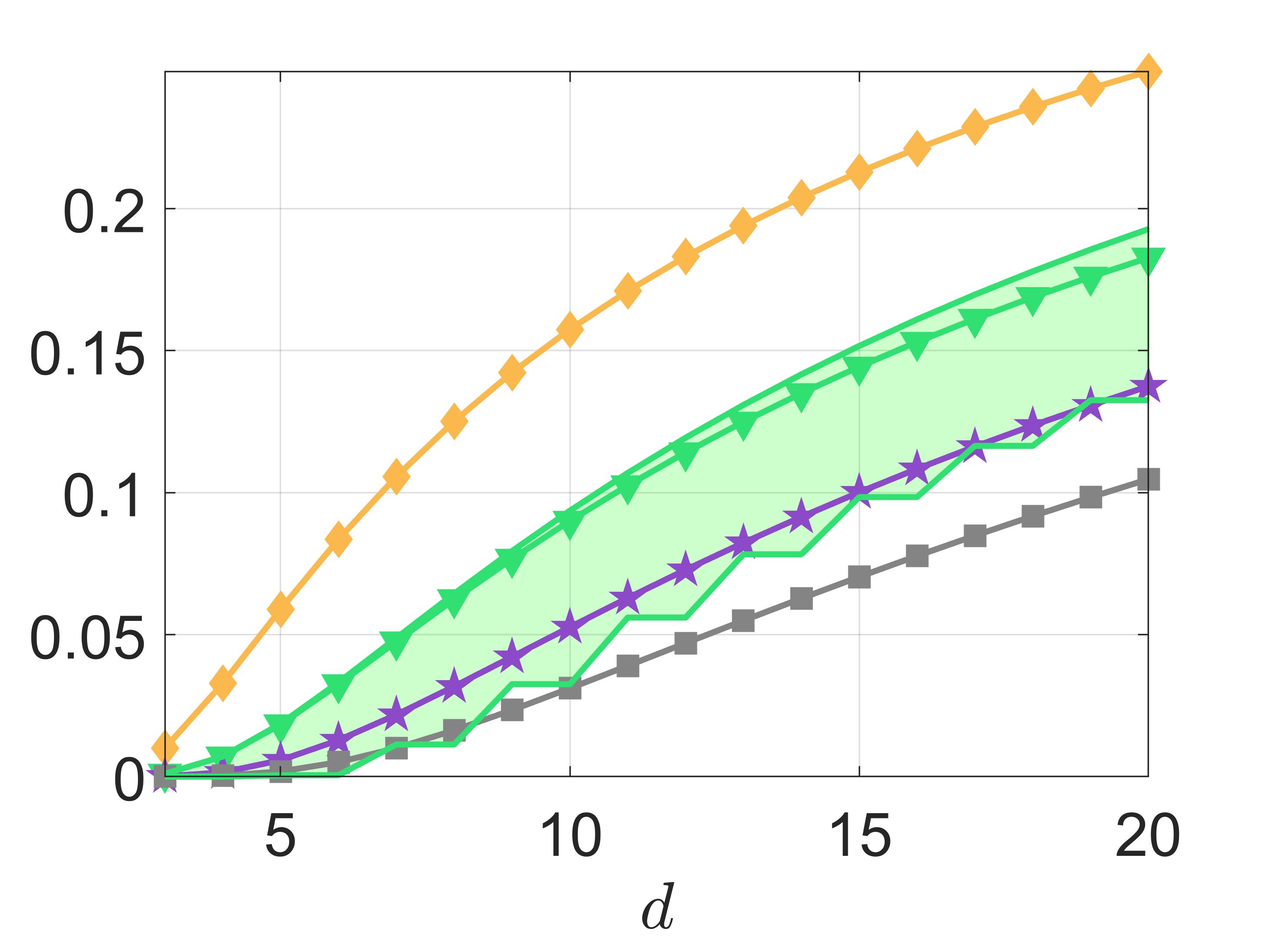}
\includegraphics[width=.49\linewidth]{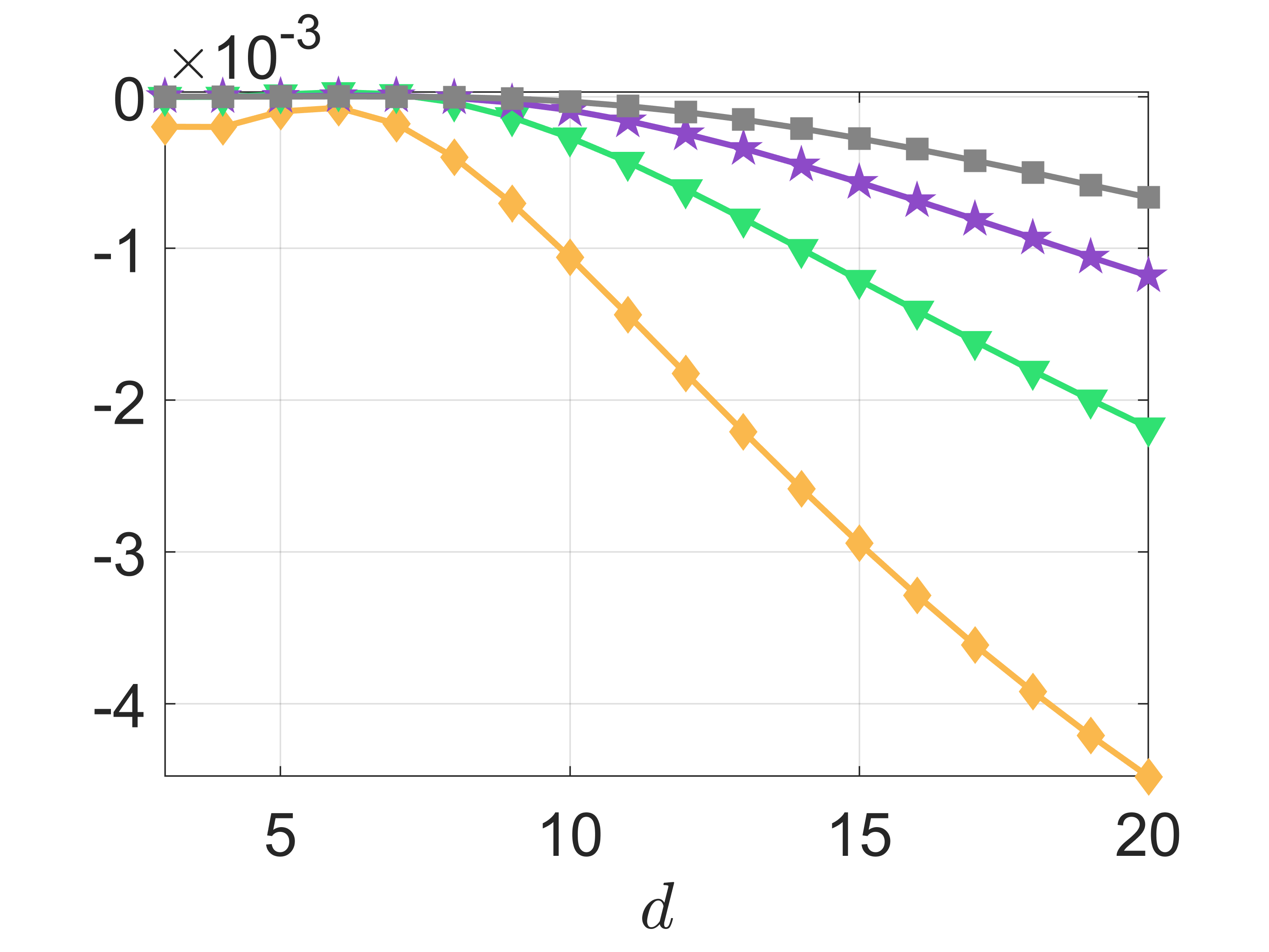}
\end{center}
\caption{\small Values of $\kappa_{n,d}$ (left) and of $\widehat {a^*}(d,n,2)-a^*(d,n,2)$ (right) as functions of $d$ for different values of $n$: $n=10^2$ ({\color{Dandelion} $\blacklozenge$}), $n=10^3$ ({\color{green} $\blacktriangledown$}), $n=10^4$ ({\color{Orchid} $\bigstar$}) and $n=10^5$ ({\color{gray} {\tiny $\blacksquare$}}).
}
\label{F:sphere-approx0}
\end{figure}

The left and right panels of Figure~\ref{F:sphere-approx} show that the two approximations $\widehat D_{\mu,2}(\PP^{[n]}_a)$ and $\widehat D_{\mu,4}(\PP^{[n]}_a)$ are quite accurate when $n$ is large enough for extreme value theory to be applicable. For large $d$, convergence to the extreme-value distribution is slow and $\widehat D_{\mu,s}^{1/s}(\PP^{[n]}_a)$ slightly overestimates the true value $D_{\mu,s}^{1/s}(\PP^{[n]}_a)$.

The left panel of Figure~\ref{F:sphere-approx-2} provides another illustration of the accuracy of the estimation, and presents normalised values $n^{1/d}\,\widehat D_{\mu,s}^{1/s}(\PP^{[n]}_a)$ and $n^{1/d}\,D_{\mu,s}^{1/s}(\PP^{[n]}_a)$ as functions of $n=2^m$, $m=6,7,\ldots,20$, for $s=4$ and $d=20$, both for $a=0.75$ fixed and $a$ optimal, i.e., $a^*(d,n,s)$ (the plots with $\widehat {a^*}(d,n,s)$ substituted for $a^*(d,n,s)$ are visually indistinguishable). The value $\widehat D_{\mu,s}^{1/s}(\PP^{[n]}_a)$ approaches $D_{\mu,s}^{1/s}(\PP^{[n]}_{a^*})$ for $n$ such that $a^*=a^*(d,n,s) \simeq 0.75$. On the right panel of Figure~\ref{F:sphere-approx-2}, for each $n$ considered, 100 values of $n^{1/d}\,E_{\mu,s}(\Rb_n)$ computed numerically\footnote{We used $2^{12}$ i.i.d.\ $U_i\simd \mu$, which provides an accurate enough estimation of $E_{\mu,s}(\Rb_n)$; see Remark~\ref{R:distortion-evaluation}.} for 100 random point sets $\Rb_n\simd \PP^{[n]}_{a^*}$ have been added to the plots of $n^{1/d}\,\widehat D_{\mu,s}^{1/s}(\PP^{[n]}_{a^*})$ and $n^{1/d}\,D_{\mu,s}^{1/s}(\PP^{[n]}_{a^*})$. The simulations are in perfect adequation with the exact mean value $n^{1/d}\,D_{\mu,s}^{1/s}(\PP^{[n]}_{a^*})$ and indicate that the variability of $E_{\mu,s}(\Rb_n)$ across quantisers is very small.

\begin{figure}[ht!]
\begin{center}
\includegraphics[width=.49\linewidth]{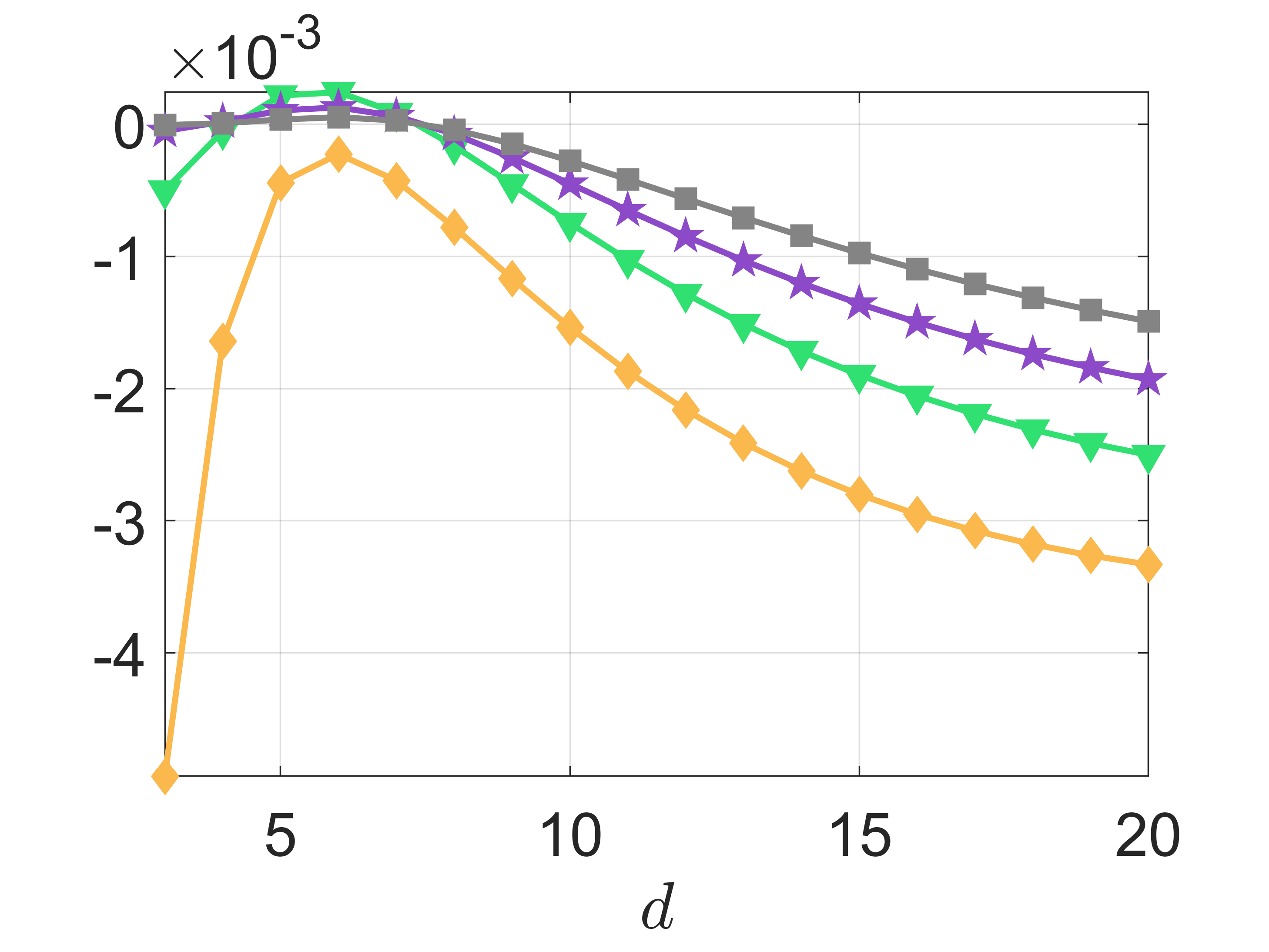}
\includegraphics[width=.49\linewidth]{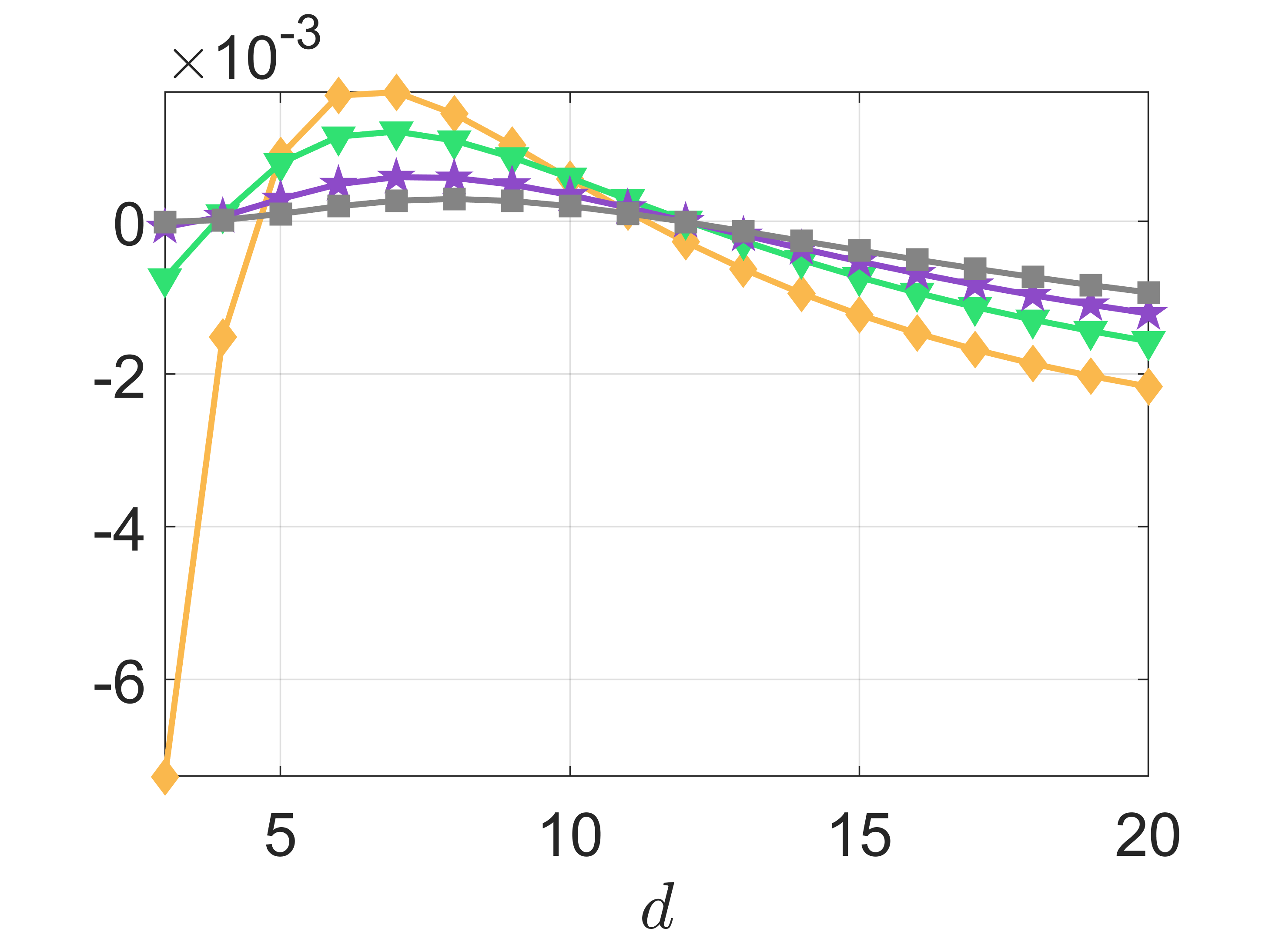} \\
\end{center}
\caption{\small Relative error $1-\widehat D_{\mu,s}^{1/s}(\PP^{[n]}_a)/D_{\mu,s}^{1/s}(\PP^{[n]}_a)$ of the approximation of $D_{\mu,s}^{1/s}(\PP^{[n]}_a)$ (left for $s=2$ and right for $s=4$) based on extreme-value theory, as a function of $d$ for different values of $n$: $n=10^2$ ({\color{Dandelion} $\blacklozenge$}), $n=10^3$ ({\color{green} $\blacktriangledown$}), $n=10^4$ ({\color{Orchid} $\bigstar$}) and $n=10^5$ ({\color{gray} {\tiny $\blacksquare$}}); $a$ is the optimal value $a^*(d,n,s)$.
}
\label{F:sphere-approx}
\end{figure}

\begin{figure}[ht!]
\begin{center}
\includegraphics[width=.49\linewidth]{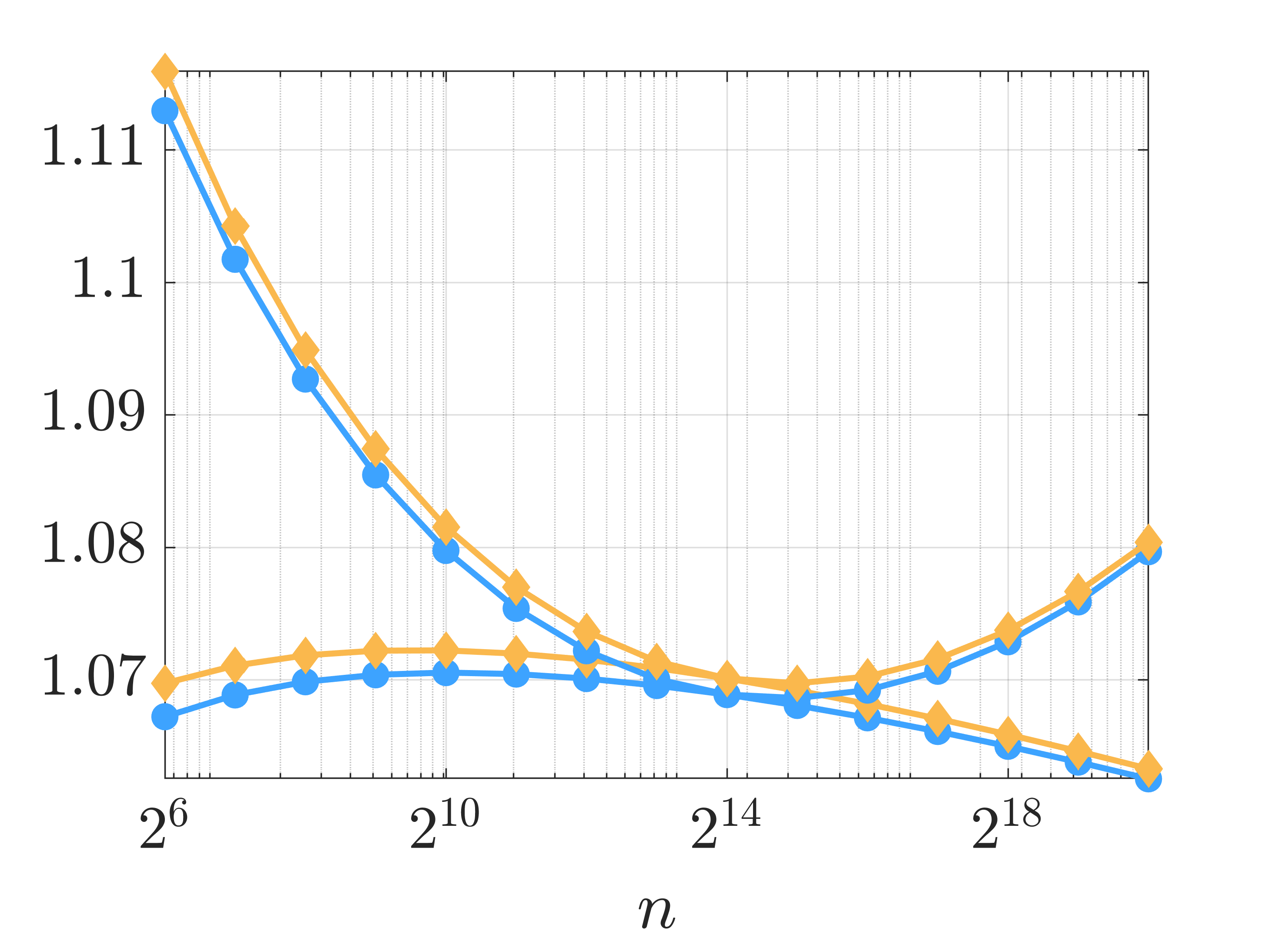}
\includegraphics[width=.49\linewidth]{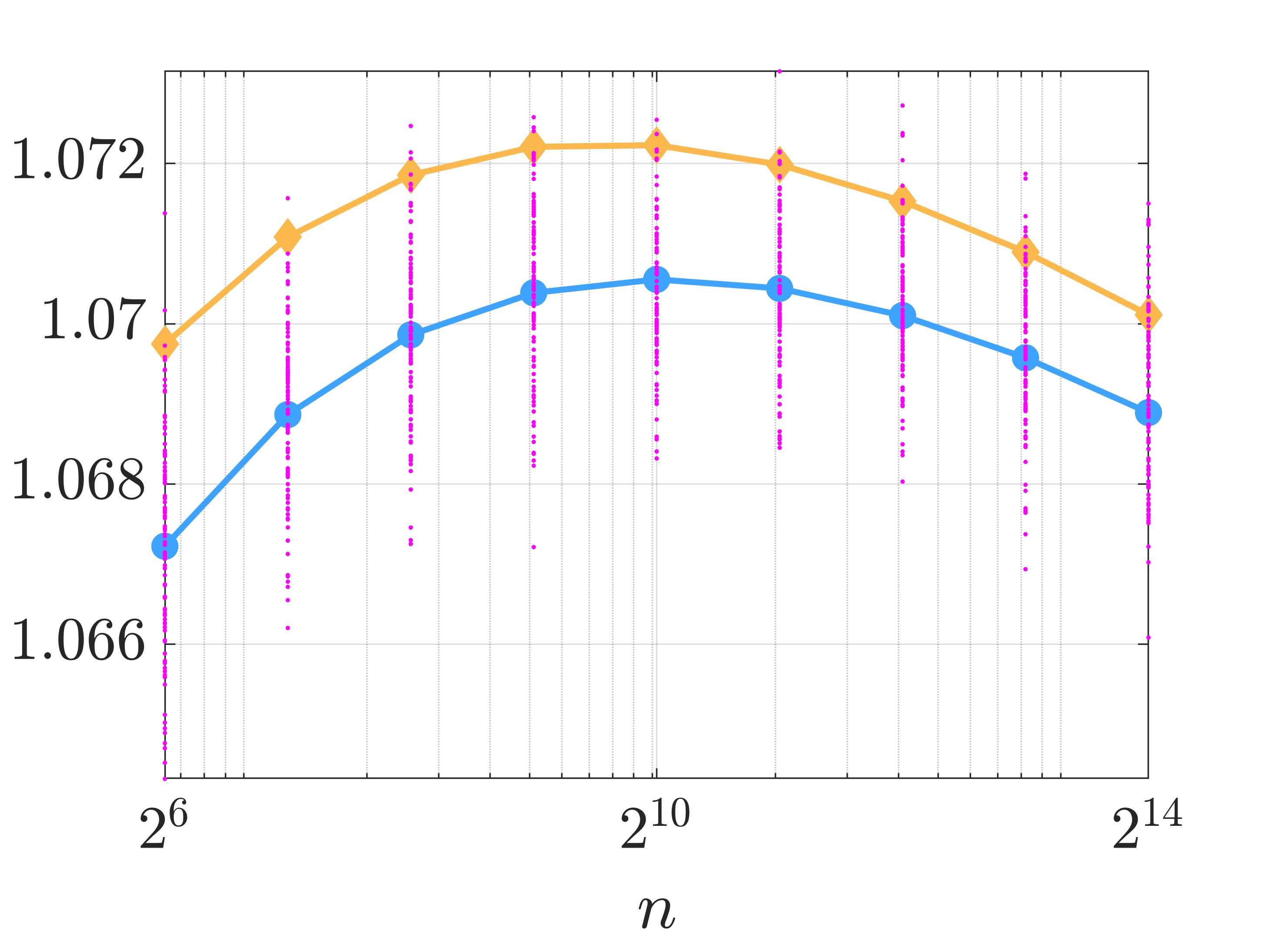}
\end{center}
\caption{\small $n^{1/d}\,\widehat D_{\mu,s}^{1/s}(\PP^{[n]}_a)$ ({\color{Dandelion} $\blacklozenge$}) and $n^{1/d}\,D_{\mu,s}^{1/s}(\PP^{[n]}_a)$ ({\color{Cerulean} $\bullet$}) as functions of $n$ for $s=4$ and $d=20$. Left: $a=0.75$ for top curves and $a=a^*(d,n,s)$ for bottom curves. Right: $a=a^*(d,n,s)$, the values of $n^{1/d}\,E_{\mu,s}(\Rb_n)$ computed numerically for 100 random quantisers $\Rb_n\simd \PP^{[n]}_{a^*}$ are shown as magenta dots.
}
\label{F:sphere-approx-2}
\end{figure}

\vsp
The property below shows that the asymptotic
behaviour of $\widehat {a^*}(d,n,s)$ when $n$ tends to infinity with $d$ does not depend on $s$.

\begin{corollary}\label{Coro:astar-asymptotic}
When $n=n(d)$ grows with $d\to\infty$, we have the following three cases for the limit of $\widehat {a^*}(d,n,s)$:\\
\noindent{\textit(i)} if $n^{1/d}\to\infty$, then $\lim_{d\to\infty} \widehat {a^*}(d,n,s) = 1$;\\
\noindent{\textit(ii)} if $n=\ml^d\,(1+o(1))$ with $\ml>1$, then $\lim_{d\to\infty} \widehat {a^*}(d,n,s) = \widehat {a^*_\ml} = \sqrt{1-1/\ml^2}$;\\
\noindent{\textit(iii)} if $\log n = o(d)$, then $\lim_{d\to\infty} \widehat {a^*}(d,n,s) = 0$.
\end{corollary}

\begin{proof}
Definition \eqref{Ds-approx} gives
\bea
\widehat E_{\mu,s}^s(n;a) = \int_{\zeta\geq 0} \left[(1-a)^2 + 4\,a\,\kappa_{n,d}\,\zeta\right]^{s/2}\, \dd F_d(\zeta) \,,
\eea
where the Weibull distribution tends to be concentrated at $1$ when $d\to\infty$. This implies that, for any given $s$, $\lim_{d\to\infty} \widehat E_{\mu,s}(n;a) - [(1-a)^2 + 4\,a\,\kappa_{n,d}]=0$, where $\kappa_{n,d}$ has the limiting behaviour indicated in Proposition~\ref{th:kappa_n}. Therefore,  $\lim_{d\to\infty} \widehat {a^*}(d,n,s) = \lim_{d\to\infty} 1-2\,\kappa_{n,d}$ does not depend on $s$ and is as indicated above for the three cases considered.
\end{proof}

On the left panel of Figure~\ref{F:sphere} we can observe the tendency $a^*(d,n,s)\to 0$ when $d$ is large but $n$ is not exponentially large compared to $d$; on the top-left panel with $d=10$ with see that $a^*(d,n,1)$ approaches 1 as $n$ increases. The left panel of Figure~\ref{F:sphere3} provides an illustration of  case (\textit{ii}) for $n=2^d$ (i.e., $\ml=2$) with $s=2,4$, for which $\widehat {a^*_\ml} = \sqrt{3}/2 \simeq 0.8660$.


\vsp
From Lemma~\ref{L:main-extreme-value2}, the same arguments as those leading to Definition~\ref{D:main-extreme-value} indicate that, for any $\Rb_n\simd \PP_a^{[n]}$ the extreme-value approximation of the $\mg$-quantile of the  distance c.d.f.\ $F_n(t; \Rb_n,\mu)$ is $\widehat q_\mg = \sqrt{(1-a)^2+4\,a\,\kappa_{n,d}\,t_\mg}$, with $t_\mg=[-\log(1-\mg)]^{1/\delta}$ the $\mg$-quantile of the Weibull distribution with c.d.f.\ $F_d(t)=1-\exp(-t^\delta)$, $\delta=(d-1)/2$.

The minimum of $\widehat q_\mg$ with respect to $a$ is obtained for $\widehat {a_*}(d,n,\mg)=1-2\,\kappa_{n,d}\,t_\mg$. Since $t_\mg\to 1$ as $d\to\infty$, for any fixed $\mg$, $\widehat {a_*}(d,n,\mg)$ has the same limits as those indicated in Corollary~\ref{Coro:astar-asymptotic} for $\widehat {a^*}(d,n,s)$, independently of $\mg$, when $n = n(d) \to \infty.$

\begin{remark}\label{R:quantiles} Note that $\widehat q_\mg = \widehat E_{\mu,2}(n;a)$ for any $a$ when $\mg=\mg(d,2)=F_d[\Gamma(1+1/\delta)]$, showing that for large $n$, minimising the expected $(\mu,2)$-distortion is equivalent to minimising the $\mg$-quantile of the mean distance c.d.f.\ for
\bea
\mg=\mg(d,2)=1- \e1^{-\e1^{-\mgbb}} + \frac{\pi^2}{6\,d}\, \e1^{-\mgbb-\e1^{-\mgbb}} +\SO(d^{-2})  \,, \ d\to\infty\,,
\eea
where $\mgbb$ is Euler's constant and $1- \e1^{-\e1^{-\mgbb}}\simeq 0.429624$. This suggests that $L_2$-quantisation is roughly equivalent to minimising the median of the distance c.d.f., a phenomenon that can also be observed for other sets than the sphere.
\fin
\end{remark}

\subsection{$\mu$ possesses a norm-concentration property} \label{S:extreme-value-general-mu}

We say that the probability measure  $\mu$ on $\RR^d$ possesses a {\em norm-concentration property} when it satisfies
\be\label{norm-concentration}
\mbox{for all } \me>0\,, \quad \mu\left\{\left| \|U\|-A \right| \geq \me\right\} \leq 2\,\exp(-c\, d\, \me^2)
\ee
when $U\sim\mu$, for some $A>0$ and $c$ a constant not depending on $d$. A multivariate distribution such that the components of $U$ are independent and sub-Gaussian with zero mean and strictly positive variance satisfies \eqref{norm-concentration}, see \cite[Th.~3.1.1]{Vershynin2018}. Here we consider distributions that satisfy \eqref{norm-concentration} and are spherically-symmetric, with the uniform distribution in $\SB_b(1)$ and the multivariate normal distribution for which $r=\|U\|$ has the density \eqref{psi-normal} as particular cases. Both satisfy \eqref{norm-concentration} with $A=1$ and, without any loss of generality we consider measures $\mu$ such that $A=1$ in the following.

We continue to consider random quantisers $\Rb_n\simd\PP_a^{[n]}$, focusing on the extreme-value approximation $\widehat D_{\mu,s}(\PP_a^{[n]})$ of the $(\mu,s)$-distortion as defined in Definition~\ref{Ds-approx}. When $n=n(d)$ grows with $d$ and $\mu$ is simply spherically symmetric, the limiting value of $\widehat {a^*}(d,n,s)$ which minimises $\widehat D_{\mu,s}(\PP_a^{[n]})$  generally depends on $s$ in cases (\textit{i}) and (\textit{ii}) of Corollary~\ref{Coro:astar-asymptotic}.  The next proposition shows that, if $\mu$ satisfies \eqref{norm-concentration}, the asymptotic behaviour of $\widehat{a^*}(d,n,s)$ remains unchanged from that described in Corollary~\ref{Coro:astar-asymptotic}.

\begin{proposition} \label{Prop:extreme-value-general-mu}
Assume that $d\geq 3$, that $\Px=\PP_a^{[n]}$, and that $\mu$ is spherically symmetric and satisfies \eqref{norm-concentration}. Then, as $d \to \infty$ with $n = n(d)$ growing accordingly, the limit of $\widehat{a^*}(d,n,s)$ follows the three cases outlined in Corollary~\ref{Coro:astar-asymptotic}.
\end{proposition}

\begin{proof} Consider the random variable $T_{n,d}=(r-a)^2 + 4\,a\,r\,\kappa_{n,d}\,\zeta$ in \eqref{Ds-approx}. The limiting value $\kappa_\infty$ of the quantile $\kappa_{n,d}$ in the three cases considered is given in Proposition~\ref{th:kappa_n}, and $T_{n,d}$ converges in probability to $T_\infty=(1-a)^2 + 4\,a\,\kappa_\infty$ when $d\to\infty$. From \eqref{norm-concentration}, for any given $s>0$ the sequence $\{T_{n,d}^{s/2}\}$ is bounded in $L^p$ for any $p$ and is thus uniformly integrable. Therefore, $\widehat D_{\mu,s}(\PP_a^{[n]})\to T_\infty^{s/2}$, which is minimum for the values of $a$ indicated in Corollary~\ref{Coro:astar-asymptotic}.
\end{proof}

To illustrate the above analysis, we consider again two specific cases for $\mu$: uniform distribution in the unit ball $\SB_d(1)$, in the small $n$ regime with fixed $n$ and increasing $d$; multivariate normal distribution $\SN(\0b_d, \Ib_d/d)$, in the exponential regime with $n = 2^d$.

\subsubsection{$\mu$ is uniform in the unit ball $\SB_d(1)$} \label{S:extreme-value-ball}


In the non-asymptotic regime, $\widehat{a^*}(d,n,2)$ is given by \eqref{hat_a_2} for $s=2$; for $s=4$, $\widehat{a^*}(d,n,4)$ is a root of the third-degree polynomial corresponding to the derivative of \eqref{E4^2-general}.

Figure~\ref{F:bstar_ALLastar_ball_s2_D} displays the following quantities as functions of $d$ for two values of $n$: $a^*(d,n,2)$, its approximation $\widehat{a^*}(d,n,2)$,
and $b^*(d,n,2)$, which minimises $D_{\mu,2}(\PP_{0,b}^{[n]})$ (see Section~\ref{S:ball}). The comparative behaviour of $a^*(d,n,2)$ and $b^*(d,n,2)$ has already been discussed in Section~\ref{S:ball}; the plots of $a^*(d,n,2)$ and $\widehat{a^*}(d,n,2)$ are practically indistinguishable.

Figure~\ref{F:ALLeff_ball_s2_D} displays the efficiencies $D_{\mu,s}^{1/s}(\PP_a^{[n]})/D_{\mu,s}^{1/s}(\PP_{0,b^*}^{[n]})$ as functions of $d$ for $a=a^*(d,n,2)$ and $a=\widehat{a^*}(d,n,2)$. These efficiencies are quasi indistinguishable, and $\PP_{a^*}^{[n]}$ and $\PP_{\widehat{a^*}}^{[n]}$ both outperform $\PP_{0,b^*}^{[n]}$ for the values of $n$ and $d$ considered. Additionally, $\PP_{0,b^*}^{[n]}$ significantly outperforms $\PP_{0,1}^{[n]}$; see Figure~\ref{F:ball} in Section~\ref{S:ball}.

\begin{figure}[ht!]
\begin{center}
\includegraphics[width=.49\linewidth]{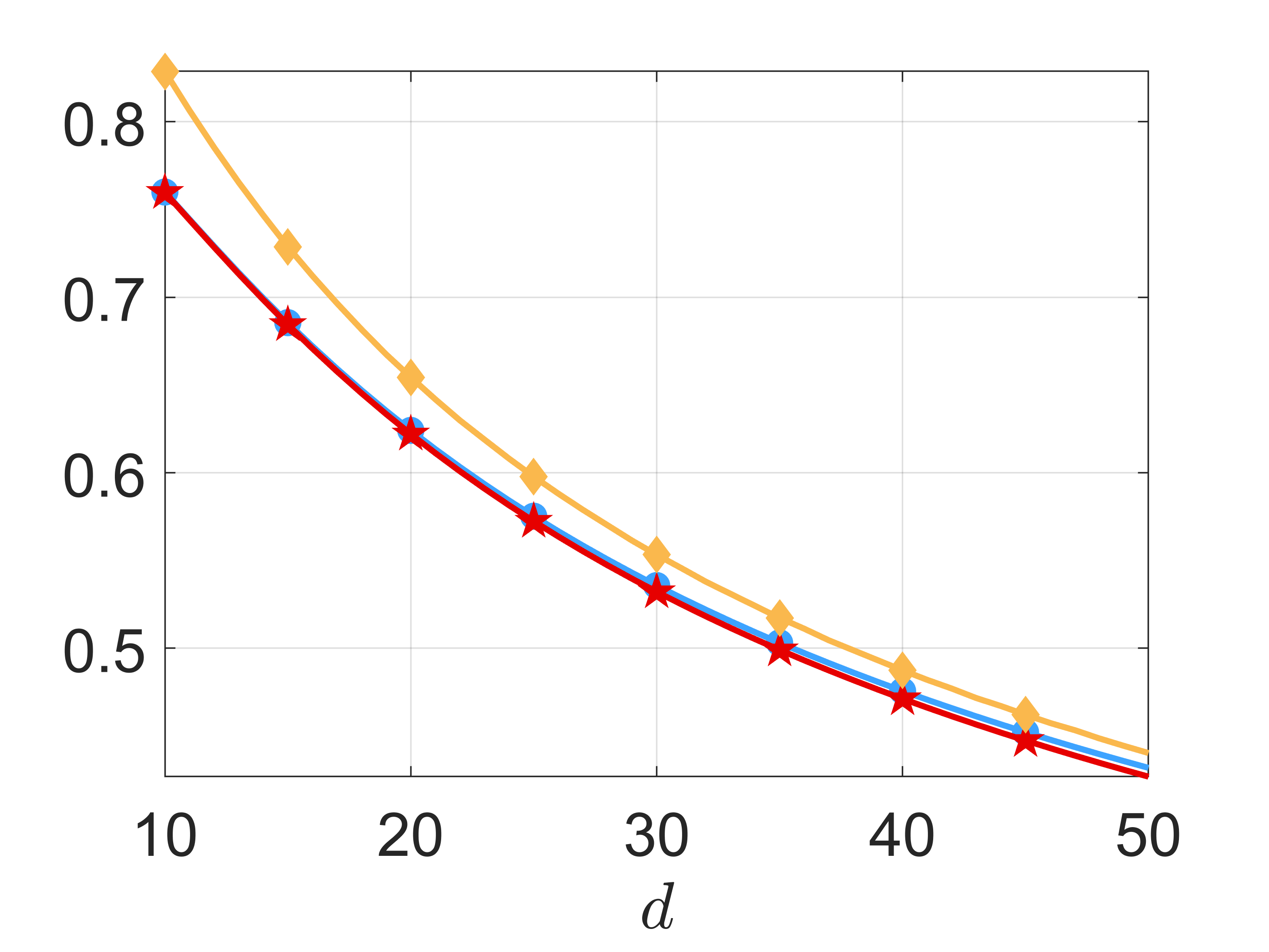}
\includegraphics[width=.49\linewidth]{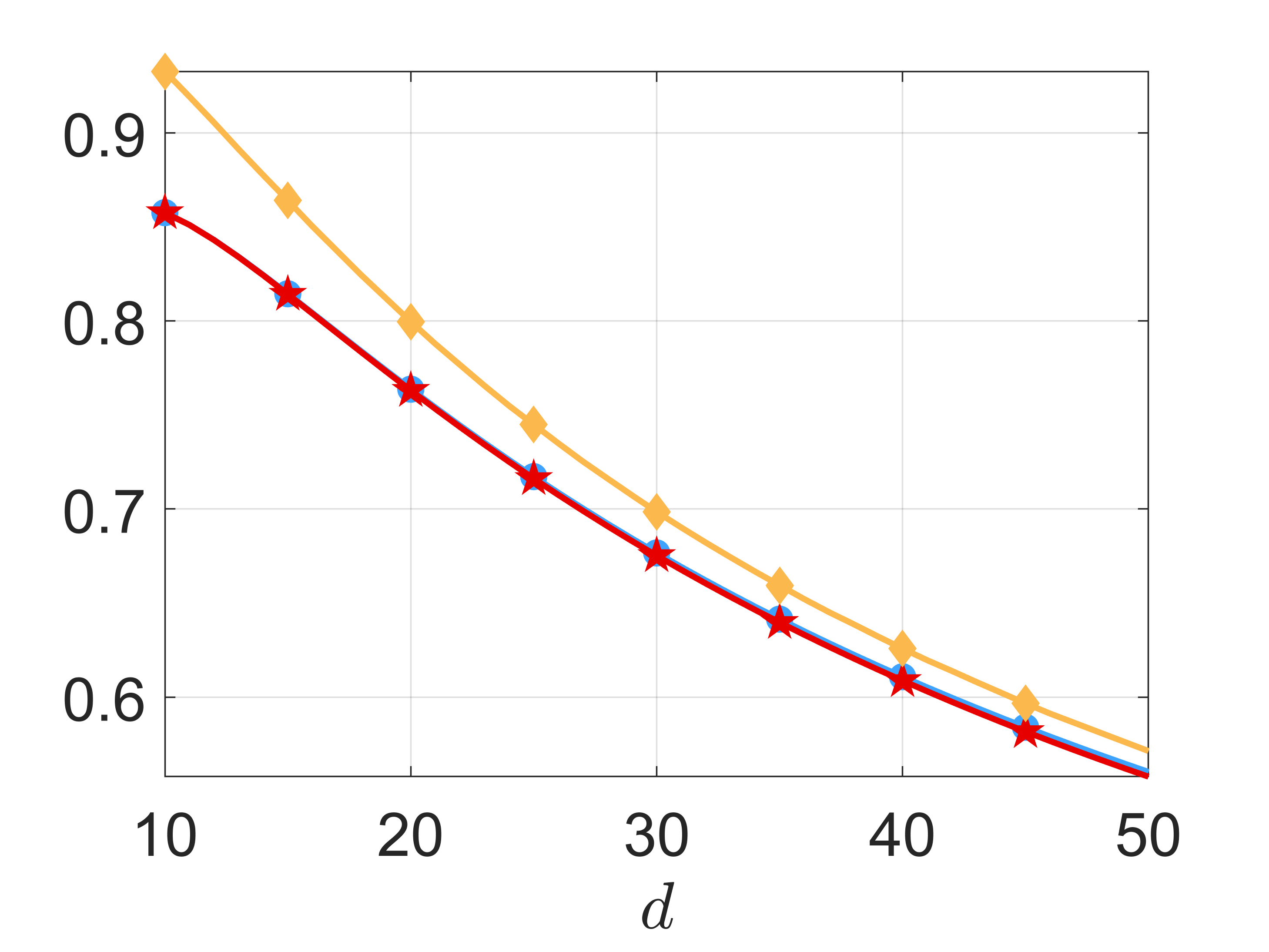}
\end{center}
\caption{\small $a^*(d,n,2)$ ({\color{Cerulean} $\bullet$}), $\widehat{a^*}(d,n,2)$ ({\color{red} $\bigstar$})
and $b^*(d,n,2)$ ({\color{Dandelion} $\blacklozenge$}) as functions of $d$ for $n=1\,000$ (left) and $n=100\,000$ (right).}
\label{F:bstar_ALLastar_ball_s2_D}
\end{figure}

\begin{figure}[ht!]
\begin{center}
\includegraphics[width=.49\linewidth]{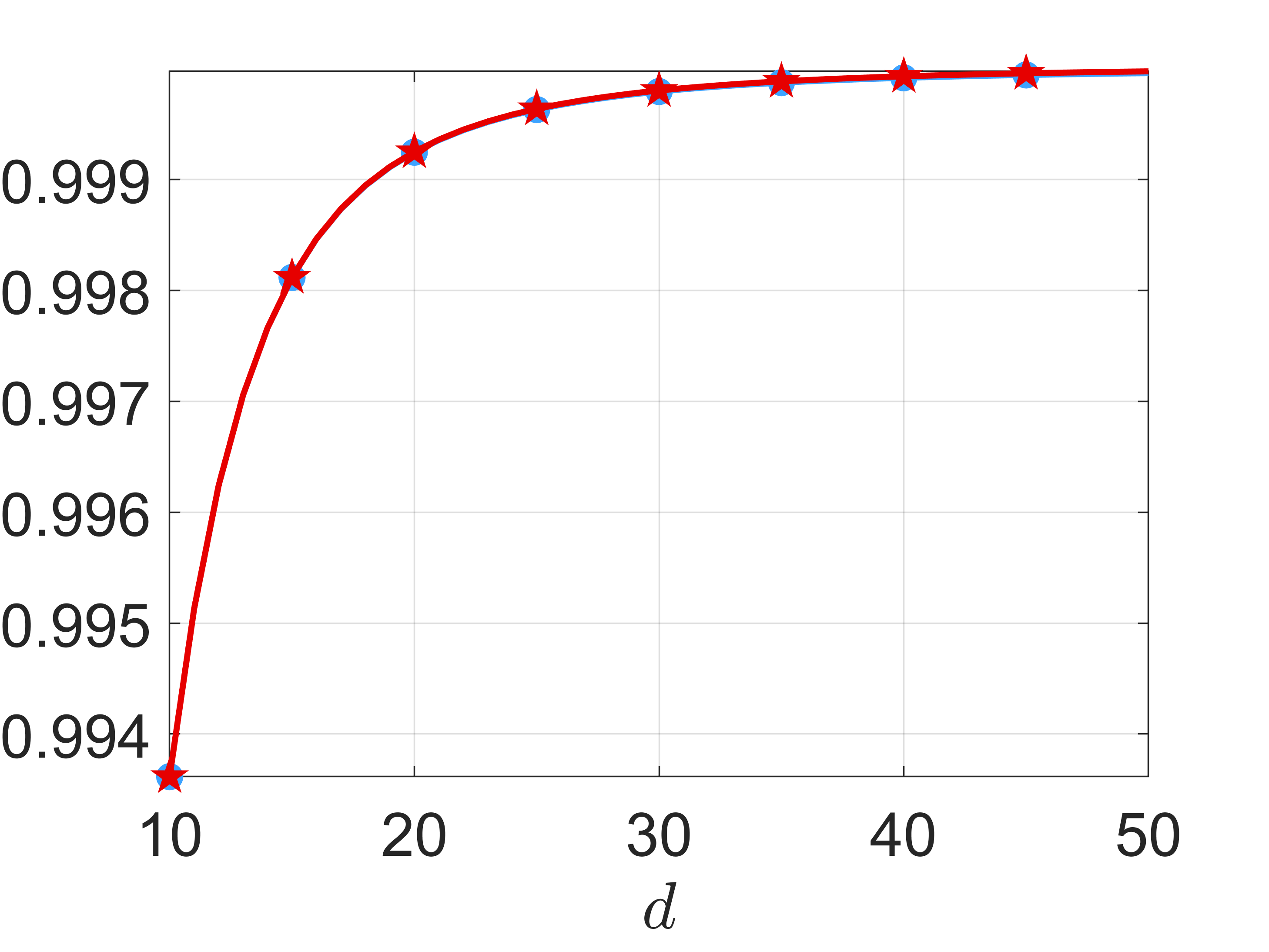}
\includegraphics[width=.49\linewidth]{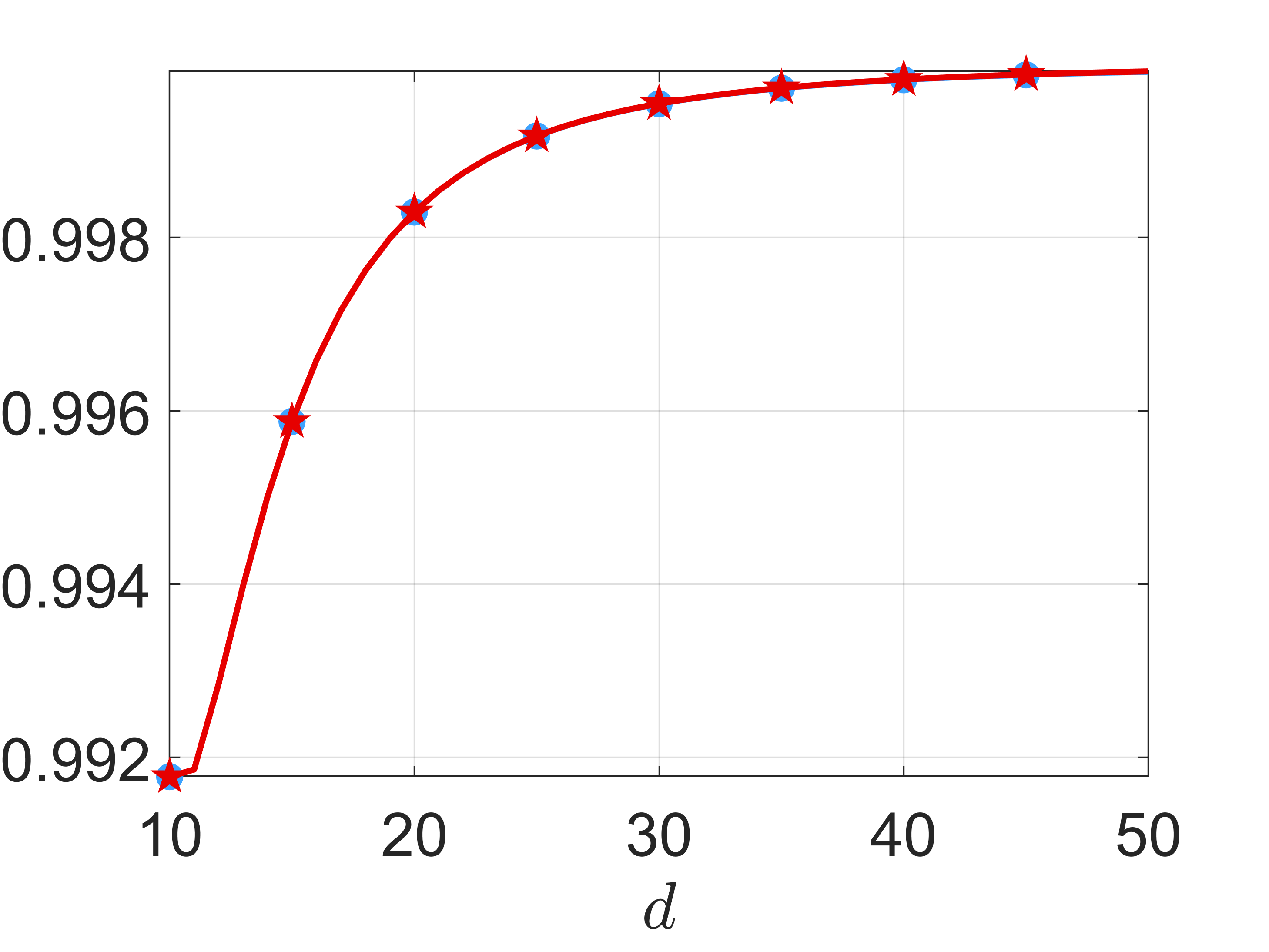}
\end{center}
\caption{\small Efficiencies $D_{\mu,s}^{1/s}(\PP_a^{[n]})/D_{\mu,s}^{1/s}(\PP_{0,b^*}^{[n]})$ as functions of $d$: $a=a^*(d,n,2)$ ({\color{Cerulean} $\bullet$}), $a=\widehat{a^*}(d,n,2)$ ({\color{red} $\bigstar$});
$n=1\,000$ (left) and $n=100\,000$ (right).}
\label{F:ALLeff_ball_s2_D}
\end{figure}



\subsubsection{$\mu$ is normal $\SN(\0b_d,\Ib_d/d)$} \label{S:extreme-value-normal}

We noticed in Section~\ref{S:normal} that $D_{\mu,s}(\PP_{a^*}^{[n]}) < D_{\mu,s}(\mu^{[n]}_{\ms^*})$ for $n \lesssim n_0\, \ml_0^d$ for some $n_0$ and $\ml_0$, where $a^*=a^*(d,n,s)$ and $\ms^*$ respectively minimise $D_{\mu,s}(\PP_a^{[n]})$ and $D_{\mu,s}(\mu^{[n]}_{\ms})$. Here we investigate the behaviour of $D_{\mu,s}(\PP_{a}^{[n]})$ for other choices of $a$ in a similar asymptotic regime where $n$ grows exponentially fast with $d$.

Given that $M_{\psi,1}=\Ex\{\|U\|\} \to 1$ and $\var\{\|U\|\} \to 0$ as $d\to\infty$ when $U \sim \mu$, we analyse the following cases for $a$.
(\textit{i}) Uniform distribution on $\SS_{d-1}(1)$:
we consider $a = \widetilde{a^*}(d,n,s)$ which minimises $D_{\widetilde\mu,s}^{1/s}(\PP_a^{[n]})$ for $\widetilde\mu$ uniform on $\SS_{d-1}(1)$, as discussed in Section~\ref{S:sphere}.
(\textit{ii}) Scaled uniform distribution: we also consider $a = M_{\psi,1}\widetilde{a^*}(d,n,s)$, obtained when $\widetilde\mu$ is uniform on $\SS_{d-1}(M_{\psi,1})$.
As we consider the asymptotic regime $n = 2^d$, we additionally examine (\textit{iii}) the limiting value $\widehat{a^*_\ml} = \sqrt{3}/2$ from Corollary~\ref{Coro:astar-asymptotic}-(\textit{ii});
(\textit{iv}) the value $\widehat{a^*}(d,n,s)$, which minimises $\widehat{E}_{\mu,s}^s(n;a)$ given by \eqref{Ds-approx}, where the quantile $\kappa_{n,d}$ is replaced by its asymptotic value $(1/2)(1 - \sqrt{3}/2)$.

 Figure~\ref{F:normal_n2d_d3-20_s4} (left) displays $\ms^*(d,n,s)$, $a^*(d,n,s)$, $\widetilde{a^*}(d,n,s)$ and $M_{\psi,1}\widetilde{a^*}(d,n,s)$ as functions of $d$ for $s=4$, together with $\widehat{a^*_\ml}$ (indicated by an horizontal line) and
$\widehat{a^*}(d,n,s)$.
The right panel in the same figure shows the efficiency of the optimal normal distribution $\mu_{\ms^*}$ compared to $\PP_a$ uniform on $\SS_{d-1}(a)$, that is, $D_{\mu,s}^{1/s}(\PP_a^{[n]})/D_{\mu,s}^{1/s}(\mu_{\ms^*}^{[n]})$, for the five choices of $a$ considered: $\PP_{a}$ performs better than $\mu_{\ms^*}$ (for $d\geq 4$ when $a=M_{\psi,1}\widetilde{a^*}(d,n,s)$). Note that the efficiencies for $\widehat{a^*_\ml}$ (purple triangles) and $a^*(d,n,s)$ (blue dots) almost coincide for $d\geq 5$.

\begin{figure}[ht!]
\begin{center}
\includegraphics[width=.49\linewidth]{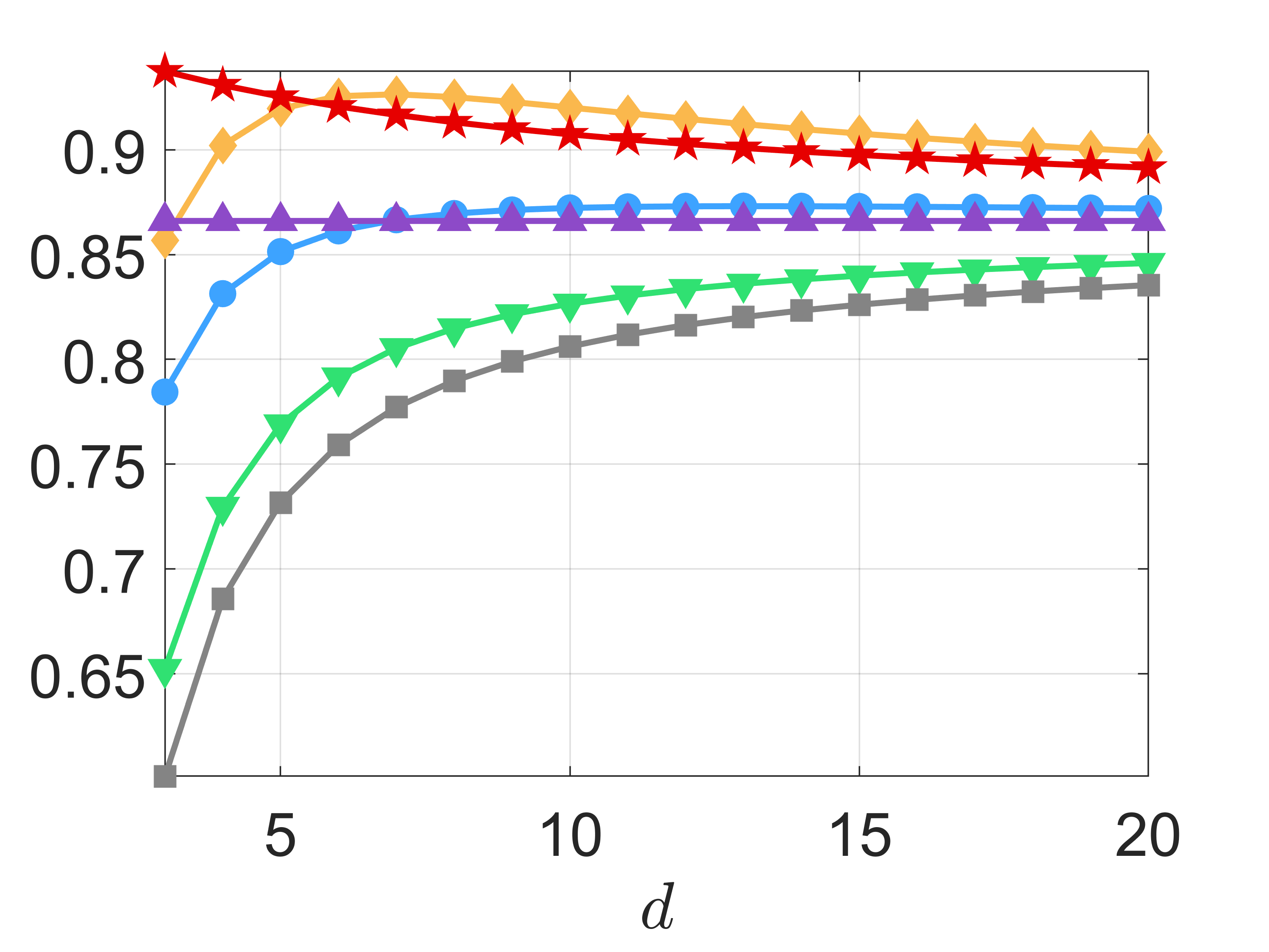}
\includegraphics[width=.49\linewidth]{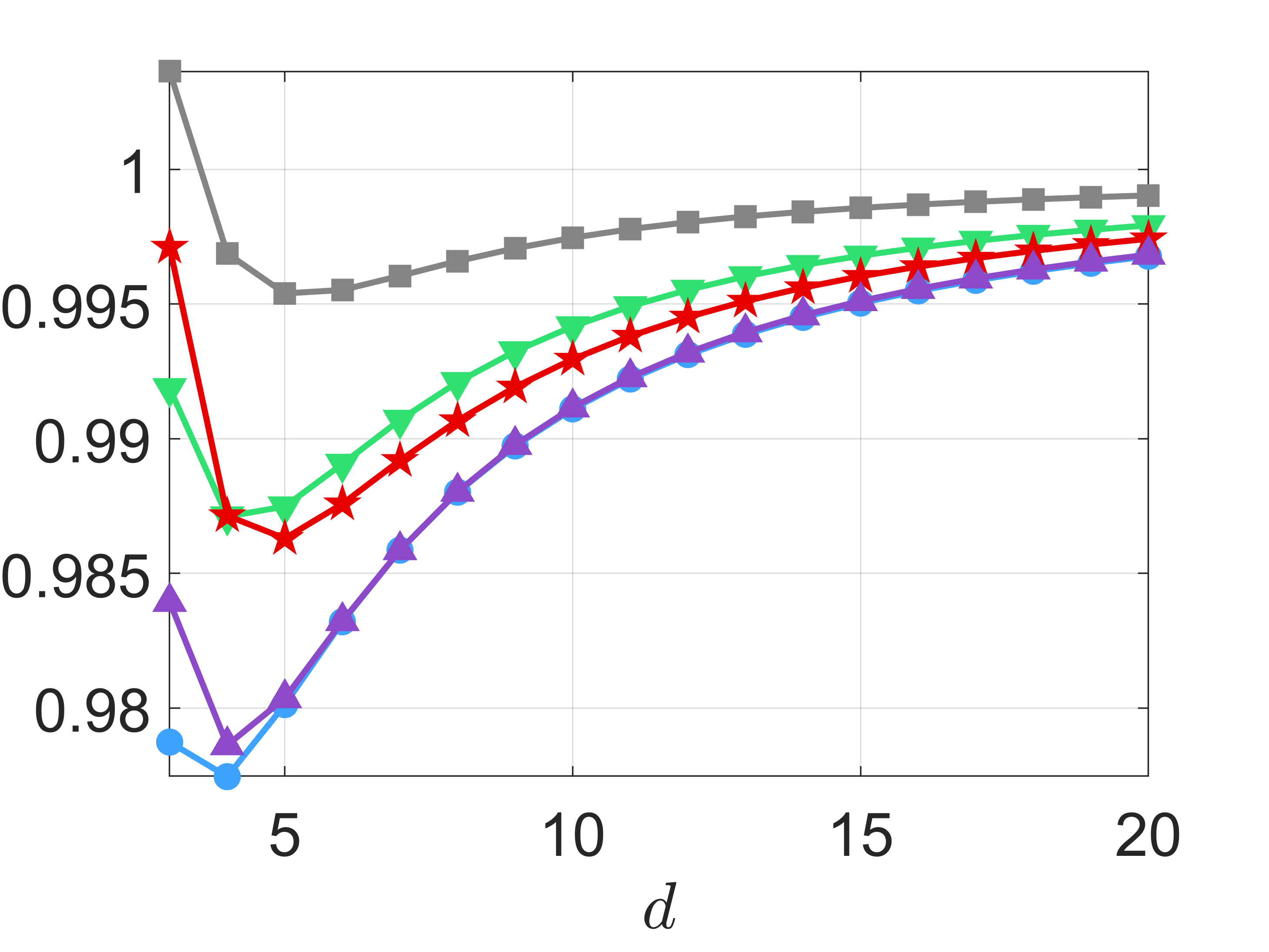}
\end{center}
\caption{\small  Left: $\ms^*(d,n,s)$ ({\color{Dandelion} $\blacklozenge$}),
$a^*(d,n,s)$ ({\color{Cerulean} $\bullet$}),
$\widetilde{a^*}(d,n,s)$ ({\color{green} $\blacktriangledown$}),
$M_{\psi,1}\widetilde{a^*}(d,n,s)$ ({\color{gray} {\tiny $\blacksquare$}}),
$\widehat{a^*_\ml}=\sqrt{3}/2$ ({\color{Orchid} $\blacktriangle$}) and
$\widehat{a^*}(d,n,s)$ ({\color{red} $\bigstar$}) as functions of $d$.
Right: efficiencies $D_{\mu,s}^{1/s}(\PP_a^{[n]})/D_{\mu,s}^{1/s}(\mu_{\ms^*}^{[n]})$ as functions of $d$ for the choices of $a$ indicated on the left panel ($n=2^d$, $s=4$).}
\label{F:normal_n2d_d3-20_s4}
\end{figure}

\begin{remark}
From Corollary~\ref{Coro:astar-asymptotic} and Proposition~\ref{Prop:extreme-value-general-mu},all quantities  $a^*(d,n,s)$, $\widetilde{a^*}(d,n,s)$,
$M_{\psi,1}\widetilde{a^*}(d,n,s)$ and
$\widehat{a^*}(d,n,s)$  tend to
$\widehat{a^*_\ml}=\sqrt{3}/2$ when $d\to\infty$. Our numerical results suggest that, in the considered asymptotic regime, $\ms^*(d,n,s)$ also converges to the limiting value $\widehat{a^*_\ml} = \sqrt{3}/2$. However, a rigorous theoretical analysis is still required to confirm this observation. The same is true for the asymptotic behaviour of $b^*(d,n,s)$ in the optimised quantiser $\PP_{0,b^*}^{[n]}$ when $\mu$ is uniform in $\SB_d(1)$. From \eqref{eq:quantile2} in Proposition~\ref{th:dist}, we have to analyse the behaviour of $d(\ub,\Rb_n)=\min_{i=1,\ldots,n} \left[(\|\ub\|-R_i)^2 +4\,\|\ub\|\,R_i \, \zeta_i\right]$ when $n\to\infty$, where the random variables $R_i$ are $\zeta_i$ are i.i.d.\ and mutually independent, with $\zeta_i \eqd \beta_{\delta,\delta}$. The assumption that $R=\|X\|$ with the probability measure $\PP$ of $X$ satisfying a norm-concentration property like \eqref{norm-concentration} does not directly imply an extreme-value property for $d(\ub,\Rb_n)$. The cases where $\PP = \PP_a$ (uniform distribution on $\SB_d(a)$) and $\PP = \mu_\ms$ (multivariate normal distribution) already pose significant challenges.
\fin
\end{remark}

\section{Conclusions}~\label{S:conclusion}

We have demonstrated that, for a spherically symmetric measure $\mu$ in $\RR^d$, unless the sample size $n$ is extremely large, a random quantiser uniformly distributed on a sphere of suitable radius can significantly outperform a random quantiser whose distribution follows the asymptotic limit predicted by Zador’s theorem. The optimal radius can be determined numerically by minimising a triple integral, which can be evaluated with arbitrary precision. Additionally, an approximation of the optimal radius is available, derived from extreme-value theory. Since the variability of the distortion for such random quantisers rapidly diminishes as $n$ increases, this approach is highly practical for constructing high-performance quantisers.

While the restriction to spherically symmetric distributions may appear limiting, it provides a robust foundation for broader applications. For instance, quantising the uniform measure on the $d$-dimensional hypercube $[-1,1]^d$ presents a compelling and challenging problem, particularly in space-filling design for computer experiments (see, e.g., \cite{PronMul2012}).  While greedy quantisation offers an attractive approach for the incremental construction of designs in low dimensions \cite{LuschgyP2015,NogalesPR2021}, it becomes computationally infeasible for large $d$. In such cases, recent results from \cite{noonanhigh2024} allow the uniform measure to be approximated by a spherically symmetric distribution. Similarly, one can consider random quantisers distributed according to a product measure, which can itself be approximated by a spherically symmetric distribution. Preliminary findings then suggest that quantisers distributed on the vertices of a smaller hypercube exhibit promising performance, paving the way for further advances in this direction.

\bibliographystyle{plain}
\bibliography{high-d} 

\end{document}